\documentclass[10pt]{article}
\usepackage[utf8]{inputenc}
\usepackage[T1]{fontenc}
\usepackage{amsmath, amssymb, amsthm}
\usepackage{geometry}
\usepackage{enumitem}
\usepackage{mathrsfs}
\usepackage{fontawesome5}
\usepackage{xcolor}
\usepackage{esint}
\usepackage{verbatim}
\usepackage[colorlinks=true, linkcolor=blue, citecolor=blue, urlcolor=blue]{hyperref}
\usepackage{shuffle}
\usepackage{enumitem}

\numberwithin{equation}{section} 

\newtheorem{theorem}{Theorem}[section]
\newtheorem{proposition}[theorem]{Proposition}
\newtheorem{lemma}[theorem]{Lemma}
\newtheorem{corollary}[theorem]{Corollary}

\theoremstyle{definition}
\newtheorem{definition}[theorem]{Definition}
\newtheorem{remark}[theorem]{Remark}
\newtheorem{example}[theorem]{Example}
\newtheorem*{remark*}{Remark}

\newcommand{\R}{\mathbb{R}}
\newcommand{\N}{\mathbb{N}}
\newcommand{\infin}{\infty}
\newcommand{\rarr}{\rightarrow}
\newcommand{\vertiii}[1]{{\vert\kern-0.25ex\vert\kern-0.25ex\vert #1 
    \vert\kern-0.25ex\vert\kern-0.25ex\vert}}
\renewcommand{\bold}[1]{\textup{\textbf{#1}}}

\usepackage[T1]{fontenc}
\usepackage{lmodern}

\title{A Direct Approach to Rough Paths in Time-Weighted Besov Spaces}
\author{Atiksh Gupta}
\date{23 September 2026}

\begin{document}

\maketitle

\begin{abstract}
    We extend the Besov rough path framework of Friz and Seeger to paths lying in Besov spaces that are Muckenhoupt-weighted in time. We establish a weighted Besov sewing lemma, construct Young and rough integrals, and prove local Lipschitz continuity of the Itô-Lyons map in weighted Besov norms. The resulting estimates deliver ``location-sensitive'' control of errors in RDE solution increments. We illustrate the framework with a diffusion whose coefficient is singular at a specified time.
\end{abstract}

\tableofcontents

\section{Introduction}
Rough path theory studies differential equations driven by signals whose irregularity prevents the direct use of classical integration. Its basic construction augments the driving signal with iterated integrals, providing the information needed to define the equation and control its solutions. A central feature of the theory is stability: when the enhanced driving signal is perturbed, the solution changes continuously in the ``rough path topology''.

Friz and Seeger \cite{FrizSeeger2022} developed such a theory directly in the Besov setting. Their construction combines a Besov sewing lemma with embeddings and controlled rough paths, leading to integration, differential equations, and stability estimates in Besov norms. Friz, Kern, and Zorin-Kranich \cite{FrizKernZorinKranich2025} subsequently established another route. They characterise Besov regularity by taking Besov norms of the variation accumulated on each subinterval, and use suitable group structures to extend this principle to rough path differences and controlled remainders. Precisely localised variation estimates can then be transferred to the Besov setting. These approaches provide different ways of preserving the regularity of the data throughout the analysis.

The present paper extends the direct approach to Besov spaces with a Muckenhoupt weight in time. The motivation is that the size and distribution of a signal's fluctuations may depend strongly on where an increment occurs. For example, a driving process may have much larger fluctuations near a specified time than elsewhere. A time weight allows the norm to account for this behaviour explicitly. Regions where the weight is small contribute less to the averaged increment norm, while regions where it is large contribute more. The aim is to carry this choice of measurement through integration and the solution map.

Our spaces weight an increment by the average of the weight over its interval. More precisely, for an increment from \(s\) to \(s+h\), we use
\[
w_h(s) = \frac{1}{h}\int_s^{s+h} w(u)\,du.
\]
This choice associates the weight with the whole increment interval. It is also compatible with subdivision: the average over an interval is the corresponding weighted combination of the averages over its children. The Muckenhoupt condition provides quantitative comparisons between these averages, including comparisons between neighbouring intervals. Together, these properties allow us to control the subdivisions and changes of basepoint that occur in the direct construction. When \(w\equiv 1\), the definitions recover the unweighted Besov-type spaces of \cite{FrizSeeger2022}.

We establish weighted Young and rough integration, together with existence, uniqueness, and local Lipschitz continuity of the solution map for the resulting differential equations. The estimates control the integrals, controlled remainders, and differences of solutions in the original weighted Besov norms. In particular, on bounded sets of admissible data, a small change in the enhanced driver produces a proportionally small change in the solution in the corresponding weighted topology. We develop the level-two theory in detail and then extend the controlled-path construction to arbitrary finite levels in the weakly geometric setting.

The main analytic tool is a weighted Besov sewing lemma. Sewing starts with a proposed increment that is approximately additive and constructs a genuine additive increment by showing that the errors introduced by successive subdivisions are summable. Our interval-weight estimates provide the bounds needed to carry out this procedure in the weighted spaces. We treat both the regime in which summability follows from a strict smoothness inequality and critical cases in which it comes from summability across scales. Pointwise embeddings and interpolation then supply the additional control needed for the integration and fixed-point arguments.

Deweighting is a useful part of this construction. Weighted Besov control implies unweighted Besov control at a reduced time-integrability exponent, allowing us to use existing unweighted results. Applied on individual intervals, deweighting also gives pointwise bounds involving the weighted mass of the interval. It does not, however, recover the original weighted norm from the weaker unweighted information. The weighted sewing and stability estimates supply this further control, retaining the chosen time weight in the bounds for integrals and solutions.

We illustrate the framework with a diffusion whose deterministic coefficient behaves like \(t^{-\gamma}\) near the initial time, where \(0<\gamma<1/2\). Its fluctuations are concentrated near zero, and this concentration can reduce its unweighted Besov regularity. A suitable power weight compensates for the coefficient's growth in the averaged increment estimates. We prove almost-sure weighted Besov regularity at the Brownian smoothness exponent, together with corresponding results for the iterated integrals. Applying the deterministic solution theory then gives bounds for solution-increment differences that depend explicitly on both the length and the location of the interval. These estimates distinguish increments near the singular time from increments farther away.

\subsection{Organisation}

In \S 2 we define our weighted Besov spaces and establish the interval-weight estimates, deweighting, and elementary properties used throughout the paper. \S 3 develops the embeddings and interpolation estimates that connect averaged regularity with pointwise control. \S 4 contains the weighted sewing theorems and integrability regain, and \S 5 applies them to Young integration and Young differential equations.

In \S 6 we develop weighted rough paths, the Lyons extension theorem, controlled rough integration, and rough differential equations, including stability and the Davie characterisation. The arbitrary-level construction is given in \S 6.5. \S 7 treats stochastic processes and the singular diffusion example. Appendix A compares alternative choices of increment weight and explains the role of interval averaging.

\subsection{Notation}

Throughout, $T$ denotes a finite positive real. For $a<b$, we write
$$
\bigtriangleup_n(a,b)
=\{(t_1,\ldots,t_n):a\leq t_1\leq\cdots\leq t_n\leq b\}.
$$
In particular,
$$
\bigtriangleup_2(0,T)=\{(s,t):0\leq s\leq t\leq T\},
\qquad
\bigtriangleup_3(0,T)=\{(s,u,t):0\leq s\leq u\leq t\leq T\}.
$$
For a vector-valued path $f$, we denote its increment by
$$
\delta f_{s,t}=f_t-f_s,
$$
and, for a two-parameter map $A$, we denote its additive defect by
$$
\delta A_{s,u,t}=A_{s,t}-A_{s,u}-A_{u,t}.
$$
For metric-space-valued paths, increments are measured by $d(f_s,f_t)$. Additive two-parameter maps have zero diagonal values; multiplicative tensor-valued maps have the tensor unit on the diagonal.

Given a nondegenerate interval $I\subset[0,T]$, we denote its length by $|I|$ and set
$$
w(I)=\int_Iw(u)\hspace{.2em}du,
\qquad
\langle w\rangle_I=\frac{w(I)}{|I|}.
$$
The Muckenhoupt classes and their characteristics $[w]_{A_r}$ are defined in \S 2. We use
$$
B_h(x)=[x-h,x+h]\cap[0,T].
$$
All statements involving restrictions to an interval use the same weight restricted to that interval.

We write $a\land b=\min\{a,b\}$ and $a\lor b=\max\{a,b\}$. We adopt the conventions
$$
\frac1\infin=0,
\qquad
\frac\infin c=\infin\quad(0<c<\infin),
\qquad
w(I)^{-1/\infin}=1.
$$
At an infinite integrability or scale index, the corresponding norm is an essential supremum or supremum, as specified in the definitions. In particular, the increment weight has no effect when $p=\infin$.

The notation $B^\alpha_{p,q}$ is used for path spaces, $\mathbb B^\alpha_{p,q}$ for two-parameter spaces, and $\overline{\mathbb B}^\alpha_{p,q}$ for three-parameter spaces. A weight is displayed as an argument, and a restriction to $J$ is indicated by a semicolon, as in $B^\alpha_{p,q}(w;J)$. We use square brackets for path seminorms. The precise definitions, including general moduli and indices below one, are given in \S 2. Once a continuous representative has been obtained, subsequent pointwise expressions use that representative.

For a partition
$$
\pi=\{s=\tau_0<\tau_1<\cdots<\tau_N=t\}
$$
of $[s,t]$, we set
$$
\|\pi\|=\max_{1\leq i\leq N}(\tau_i-\tau_{i-1}).
$$
When $\pi=\{0=u_0<\cdots<u_N=1\}$ is a partition of $[0,1]$, its rescaling to $[s,t]$ has points $s+u_i(t-s)$. For a two-parameter map $\Xi$, write
$$
(I_\pi\Xi)_{s,t}
=\sum_{i=1}^N
\Xi_{s+u_{i-1}(t-s),\,s+u_i(t-s)},
\qquad
R_\pi\Xi=I_\pi\Xi-\Xi.
$$
The dyadic partition of $[0,1]$ at level $n$ is
$$
\pi_n=\{k2^{-n}:0\leq k\leq2^n\}.
$$

For normed vector spaces $V,E$, we denote by $\mathcal L(V,E)$ the space of bounded linear maps with its operator norm. Finite-dimensional tensor spaces are equipped with fixed compatible norms. We write $A\lesssim B$ if $A\leq CB$ for a constant independent of the variables being estimated, and $A\asymp B$ if both inequalities hold. Subscripts indicate relevant dependencies. Constants may change from line to line; dependence on weight normalisation or on random data is stated when it occurs.

\begin{definition}[Bounded Hölder Classes]
\label{def: Bounded Holder Classes}
    Let $V,E$ be finite-dimensional normed vector spaces and let $k\geq0$ be an integer. We write $C_b^k(V;E)$ for the space of $k$ times continuously differentiable maps $F:V\to E$ for which
    $$
    \|F\|_{C_b^k}
    =\sum_{j=0}^k\|D^jF\|_\infin<\infin,
    $$
    with $D^0F=F$. For $0<\eta\leq1$, we write $C_b^{k,\eta}(V;E)$ for the subspace satisfying
    $$
    [D^kF]_\eta
    =\sup_{x\ne y}
    \frac{\|D^kF(x)-D^kF(y)\|}{|x-y|^\eta}
    <\infin,
    $$
    equipped with
    $$
    \|F\|_{C_b^{k,\eta}}
    =\|F\|_{C_b^k}+[D^kF]_\eta.
    $$
    In particular, $C_b^{2,1}$ consists of bounded twice continuously differentiable maps with bounded first and second derivatives and globally Lipschitz second derivative.
\end{definition}

\section{Weighted Besov Path Spaces}
In this section we define the weighted Besov spaces that our framework is ``native'' to. We begin by recalling some standard definitions and properties of Muckenhoupt weights. We then introduce weighted analogues of the Besov-type spaces of \cite[\S 2]{FrizSeeger2022}. Immediately after, we present a \textit{deweighting embedding}, which shows that all paths in a weighted Besov space also lie in an unweighted Besov space of identical smoothness but reduced integrability. We conclude with certain properties of weighted Besov paths which will be used in later sections.

\subsection{Muckenhoupt Weights}
\begin{definition}[Weights]
\label{def: Muckenhoupt Weights}
    A weight $w: [0, T] \rarr [0, \infin]$ is a locally integrable function taking values in $(0, \infin)$ Lebesgue-a.e. For $1 < r < \infin$ we define the $A_r$ Muckenhoupt constant $[w]_{A_r}$ with
    $$
    [w]_{A_r} = \sup_I \langle w \rangle_I \langle w^{1-r'} \rangle_I^{r - 1}
    $$
    where $I$ ranges over all subintervals of $[0, T]$ and $r'$ is the conjugate exponent to $r$, meaning $\frac{1}{r} + \frac{1}{r'} = 1$. We say $w \in A_r$ if $[w]_{A_r} < \infin$. In the case $r = 1$ we define
    $$
    [w]_{A_1} = \sup_I \langle w \rangle_I \| w^{-1}\|_{L^\infin(I)}
    $$
    and we say $w \in A_1$ if $[w]_{A_1} < \infin$. For $r > 1$, we denote by $\sigma$ the dual weight $w^{1-r'}$ to $w \in A_r$.
\end{definition}

We'll use the following standard results constantly in the sequel.

\begin{lemma}
\label{lem: Standard Muckenhoupt Weights Properties}
    Let $1 \leq r < \infin$ and let $w \in A_r$. Then
    \begin{enumerate}[label=(\roman*)]
        \item (Small-Set Condition) If $E \subset I$ is measurable then
        $$
        \left(
        \frac{|E|}{|I|}
        \right)^r \leq [w]_{A_r} \frac{w(E)}{w(I)}.
        $$

        \item ($w$ induces a Doubling Measure) For any $\lambda \geq 1$ and interval $I$,
        $$
        w(\lambda I) \leq (2\lambda)^r [w]_{A_r} w(I),
        $$
        where $\lambda I$ is the concentric dilation of $I$ (intersected with $[0, T]$ if necessary).

        \item (Openness and Reverse Hölder Inequality) If $r > 1$ then there exists an $\epsilon_0 \in (0, r - 1)$ such that $w \in A_{r - \epsilon}$ for all $0 < \epsilon \leq \epsilon_0$. Moreover, there exist $0 < C, \gamma < \infin$ with
        $$
        \langle w^{1 + \gamma} \rangle_{I}^{1/(1 + \gamma)} \leq C \langle w \rangle_I
        $$ for every interval $I \subset [0, T]$. (All constants here depend only on $r, [w]_{A_r}$).

        \item (Duality) If $r > 1$ then $\sigma \in A_{r'}$ and
        $$
        [
        \sigma
        ]_{A_{r'}} = [w]_{A_r}^{\frac{1}{r-1}}
        $$

        \item (Weight Classes are Increasing) Let $1 \leq r < s < \infin$ and $w \in A_r$. Then $w \in A_s$ with $[w]_{A_s} \leq [w]_{A_r}$.
    \end{enumerate}
\end{lemma}
\begin{proof}
    Proofs of these properties may be found in \cite[\S 7]{Grafakos2014Classical}.
\end{proof}

Our weighted Besov setting will rely on ``interval-based averages'': given some $w \in A_r$, interval size $0 < h \leq T$, and basepoint $s \in [0, T - h]$, we set $w_h(s) = \frac{1}{h} \int^{s + h}_{s} w(t) \hspace{.2em} dt$. These objects obey the following comparison bounds.

\begin{lemma}
\label{lem: Shift and Scale Bounds for Muckenhoupt Weights}
    Let $1 \leq r < \infin$ and $w \in A_r$.
    \begin{enumerate}[label=(\roman*)]
        \item If $J \subset I$ then
        $$
        \langle w \rangle_I \leq [w]_{A_r} \left( \frac{|I|}{|J|} \right)^{r - 1} \langle w \rangle_J, \qquad \langle w \rangle_J \leq \frac{|I|}{|J|} \langle w \rangle_I.
        $$
        In particular, for every $\lambda \geq 1$,
        $$
        w_h(s) \leq [w]_{A_r} \lambda^{r - 1} w_{h/\lambda}(s), \qquad w_{h/\lambda}(s) \leq \lambda w_h(s).
        $$

        \item Let $\lambda \geq 0$ and $|b| \leq \lambda h$. Whenever both $[s, s + h]$ and $[s - b, s - b + h]$ are contained in $[0, T]$, we have
        $$
        w_h(s - b) \leq [w]_{A_r} (1 + \lambda)^r w_h(s).
        $$

        \item We have
        $$
        \sup_{0 < h \leq T} \int^{T - h}_{0} w_h(s) \hspace{.2em} ds \leq w([0, T]). 
        $$

        \item Let $I_1, ..., I_N$ partition an interval $I$ and set $\lambda_i = |I_i|/|I|$. Then
        $$
        \langle w \rangle_I = \sum_{i = 1}^{N} \lambda_i \langle w \rangle_{I_i}.
        $$
        In particular, for $0 < \theta < 1$,
        $$
        w_h(s) = \theta w_{\theta h}(s) + (1 - \theta) w_{(1 - \theta)h}(s + \theta h).
        $$

        \item Given the partition of (iv), let $0 < p < \infin$ and set $a = \max\{p, r\}$. Then for every normed vector space $V$ and $z_1, ..., z_N \in V$,
        $$
        \langle w \rangle_I \left\|\sum_{i = 1}^{N} z_i \right\|_V^p \leq [w]_{A_r} \sum_{i = 1}^{N} \lambda_i^{1 - a} \langle w \rangle_{I_i} \|z_i\|_V^p.
        $$
        For a partition into $N$ equal intervals, this becomes
        $$
        \langle w \rangle_I \left\|\sum_{i = 1}^{N} z_i \right\|_V^p \leq [w]_{A_r} N^{a - 1} \sum_{i = 1}^{N} \langle w \rangle_{I_i} \|z_i\|_V^p.
        $$
    \end{enumerate}
\end{lemma}
\begin{proof}
    For (i), the small-set condition, Lemma~\ref{lem: Standard Muckenhoupt Weights Properties}(i), applied to $J \subset I$ (along with some rearranging), gives
    $$
    \langle w \rangle_I = \frac{w(I)}{|I|} \leq [w]_{A_r} \left( \frac{|I|}{|J|} \right)^{r - 1} \langle w \rangle_{J}.
    $$
    Conversely, $w(J) \leq w(I)$, so
    $$
    \langle w \rangle_J \leq \frac{|I|}{|J|} \langle w \rangle_I.
    $$
    Taking $I = [s, s + h], J = [s, s + h/\lambda]$ gives (i).

    For (ii), set $I = [s, s + h]$ and $J = [s - b, s - b + h]$, and let $H$ be the smallest interval containing both $I, J$. Then $|H| \leq h + |b| \leq (1 + \lambda)h$. Applying the second inequality of (i) to $J \subset H$, followed by the first inequality to $I \subset H$, gives
    $$
    w_h(s - b) = \langle w \rangle_J \leq \frac{|H|}{h} \langle w \rangle_H \leq [w]_{A_r} (1 + \lambda)^{r} w_h(s).
    $$

    For (iii), for fixed $0 < h \leq T$, Tonelli's theorem gives
    \begin{align*}
        \int^{T - h}_{0} w_h(s) \hspace{.2em} ds &= \frac{1}{h} \int^{T - h}_{0} \int^{s + h}_{s} w(t) \hspace{.2em} dt \hspace{.2em} ds
        \\ &=
        \frac{1}{h} \int^{T}_{0} w(t) |\{s \in [0, T - h] : s \leq t \leq s + h\}| \hspace{.2em} dt. 
    \end{align*}
    Taking the supremum over $h$ gives (iii).

    For (iv), by additivity of the integral,
    $$
    \langle w \rangle_I = \frac{1}{|I|} \sum_{i = 1}^{N} \int_{I_i} w(t) \hspace{.2em} dt = \sum_{i = 1}^{N} \frac{|I_i|}{|I|} \langle w \rangle_{I_i} = \sum_{i = 1}^{N} \lambda_i \langle w \rangle_{I_i}.
    $$
    Applying this to the partition of $[s, s + h]$ at $s + \theta h$ gives
    $$
    w_h(s) = \theta w_{\theta h}(s) + (1 - \theta) w_{(1 - \theta)h}(s + \theta h).
    $$

    For (v), we first prove the inequality
    \begin{equation}
        \langle w \rangle_I \langle f \rangle_I^a \leq [w]_{A_r} \langle f^a w \rangle_I. \label{eq: Subdivision Bound Intermediate 1}
    \end{equation}
    Suppose first that $a > 1$, then since $a \geq r$, Lemma~\ref{lem: Standard Muckenhoupt Weights Properties}(v) gives $w \in A_a$ with $[w]_{A_a} \leq [w]_{A_r}$. By Hölder's inequality, then,
    $$
    \langle f \rangle_I \leq \langle f^a w \rangle_I^{1/a} \langle w^{-1/(a-1)} \rangle_I^{(a-1)/a},
    $$ whence we get \eqref{eq: Subdivision Bound Intermediate 1} by raising to $a$th powers and multiplying by $\langle w \rangle_I$. If $a = 1$, then $r = 1$ too, so $w \in A_1$ means
    $$
    \langle w \rangle_I \leq [w]_{A_r} w(t)
    $$ for a.e $t \in I$. Multiplying by $f(t)$ and averaging over $I$ gives us \eqref{eq: Subdivision Bound Intermediate 1} here as well.
    Now, write $x_i = \|z_i\|_V$ and define
    $$
    f(t) = \frac{x_i^{p/a}}{\lambda_i}
    $$ for $t \in I_i$. Then a simple calculation shows
    $$
    \langle f \rangle_I = \sum_{i = 1}^{N} x_i^{p/a}, \qquad \langle f^a w \rangle_I = \sum_{i = 1}^{N} \lambda_i^{1 - a} \langle w \rangle_{I_i} x^p.
    $$
    Using $p/a$-subadditivity gives
    $$
    \left( \sum_{i = 1}^{N} x_i \right)^{p/a} \leq \sum_{i = 1}^{N} x_i^{p/a},
    $$
    after which the triangle inequality and \eqref{eq: Subdivision Bound Intermediate 1} give
    $$
    \langle w \rangle_I \left\| \sum_{i = 1}^{N} z_i \right\|_V^p \leq \langle w \rangle_I \left( \sum_{i = 1}^{N} x_i \right)^{p} \leq \langle w \rangle_I \left( \sum_{i = 1}^{N} x_i^{p/a} \right)^a \leq [w]_{A_r} \sum_{i = 1}^{N} \lambda_i^{1 - a} \langle w \rangle_{I_i} \|z_i\|_V^p.
    $$
    For a partition into $N$ equal intervals, we have $\lambda_i = 1/N$ and so $\lambda_i^{1 - a} = N^{a - 1}$, giving us the final assertion.
\end{proof}

Part (v) estimates all child increments jointly and is the key ingredient in obtaining the sewing threshold $1 \lor \frac{r}{p}$. Comparing the parent and child weights separately before summing the increments instead gives the potentially larger threshold $(1 \lor \frac{1}{p}) + \frac{r-1}{p}$.

\begin{remark}
    There exist multiple natural alternatives to the $w_h$ we've defined. We discuss some of these candidates in Appendix A, namely weights whose actions depend on endpoints of an interval, and weights that produce unnormalised masses of intervals. Appendix A compares how suited these options are to the sewing arguments we need later. The present choice delivers the parent-child comparisons we need for sewing, as well as (v) of the preceding lemma.
\end{remark}

\subsection{Weighted Spaces of Besov Type}
The following definitions precisely mirror those of \cite[\S 2]{FrizSeeger2022}.

\begin{definition}
    Let $(E, d)$ be a complete metric space and $f: [0, T] \rarr E$ a measurable path. For $0 < p < \infin$ we define the modulus of smoothness
    $$
    \omega_p^w(f, \tau) = \sup_{0 < h \leq \tau} \left(
    \int^{T - h}_0 d(f_s, f_{s + h})^p w_h(s) \hspace{.2em} ds
    \right)^{1/p}.
    $$
    For $p = \infin$ we adopt the usual essential-supremum convention and set
    $$
    \omega^w_\infin(f, \tau) = \sup_{0 < h \leq \tau} \underset{0 \leq s \leq T - h}{\textup{esssup}} d(f_s, f_{s + h}).
    $$
    In the case $p = \infin$, then, the action of the weight is irrelevant. 
    For $\alpha \in (0, 1)$ and $0 < q \leq \infin$ we set
    $$
    [f]_{B^\alpha_{p, q}(w)} = \left\| \tau^{-\alpha} \omega^w_p(f, \tau) \right\|_{L^q((0, T], d\tau/\tau)}.
    $$
    Explicitly, when $q < \infin$ and $q = \infin$ we have
    $$
    [f]_{B^\alpha_{p, q}(w)} = \left[ \int^T_0 \left( \frac{\omega^w_p(f, \tau)}{\tau^{\alpha}} \right)^q \frac{d\tau}{\tau} \right]^{1/q}, \qquad [f]_{B^\alpha_{p, q}(w)} = \sup_{0 < \tau \leq T} \tau^{-\alpha} \omega^w_p(f, \tau).
    $$
    We write $f \in B^\alpha_{p, q}([0, T]; E; w)$ if $d(f, x_0) \in L^p(w)$ for some (and thus all) $x_0 \in E$ and if $[f]_{B^\alpha_{p, q}(w)} < \infin$.
\end{definition}

\begin{definition}
    If $\widetilde{\omega}: (0, T] \rarr (0, \infin)$ is a nondecreasing function we set
    $$
    [f]_{B^{\tilde{\omega}}_{p, q}(w)} = \left\| \frac{\omega^w_p(f, \tau)}{\widetilde{\omega}(\tau)} \right\|_{L^q((0, T], d \tau / \tau)}.
    $$
    This generalised quantity comes in handy when dealing with \textit{critical sewing} in $\S 4$. In particular, if $\widetilde{\omega}$ is taken to contain some factor that is logarithmic in $\tau$, then functions $f$ whose moduli $\omega^w_p(f, \tau)$ shrink too slowly as $\tau \rarr 0^+$ can also be considered. We write $f \in B^{\tilde{\omega}}_{p, q}(w)$ if $d(f, x_0) \in L^p(w)$ for some (and thus all) $x_0 \in E$ and if $[f]_{B^{\tilde{\omega}}_{p, q}(w)} < \infin$.
\end{definition}
    
\begin{definition}
    For measurable two and three-parameter maps $A: \bigtriangleup_2(0, T) \rarr V$ and $B: \bigtriangleup_3(0, T) \rarr V$, where $V$ is a finite-dimensional normed vector space, we define the moduli
    \begin{align*}
    &\Omega^w_p(A, \tau) = \sup_{0 < h \leq \tau} \left(
    \int^{T - h}_0 |A_{s, s + h}|^p w_h(s) \hspace{.2em} ds
    \right)^{1/p},
    \\
    &\overline{\Omega}^w_p(B, \tau) = \sup_{0 \leq \theta \leq 1} \sup_{0 < h \leq \tau} \left(
    \int^{T - h}_0 |B_{s, s + \theta h, s + h}|^p w_h(s) \hspace{.2em} ds
    \right)^{1/p}
    \end{align*}
    where in the case $p = \infin$ we use the standard essential-supremum convention.
    We set
    $$
    \| A \|_{\mathbb{B}^\alpha_{p, q}(w)} = \| \tau^{-\alpha} \Omega_p^w(A, \tau) \|_{L^q((0, T], d\tau/\tau)}, \qquad \|B\|_{\overline{\mathbb{B}}_{p, q}^\alpha(w)} = \|  \tau^{-\alpha} \overline{\Omega}_p^w(B, \tau) \|_{L^q((0, T], d\tau/\tau)}
    $$
    and if $\widetilde{\omega}: (0, T] \rarr (0, \infin)$ is a nondecreasing function we set
    $$
    \| A \|_{\mathbb{B}^{\tilde{\omega}}_{p, q}(w)} = \| \widetilde{\omega}(\tau)^{-1} \Omega_p^w(A, \tau) \|_{L^q((0, T], d\tau/\tau)}, \qquad \|B\|_{\overline{\mathbb{B}}_{p, q}^{\tilde{\omega}}(w)} = \|  \widetilde{\omega}(\tau)^{-1} \overline{\Omega}_p^w(B, \tau) \|_{L^q((0, T], d\tau/\tau)}.
    $$
    We define these norms not just for $0 < \alpha < 1$, but for all $\alpha > 0$, since later sewing defects have smoothness exponents greater than $1$. In the following path space definitions, however, we restrict $0 < \alpha < 1$.
    We write $A \in \mathbb{B}^\alpha_{p, q}(w)$ if $\| A \|_{\mathbb{B}^\alpha_{p, q}(w)} < \infin$, and we write $B \in \overline{\mathbb{B}}^\alpha_{p, q}(w)$ if $\| B \|_{\overline{\mathbb{B}}^\alpha_{p, q}(w)} < \infin$.
    Lastly, we define
    $$
    \mathbb{B}^{\tilde{\omega}}_{p, \infin; \circ}(w) = \left\{ A \in \mathbb{B}^{\tilde{\omega}}_{p, \infin}(w) : \lim_{\tau \rarr 0^+} \frac{\Omega^w_p(A, \tau)}{\widetilde{\omega}(\tau)} = 0 \right\}.
    $$
    This subspace of $\mathbb{B}^{\tilde{\omega}}_{p, \infin}(w)$ appears when we consider critical sewing.
\end{definition}

\begin{remark}[Quasi-Banach Structure]
    If the target spaces $E, V$ are Banach spaces then all the preceding path spaces are complete. In particular, when $p, q \geq 1$ then, after equipping them with the natural norm, these spaces are themselves Banach spaces, and if either $p, q < 1$, then the natural ``norm'' becomes a quasi-norm, and the spaces become quasi-Banach spaces. Either way, we use completeness freely throughout the rest of the framework.
\end{remark}

\begin{remark}[Parent Weight Used in Three-Point Norm]
    In the three-point norm $\overline{\Omega}^w_p$ we integrate against the parent-weight $w_h(s)$, instead of children weights $w_{\theta h}(s)$ or $w_{(1 - \theta)h}(s + \theta h)$ since the additive defect $\delta A_{s, u, t}$ relevant to sewing measures the ``error when splitting'' across the parent interval $[s, t]$.
\end{remark}

\subsection{Deweighting}
Friz and Seeger prove that Besov spaces embed into Hölder spaces at a small cost in the ``smoothness'' exponent. We prove an analogous \textit{deweighting} result: weighted Besov spaces embed into Besov spaces at a cost in the ``integrability'' exopnent.

\begin{theorem}[Deweighting]
\label{thm: Deweighting}
    Let $1 \leq r < \infin$, $w \in A_r$, and $0 < p < \infin$. Set $u = p/r$ and let $J = [a, b] \subset [0, T]$. Define
    $$
    c_J = \begin{cases}
        \sigma(J)^{(r-1)/p}, & r > 1, \\
        \left( \underset{J}{\textup{esssinf }} w \right)^{-1/p}, & r = 1.
    \end{cases}
    $$
    Then
    $$
    c_J \leq [w]_{A_r}^{1/p} \frac{|J|^{r/p}}{w(J)^{1/p}}.
    $$
    Let $(E, d)$ be a complete metric space and $V$ a finite-dimensional normed vector space. For every measurable path $f: J \rarr E$, two-parameter map $A: \bigtriangleup_2(J) \rarr V$, and three-parameter map $B: \bigtriangleup_3(J) \rarr V$, and every $0 < \tau \leq |J|$, we have
    \begin{align*}
        &\omega_u(f, \tau; J) \leq c_J \omega^w_p(f, \tau; J), \\
        &\Omega_u(A, \tau; J) \leq c_J \Omega^w_p(A, \tau; J), \\
        &\overline{\Omega}_u(B, \tau; J) \leq c_J \overline{\Omega}^w_p(B, \tau; J).
    \end{align*}
    Consequently, for every $0 < q \leq \infin$ and nondecreasing modulus $\widetilde{\omega}: (0, |J|] \rarr (0, \infin)$,
    \begin{align*}
        &[f]_{B^{\tilde{\omega}}_{u, q}(J)} \leq c_J [f]_{B^{\tilde{\omega}}_{p, q}(w; J)}, \\
        &\|A\|_{\mathbb{B}^{\tilde{\omega}}_{u, q}(J)} \leq c_J \|A\|_{\mathbb{B}^{\tilde{\omega}}_{p, q}(w; J)}, \\
        &\|B\|_{\overline{\mathbb{B}}^{\tilde{\omega}}_{u, q}(J)} \leq c_J \|B\|_{\overline{\mathbb{B}}^{\tilde{\omega}}_{p, q}(w; J)}.
    \end{align*}
    The same $c_J$ gives
    $$
    \|g\|_{L^u(J)} \leq c_J \|g\|_{L^p(w; J)}
    $$
    for every measurable map $g$ on $J$. When $p = \infin$, all preceding inequalities hold with $u = \infin$ and $c_J = 1$. 
\end{theorem}
\begin{proof}
    This result may be proven with Hölder's inequality followed by Jensen and Fubini. First, suppose $r > 1$ and $0 < p < \infin$. The $A_r$ condition gives
    $$
    \frac{w(J)}{|J|} \left( \frac{\sigma(J)}{|J|} \right)^{r - 1} \leq [w]_{A_r},
    $$
    after which a simple rearrangement gives
    $$
    c_J \leq [w]_{A_r}^{1/p} \frac{|J|^{r/p}}{w(J)^{1/p}}.
    $$

    For $0 < h < |J|$, set $J_h = [a, b - h]$ and
    $$
    \sigma_h(s) = \frac{1}{h} \int^{s + h}_{s} \sigma(t) \hspace{.2em} dt
    $$ for $s \in J_h$. Since $x \mapsto x^{-1/(r-1)}$ is convex on $(0, \infin)$, Jensen's inequality gives
    $$
    w_h(s)^{-1/(r-1)} = \left( \frac{1}{h} \int^{s + h}_{s} w(t) \hspace{.2em} dt \right)^{-1/(r-1)} \leq \sigma_h(s).
    $$
    Tonelli's theorem then yields
    $$
    \int^{b - h}_{b} \sigma_h(s) \hspace{.2em} ds = \frac{1}{h} \int^{b}_{a} \sigma(t) |s \in [a, b - h] : s \leq t \leq s + h| \hspace{.2em} dt \leq \sigma(J).
    $$
    Hence
    \begin{equation}
        \int_{J_h} w_h(s)^{-1/(r - 1)} \hspace{.2em} ds \leq \sigma(J). \label{eq: Deweighting Intermediate 1}
    \end{equation}
    Now, let $g: J_h \rarr [0, \infin]$ be measurable. Then Hölder's inequality with $r, \frac{r}{r-1}$ gives
    \begin{align*}
        \int_{J_h} g(s)^u \hspace{.2em} ds &= \int_{J_h} (g(s)^p w_h(s))^{1/r} w_h(s)^{-1/r} \hspace{.2em} ds
        \\ &\leq
        \left( \int_{J_h} g(s)^p w_h(s) \hspace{.2em} ds \right)^{1/r} \left( \int_{J_h} w_h(s)^{-1/(r-1)} \hspace{.2em} ds \right)^{(r-1)/r}
    \end{align*}
    Using \eqref{eq: Deweighting Intermediate 1} and raising to the power $1/u$ gives the intermediate inequality.
    \begin{equation}
        \|g\|_{L^u(J_h)} \leq c_J \|g\|_{L^p(w_h; J_h)}. \label{eq: Deweighting Intermediate 2}
    \end{equation}

    For a measurable path $f: J \rarr E$, applying \eqref{eq: Deweighting Intermediate 2} to $g(s) = d(f_s, f_{s + h})$ and taking the supremum over all $0 < h \leq \tau$ gives
    $$
    \omega_u(f, \tau; J) \leq c_J \omega^w_p(f, \tau; J).
    $$
    Similarly, applying \eqref{eq: Deweighting Intermediate 2} to $g(s) = \| A_{s, s + h} \|_V$ and taking the supremum over all $0 < h \leq \tau$ gives
    $$
    \Omega_u(A, \tau; J) \leq c_J \Omega^w_p(A, \tau; J).
    $$
    Finally, for $\theta \in [0, 1]$, applying \eqref{eq: Deweighting Intermediate 2} to $g(s) = \| B_{s, s + \theta h, s + h} \|_V$ and taking suprema over $0 < h \leq \tau$ and $0 \leq \theta \leq 1$ gives
    $$
    \overline{\Omega}_u(B, \tau; J) \leq c_J \overline{\Omega}^w_p(B, \tau; J).
    $$
    Dividing the preceding three inequalities by $\widetilde{\omega}$, which is valid since this function is strictly greater than $0$, then taking the $L^q((0, |J|], d\tau/\tau)$ quasi-norm gives the desired norm inequalities.

    The full form of \eqref{eq: Deweighting Intermediate 2}, that is, with $J$ instead of $J_h$, follows simply from Hölder's inequality and the $w^{1/r} \cdot w^{-1/r}$ trick.

    Now, suppose $r = 1$ and set $m_J = \text{esssinf }_{t \in J} w(t)$. The inequality
    $$
    c_J = m_J^{-1/p} \leq [w]_{A_r}^{1/p} \frac{|J|^{1/p}}{w(J)^{1/p}}
    $$
    follows from the $A_1$ condition. Since $u = p$ here, we directly get \eqref{eq: Deweighting Intermediate 1} for any measurable $g$, whence we get the desired estimates by the same substitutions as above.

    Finally, when $p = \infin$ the weight is ignored entirely anyways, so the weighted-space inequalities are trivial. Since the null sets of $w(t) \hspace{.2em} dt$ and $dt$ are the same (as $w > 0$ a.e), we get $\|g\|_{L^\infin(w; J)} = \|g\|_{L^\infin(J)}$, so all inequalities hold in this case too.
\end{proof}

\begin{corollary}[Deweighted Embeddings]
    Let $1 \leq r < \infin$, $w \in A_r$, and $0 < p < \infin$, $0 < q \leq \infin$. Let $\tilde{\omega}: (0, T] \rarr (0, \infin)$ be a nondecreasing function. Then for every interval $J \subset [0, T]$,
    $$
    B^{\tilde{\omega}}_{p, q}(w; J) \hookrightarrow B^{\tilde{\omega}}_{p/r, q}(J)
    $$
    continuously, with constant at most $c_J$. Additionally, for each integer $k \geq 1$,
    \begin{align*}
        &\mathbb{B}^{\tilde{\omega}^k}_{p/k, q/k}(w; J) \hookrightarrow \mathbb{B}^{\tilde{\omega}^k}_{p/(kr), q/k}(J), \\
        &\overline{\mathbb{B}}^{\tilde{\omega}^k}_{p/k, q/k}(w; J) \hookrightarrow \overline{\mathbb{B}}^{\tilde{\omega}^k}_{p/(kr), q/k}(J).
    \end{align*}
    continuously.
\end{corollary}
\begin{proof}
    This follows immediately from Theorem~\ref{thm: Deweighting}.
\end{proof}

\begin{remark}
    Under the hypotheses of Theorem~\ref{thm: Deweighting}, the same deweighting holds if our choice of $w_h(s)$ is replaced by
    $$
    \rho_h(s) = w(s), \qquad \rho_h(s) = w(s + h), \qquad \rho_h(s) = \frac{w(s + h) + w(s)}{2}.
    $$
    This is addressed in Appendix A.
\end{remark}

The deweighting theorem allows us to apply the unweighted results of \cite{FrizSeeger2022} whenever the transferred indices satisfy their hypotheses. Applied on individual intervals, it also yields pointwise estimates that retain dependence on the local weighted mass. These conclusions do not, however, by themselves provide sewing and stability estimates in the original weighted Besov norms. Establishing such estimates, including control of the relevant remainder norms, is the purpose of the weighted theory developed in the sequel.

\subsection{Vacuity Criterion}
\begin{proposition}[Vacuity]
\label{prop: Vacuity Criterion}
    Let $1 \leq r < \infin$, $w \in A_r$, and $0 < p \leq \infin$, and define
    $$
    \nu(p, r) = 1 \lor \frac{r}{p},
    $$
    with $r/\infin = 0$. Suppose $f: [0, T] \rarr E$ is a measurable metric space--valued satisfying
    $$
    \lim_{\tau \rarr 0^+} \frac{\omega^w_p(f, \tau)}{\tau^{\nu(p, r)}} = 0
    $$
    Then $f$ has a version that is constant on $[0, T]$. In particular, this conclusion holds if
    $$
    \left\| \tau^{-\gamma} \omega^w_p(f, \tau) \right\|_{L^q((0, T], d\tau/\tau)} < \infin
    $$ and either
    $$
    \gamma > \nu(p, r) \hspace{.5em} \textup{ and } \hspace{.5em} 0 < q \leq \infin, \qquad \textup{ or } \qquad \gamma = \nu(p, r) \hspace{.5em} \textup{ and } \hspace{.5em} 0 < q < \infin.
    $$
\end{proposition}
\begin{proof}
    We deweight and use the vacuity result of \cite{FrizSeeger2022}.
    Suppose first that $p < \infin$ and set $u = p/r$. Theorem~\ref{thm: Deweighting} gives
    $$
    \omega_u(f, \tau) \leq c_{[0, T]} \omega^w_p(f, \tau).
    $$
    Since $1 \lor u^{-1} = \nu(p, r)$, our hypothesis gives
    $$
    \lim_{\tau \rarr 0^+} \frac{\omega_u(f, \tau)}{\tau^{1 \lor u^{-1}}} = 0.
    $$
    Then \cite[Lemma 2.1]{FrizSeeger2022} gives a constant version of $f$. When $p = \infin$ the weighted modulus is already unweighted, so that lemma applies directly.
    For the final assertions, suppose $q < \infin$. Monotonicity of $\omega^w_p(f, \cdot)$ gives, for $0 < \tau \leq T/2$,
    $$
    \left( \frac{\omega^w_p(f, \tau)}{\tau^\gamma} \right)^q \lesssim \int^{2\tau}_{\tau} \left( \frac{\omega^w_p(f, u)}{u^\gamma} \right)^q \hspace{.2em} \frac{du}{u}.
    $$
    The right-hand side goes to $0$. Multiplying by $\tau^{\gamma - \nu(p, r)}$ then gives the desired limit whenever $\gamma \geq \nu(p, r)$. If $q = \infin$ and $\gamma > \nu(p, r)$ then
    $$
    \frac{\omega^w_p(f, \tau)}{\tau^{\nu(p, r)}} \leq \tau^{\gamma - \nu(p, r)} \sup_{0 < u \leq T} \frac{\omega^w_p(f, u)}{u^\gamma} \rarr 0.
    $$
\end{proof}

\begin{remark}
    The threshold $\nu = 1 \lor \frac{r}{p}$ also governs the subdivision estimate later used in weighted sewing. Estimating the child-intervals collectively through Hölder's inequality gives a coefficient of $n^{\nu}$. Estimating them separately instead gives us the potentially larger exponent $(1 \lor \frac{1}{p}) + \frac{r-1}{p}$, which is in some sense the ``naive'' threshold obtained through direct weighted analogues of the proofs of \cite{FrizSeeger2022}.
\end{remark}

\begin{proposition}
\label{prop: Two-Parameter Cocycle Result}
    Let $V$ be a finite-dimensional vector space and $A: \bigtriangleup_2(0, T) \rarr V$ a measurable map. If
    $$
    \delta A_{s, u, t} = 0
    $$ for almost every $(s, u, t) \in \bigtriangleup_3(0, T)$ then there exists a measurable map $F: [0, T] \rarr V$ such that $A_{s, t} = F_t - F_s$ for almost every $(s, t) \in \bigtriangleup_2(0, T)$.
    Moreover, if for some $0 < p \leq \infin$,
    $$
    \overline{\Omega}^w_p(\delta A, \tau) = 0 \qquad \textup{ for all } 0 < \tau \leq T,
    $$
    then $F$ satisfies
    $$
    \Omega^w_p(A - \delta F, \tau) = 0 \qquad \textup{ for all } 0 < \tau \leq T.
    $$
\end{proposition}
\begin{proof}
    We follow here the argument of \cite[Lemma 2.9]{FrizSeeger2022}.
    We extend $A$ to $[0, T]^2$ by defining a map $\hat{A}$ with
    $$
    \hat{A}_{s, t} = \begin{cases}
        A_{s, t}, & s < t,
        \\
        -A_{t, s}, & s > t,
        \\
        0, & s = t.
    \end{cases}
    $$
    The assumption $\delta A = 0$ a.e on $\bigtriangleup_3(0, T)$ paired with antisymmetry gives
    $$
    \hat{A}_{s, t} = \hat{A}_{s, a} + \hat{A}_{a, t}
    $$ for a.e $(s, a, t) \in [0, T]^3$.
    Fubini then allows us to choose an $a \in (0, T)$ such that $t \mapsto \hat{A}_{a, t}$ is measurable and $\hat{A}_{s, t} = \hat{A}_{s, a} + \hat{A}_{a, t}$ for a.e $(s, t) \in [0, T]^2$. Set $F_t = \hat{A}_{a, t}$, then antisymmetry gives
    $$
    \hat{A}_{s, t} = -\hat{A}_{a, s} + \hat{A}_{a, t} = F_t - F_s
    $$ for a.e $(s, t) \in [0, T]^2$. Restricting to $s < t$ gives $A_{s, t} = F_t - F_s$ for a.e $(s, t) \in \bigtriangleup_2(0, T)$.
    For the additional assertion, set $R = A - \delta F$. We know $R = 0$ a.e on $\bigtriangleup_2(0, T)$. Since $\delta R = \delta A$, the stronger hypothesis gives, for each $0 < h < T$ and $0 < \theta < 1$,
    $$
    R_{s, s + h} = R_{s, s + \theta h} + R_{s + \theta h, s + h}
    $$ for a.e $s \in [0, T - h]$. For fixed $h$, by performing the changes of variable
    $$
    (s, \theta) \mapsto (s, s + \theta h), \qquad (s, \theta) \mapsto (s + \theta h, s + h),
    $$
    the terms $R_{s, s + \theta h}, R_{s + \theta h, s + h}$ vanish for a.e $(s, \theta)$. Fubini's theorem then gives $R_{s, s + h} = 0$ for a.e $s$, for each fixed $h$, as desired.
\end{proof}

\subsection{A Dyadic Estimate}
\begin{proposition}
\label{prop: A Dyadic Estimate for Weighted Besov Paths}
    Let $0 < \alpha < 1$, $0 < p < \infin$, $0 < q \leq \infin$, and $w \in A_r$. If
    $$
    \alpha > \frac{r-1}{p}
    $$ then
    $$
    [f]_{B^\alpha_{p, q}(w)} \asymp \| \{h_n^{-\alpha} D_n f\}_{n \geq 1} \|_{\ell^q},
    $$ with constants depending on $\alpha, p, q, [w]_{A_r}$, where
    $$
    h_n = 2^{-n} T, \qquad D_n f = \left( \int^{T - h_n}_0 d(f_s, f_{s + h_n})^p w_{h_n}(s) \hspace{.2em} ds \right)^{1/p}
    $$
\end{proposition}
\begin{proof}
    This is precisely a weighted adaptation of \cite[Lemma 2.2]{FrizSeeger2022}.
    Set
    $$
    h_n=2^{-n}T,\qquad D_nf=\left(\int_0^{T-h_n} d(f_s,f_{s+h_n})^p w_{h_n}(s)\,ds\right)^{1/p}, \qquad n\ge1,
    $$
    and write
    $$
    \rho=\min\{1,p\}, \qquad \beta=\frac{r-1}{p}.
    $$

    For the easier direction, monotonicity of the modulus gives
    $$
    D_nf\leq \omega_p^w(f,\tau) \qquad\text{whenever }h_n\leq\tau\le2h_n.
    $$
    Integrating over these scale intervals, or taking a supremum when
    $q=\infty$, yields
    $$
    \left\|\{h_n^{-\alpha}D_nf\}_{n\ge1}\right\|_{\ell^q}
    \lesssim [f]_{B^\alpha_{p,q}(w)}.
    $$

    For the converse, fix $0<h\leq h_n$ and expand $h$ in binary:
    $$
    h=\sum_{j\ge n}\varepsilon_jh_j,
    \qquad \varepsilon_j\in\{0,1\}.
    $$
    Telescope the increment over $[s,s+h]$ along the corresponding
    consecutive dyadic-length subintervals. Whenever
    $[t,t+h_j]\subset[s,s+h]$, Lemma~\ref{lem: Standard Muckenhoupt Weights Properties}(i) gives
    $$
    w_h(s)
    \leq [w]_{A_r}\left(\frac{h}{h_j}\right)^{r-1}w_{h_j}(t)
    \leq [w]_{A_r}2^{(j-n)(r-1)}w_{h_j}(t).
    $$
    Thus every child increment can be estimated using its own interval
    weight. Applying Minkowski's inequality when $p\geq 1$, or
    $p$-subadditivity when $p<1$, and then changing variables in each
    child integral, gives
    $$
    \omega_p^w(f,h_n)^\rho
    \lesssim
    \sum_{j \geq n}2^{(j-n)\rho\beta}(D_jf)^\rho.
    $$
    For merely measurable $f$, the infinite telescoping step is justified
    by convergence in measure of translated paths, followed by an
    almost-everywhere convergent subsequence and Fatou's lemma.

    Multiplying by $h_n^{-\alpha\rho}$, we obtain
    $$
    \bigl(h_n^{-\alpha}\omega_p^w(f,h_n)\bigr)^\rho
    \lesssim
    \sum_{k\ge0}
    2^{-k\rho(\alpha-\beta)}
    \bigl(h_{n+k}^{-\alpha}D_{n+k}f\bigr)^\rho.
    $$
    Since $\alpha>\beta$, the geometric kernel is summable.
    Young's inequality in $\ell^{q/\rho}$ when $q/\rho\ge1$,
    and subadditivity when $q/\rho<1$, therefore yield
    $$
    \left\|\{h_n^{-\alpha}\omega_p^w(f,h_n)\}_{n\ge1}\right\|_{\ell^q}
    \lesssim
    \left\|\{h_n^{-\alpha}D_nf\}_{n\ge1}\right\|_{\ell^q}.
    $$
    The case $q=\infty$ follows by taking suprema.

    Finally, split the scale integral into dyadic intervals and use
    monotonicity of the modulus. The largest scales
    $\tau\in[T/2,T]$ are controlled by the same binary-expansion argument,
    starting with $h_1=T/2$. This gives
    $$
    [f]_{B^\alpha_{p,q}(w)}
    \lesssim
    \left\|\{h_n^{-\alpha}D_nf\}_{n\ge1}\right\|_{\ell^q}.
    $$
\end{proof}

\subsection{Composition}
We conclude \S 2 with the basic composition estimate needed for later fixed-point arguments.
\begin{proposition}[Hölder Composition Estimate]
\label{prop: Hölder Composition Estimate}
    Let $0 < \delta \leq 1$, $0 < \alpha < 1$, $0 < p, q \leq \infin$, and let $F: E \rarr E'$ be a $\delta$-Hölder map of metric spaces. Then for every $Y \in B^\alpha_{p, q}(w)$ we have
    $$
    [F(Y)]_{B^{\delta \alpha}_{p/\delta, q/\delta}(w)} \leq [F]_{C^\delta} [Y]^\delta_{B^\alpha_{p, q}(w)},
    $$ with the usual convention when either $p, q = \infin$. Moreover,
    $$
    [F(Y)]_{B^{\delta \alpha}_{p, q/\delta}(w)} \leq [F]_{C^\delta} w([0, T])^{(1 - \delta)/p} [Y]^\delta_{B^\alpha_{p, q}(w)}.
    $$
    The latter form is the direct weighted analogue of \textup{\cite[Equation (2.2)]{FrizSeeger2022}} and is the form used in the arguments of \S5, 6.
\end{proposition}
\begin{proof}
    Write $L=[F]_{C^\delta}$ and $M=w([0,T])$. Since
    $$
    d_{E'}(F(Y_s),F(Y_{s+h})) \leq L\,d_E(Y_s,Y_{s+h})^\delta,
    $$
    we have
    $$
    \omega_{p/\delta}^w(F(Y), \tau) \leq L\,\omega_p^w(Y, \tau)^\delta.
    $$
    Taking the scale norm gives
    \begin{align*}
    \relax[F(Y)]_{B^{\delta\alpha}_{p/\delta,q/\delta}(w)} &\leq L\left\| \bigl(\tau^{-\alpha}\omega_p^w(Y,\tau)\bigr)^\delta \right\|_{L^{q/\delta}((0,T],d\tau/\tau)}
    \\
    &= L[Y]_{B^\alpha_{p,q}(w)}^\delta.
    \end{align*}

    For $p < \infin$, the finite-measure embedding and Lemma~\ref{lem: Shift and Scale Bounds for Muckenhoupt Weights} give
    $$
    \|g\|_{L^p(w_h)} \leq \left(\int_0^{T-h}w_h(s)\,ds\right)^{(1-\delta)/p} \|g\|_{L^{p/\delta}(w_h)} \leq M^{(1-\delta)/p}\|g\|_{L^{p/\delta}(w_h)}.
    $$
    Consequently,
    $$
    \omega_p^w(F(Y),\tau) \leq L M^{(1-\delta)/p}\omega_p^w(Y,\tau)^\delta,
    $$
    and hence
    $$
    [F(Y)]_{B^{\delta\alpha}_{p,q/\delta}(w)} \leq L M^{(1-\delta)/p}[Y]_{B^\alpha_{p,q}(w)}^\delta.
    $$
    When $p=\infty$, both modulus estimates follow directly from
    the pointwise inequality. When $q=\infty$, replace the scale
    norms by suprema.
\end{proof}
\section{Embeddings and Pointwise Control}
\label{sec: Embeddings and Pointwise Control}

The weighted Besov norms of \S 2 control increments through averages. In this section we obtain the pointwise estimates needed for sewing and rough integration, and then combine those estimates with interpolation to regain integrability. The arguments follow the embeddings and interpolation estimates of \cite[\S 2]{FrizSeeger2022}, with local deweighting retaining the dependence on the position of each increment.

For the power and logarithmic moduli used later, we apply the unweighted embeddings on a subinterval $J$, keeping the local deweighting constant. Taking $J=[s,t]$ produces the factor $|t-s|^\alpha/w([s,t])^{1/p}$. We also, following the strategy of Friz and Seeger, give a localised Campanato criterion for more general controls depending on both position and scale.

\subsection{Localisation and Deweighting}

The following lemma localises \cite[Proposition 2.8]{FrizSeeger2022}. Its role is to permit a control that depends on the basepoint as well as the scale.

\begin{lemma}[Localised Campanato Criterion]
\label{lem: Localised Campanato Criterion}
    Let $\zeta:[0,T]\times(0,T]\to(0,\infin)$ satisfy, with constants $C_\zeta,K_\zeta\geq1$,
    \begin{align}
        \zeta(x,h)&\leq C_\zeta\zeta(x,R),
        &&0<h\leq R\leq T,
        \label{eq: Control Almost Monotonicity}\\
        \zeta(x,2h)&\leq C_\zeta\zeta(x,h),
        &&0<2h\leq T,
        \label{eq: Control Doubling}\\
        \zeta(x,h)&\leq C_\zeta\zeta(y,h),
        &&|x-y|\leq h,
        \label{eq: Control Spatial Comparability}\\
        \int_0^R\zeta(x,h)\hspace{.2em}\frac{dh}{h}
        &\leq K_\zeta\zeta(x,R),
        &&0<R\leq T.
        \label{eq: Control Dini Bound}
    \end{align}
    Let $\varrho:[0,T]^2\to[0,\infin)$ be measurable. Suppose that, for some $C_0,C_1\geq0$,
    $$
    \frac1{|B_R(x)|^2}\int_{B_R(x)}\int_{B_R(x)}
    \varrho(s,t)\hspace{.2em}ds\hspace{.2em}dt
    \leq C_0\zeta(x,R),
    \qquad x\in[0,T],\quad 0<R\leq T/8,
    $$
    and that, for some $0<\vartheta\leq1/2$ and every triple of distinct points $a,b,c$,
    $$
    \varrho(a,c)\leq\varrho(a,b)+\varrho(b,c)
    +C_1\zeta\left(a\land b\land c,
    (|a-b|\land|b-c|)^\vartheta
    (|a-b|\lor|b-c|)^{1-\vartheta}\right).
    $$
    Then there is a full-measure set $\mathcal X\subset[0,T]$ such that
    $$
    \varrho(s,t)\leq C(C_0+C_1)\zeta(s,|t-s|),
    \qquad s,t\in\mathcal X,\quad s\ne t,
    $$
    where $C$ depends only on $C_\zeta,K_\zeta,\vartheta$.
\end{lemma}
\begin{proof}
    We reduce to \cite[Proposition 2.8]{FrizSeeger2022}. First, \eqref{eq: Control Almost Monotonicity} and \eqref{eq: Control Dini Bound} imply uniform power decay. Indeed, with $\lambda=\exp(-2C_\zeta K_\zeta)$,
    $$
    \frac{\log(1/\lambda)}{C_\zeta}\zeta(x,\lambda R)
    \leq\int_{\lambda R}^R\zeta(x,h)\hspace{.2em}\frac{dh}{h}
    \leq K_\zeta\zeta(x,R),
    $$
    so $\zeta(x,\lambda R)\leq\zeta(x,R)/2$. Iteration gives constants $D,\varepsilon>0$, depending only on $C_\zeta,K_\zeta$, such that
    \begin{equation}
        \zeta(x,h)\leq D(h/R)^\varepsilon\zeta(x,R),
        \qquad 0<h\leq R\leq T.
        \label{eq: Control Power Decay}
    \end{equation}

    Fix an interval $J$ with $|J|\leq T/4$ and set
    $$
    \Phi_J(h)=\sup_{x\in J}\sup_{0<u\leq h}\zeta(x,u).
    $$
    This is finite by \eqref{eq: Control Almost Monotonicity} and \eqref{eq: Control Spatial Comparability} at scale $T$. It is nondecreasing and doubling, and \eqref{eq: Control Power Decay} gives
    $$
    \Phi_J(h)\leq D(h/R)^\varepsilon\Phi_J(R),
    \qquad
    \int_0^R\Phi_J(h)\hspace{.2em}\frac{dh}{h}
    \leq\frac D\varepsilon\Phi_J(R).
    $$
    Extend it beyond $T$ by $\Phi_J(h)=\Phi_J(T)(h/T)^\varepsilon$ if necessary. The hypotheses on $\varrho$ give the mean bound and approximate triangle inequality of the cited proposition on $J$, with control $\Phi_J$. Hence there is a full-measure set $\mathcal X_J\subset J$ on which
    $$
    \varrho(s,t)\lesssim(C_0+C_1)\Phi_J(|t-s|).
    $$
    All constants are independent of $J$.

    Intersect over intervals with rational endpoints and length at most $T/4$. If $s,t$ belong to the resulting full-measure subset of $(0,T)$ and $h=|t-s|\leq T/16$, choose such a $J$ containing them with $|J|\leq4h$. The control comparisons give
    $$
    \Phi_J(h)\lesssim\zeta(s,4h)\lesssim\zeta(s,h).
    $$
    This proves the conclusion for small increments. For larger increments, join $s$ to $t$ by at most $64$ successive gaps of length at most $T/16$, with endpoints in the same full-measure set. Iterate the approximate triangle inequality. Every additional scale is at most $|t-s|$, and every basepoint lies between $s$ and $t$, so \eqref{eq: Control Almost Monotonicity} and \eqref{eq: Control Spatial Comparability} bound all terms by a constant multiple of $(C_0+C_1)\zeta(s,|t-s|)$. The same argument applies in the reverse order.
\end{proof}

We next record the weight controls and local constants used in the embeddings. The deweighting assertion is the local form of Theorem~2.10.

\begin{lemma}[Weight Controls and Local Deweighting]
\label{lem: Weight Controls and Local Deweighting}
    Let $1\leq r<\infin$, $w\in A_r$, and $0<p<\infin$.
    \begin{enumerate}[label=(\alph*)]
        \item For $\gamma>r/p$, the control
        $$
        \zeta_{\gamma,p,w}(x,h)=\frac{h^\gamma}{w(B_h(x))^{1/p}}
        $$
        satisfies \eqref{eq: Control Almost Monotonicity}--\eqref{eq: Control Dini Bound}. More precisely,
        $$
        \zeta_{\gamma,p,w}(x,h)
        \lesssim(h/R)^{\gamma-r/p}\zeta_{\gamma,p,w}(x,R),
        \qquad 0<h\leq R\leq T.
        $$

        \item For a nondegenerate interval $J\subset[0,T]$, put
        $$
        c_J=[w]_{A_r}^{1/p}\frac{|J|^{r/p}}{w(J)^{1/p}},
        \qquad u=p/r.
        $$
        For every admissible modulus $\omega$, $0<q\leq\infin$, and measurable two-parameter map $A$,
        $$
        \|A\|_{\mathbb B^\omega_{u,q}(J)}
        \leq c_J\|A\|_{\mathbb B^\omega_{p,q}(w;J)}.
        $$
        The same estimate holds for path seminorms.

        \item For $x\in J$ and $0<h\leq|J|$,
        $$
        w(B_h(x))\geq[w]_{A_r}^{-1}(h/|J|)^r w(J).
        $$
        Consequently, for every positive function $\omega$,
        $$
        \frac{\omega(h)}{w(B_h(x))^{1/p}}
        \leq c_J\omega(h)h^{-r/p}.
        $$
    \end{enumerate}
    For $p=\infin$, use $\zeta_{\gamma,\infin,w}(x,h)=h^\gamma$, $c_J=1$, and $u=\infin$.
\end{lemma}
\begin{proof}
    The small-set estimate gives
    $$
    \frac{w(B_R(x))}{w(B_h(x))}
    \lesssim[w]_{A_r}(R/h)^r,
    $$
    proving the power bound in (a). Its positive exponent gives \eqref{eq: Control Almost Monotonicity} and \eqref{eq: Control Dini Bound}. Moreover,
    $$
    \zeta_{\gamma,p,w}(x,2h)\leq2^\gamma\zeta_{\gamma,p,w}(x,h).
    $$
    If $|x-y|\leq h$, then $B_h(x)\subset B_{2h}(y)$ and conversely, so doubling of the weighted measure gives \eqref{eq: Control Spatial Comparability}.

    Part (b) follows from Theorem~2.10 and its bound on the local deweighting constant. For (c), the interval $E=B_h(x)\cap J$ has length at least $h$. Apply the small-set estimate to $E\subset J$ and use $w(B_h(x))\geq w(E)$. The conventions for $p=\infin$ give the remaining assertions directly.
\end{proof}

\subsection{The One-Parameter Embedding}

The next theorem follows from \cite[Proposition 2.1]{FrizSeeger2022} by local deweighting. Keeping the local constant gives the stated interval-mass factor.

\begin{theorem}[Weighted Besov--H\"older Embedding]
\label{thm: Weighted Besov Holder Embedding}
    Let $(E,d)$ be a complete metric space, let $1\leq r<\infin$ and $w\in A_r$, and suppose
    $$
    0<p,q\leq\infin,\qquad r/p<\alpha<1.
    $$
    Every $f\in B^\alpha_{p,q}(w;E)$ has a unique continuous representative. For this representative and every $0\leq s<t\leq T$,
    \begin{equation}
        d(f_s,f_t)\leq C[w]_{A_r}^{1/p}
        \frac{|t-s|^\alpha}{w([s,t])^{1/p}}
        [f]_{B^\alpha_{p,q}(w;[s,t])},
        \label{eq: Local Weighted Path Embedding}
    \end{equation}
    where $C$ depends only on $\alpha,p,q,r$. In particular,
    $$
    d(f_s,f_t)\lesssim
    \frac{|t-s|^\alpha}{w([s,t])^{1/p}}
    [f]_{B^\alpha_{p,q}(w)},
    \qquad
    B^\alpha_{p,q}(w;E)\hookrightarrow C^{\alpha-r/p}([0,T];E).
    $$
    For $p<\infin$, the latter seminorm bound is
    $$
    [f]_{C^{\alpha-r/p}}
    \leq C[w]_{A_r}^{2/p}
    \left(\frac{T^r}{w([0,T])}\right)^{1/p}
    [f]_{B^\alpha_{p,q}(w)}.
    $$
    For $p=\infin$, the weight factors are omitted.
\end{theorem}
\begin{proof}
    Suppose first that $p<\infin$ and put $u=p/r$. Theorem~2.10, including its $L^p$ estimate, and \cite[Proposition 2.1]{FrizSeeger2022} supply a continuous representative, since $\alpha>1/u$. On $J=[s,t]$, with $L=t-s$, the same unweighted embedding and Lemma~\ref{lem: Weight Controls and Local Deweighting}(b) give
    $$
    d(f_s,f_t)\leq CL^{\alpha-1/u}[f]_{B^\alpha_{u,q}(J)}
    \leq C[w]_{A_r}^{1/p}
    \frac{L^\alpha}{w(J)^{1/p}}
    [f]_{B^\alpha_{p,q}(w;J)}.
    $$
    The constants are uniform under translation and rescaling of $J$. This proves \eqref{eq: Local Weighted Path Embedding}, and restriction gives the global-seminorm version. The small-set bound
    $$
    w(J)\geq[w]_{A_r}^{-1}(L/T)^r w([0,T])
    $$
    gives the H\"older estimate. For $p=\infin$, apply the unweighted embedding directly. Uniqueness follows because continuous representatives agreeing almost everywhere agree everywhere.
\end{proof}

\begin{remark}[Global and Local Deweighting]
\label{rem: Global and Local Deweighting}
    Applying deweighting once on $[0,T]$ gives, for $r>1$,
    $$
    d(f_s,f_t)\lesssim
    \sigma([0,T])^{(r-1)/p}|t-s|^{\alpha-r/p}
    [f]_{B^\alpha_{p,q}(w)},
    \qquad \sigma=w^{-1/(r-1)}.
    $$
    Applying it on $J=[s,t]$ retains $\sigma(J)$, which the $A_r$ inequality converts into the interval-mass factor in \eqref{eq: Local Weighted Path Embedding}. Thus local deweighting and the unweighted embedding already recover the location-sensitive estimate. The direct weighted sewing and stability results proved later additionally control the weighted Besov norms of the integrals and remainders.
\end{remark}

\begin{example}[Power Weights]
\label{ex: Power Weight Embedding}
    Let $w(t)=t^\beta$, with $\beta>-1$. Then
    $$
    w\in A_r\quad\Longleftrightarrow\quad-1<\beta<r-1
    \quad(1<r<\infin),
    $$
    while $w\in A_1$ exactly when $-1<\beta\leq0$. These statements follow by evaluating the averages of $t^\beta$ and its dual weight; for $A_1$, compare the average with the essential infimum.

    Suppose $0<p<\infin$, $0<q\leq\infin$, and
    $$
    \frac{1+\max\{\beta,0\}}p<\alpha<1.
    $$
    An admissible weight index exists for Theorem~\ref{thm: Weighted Besov Holder Embedding}. Writing $N=[f]_{B^\alpha_{p,q}(t^\beta)}$, elementary integration gives
    $$
    w([s,s+h])=\frac{(s+h)^{\beta+1}-s^{\beta+1}}{\beta+1}
    \asymp_\beta h(s+h)^\beta,
    $$
    and therefore
    $$
    d(f_s,f_{s+h})\lesssim
    Nh^{\alpha-1/p}(s+h)^{-\beta/p}
    \asymp_\beta Nh^{\alpha-1/p}\max\{s,h\}^{-\beta/p}.
    $$
    In particular,
    $$
    d(f_s,f_{s+h})\lesssim
    \begin{cases}
        Nh^{\alpha-1/p}s^{-\beta/p},&0<h\leq s,\\
        Nh^{\alpha-(1+\beta)/p},&0\leq s\leq h.
    \end{cases}
    $$
    The local seminorm on $[s,s+h]$ may replace $N$. For $\beta>0$, taking $r=1+\beta+\varepsilon$ in the global H\"older embedding only gives exponent $\alpha-(1+\beta+\varepsilon)/p$. Letting $\varepsilon\downarrow0$ does not establish the endpoint exponent, since the constants need not stay bounded. The displayed endpoint power follows from the explicit interval-mass estimate.
\end{example}

\subsection{Two-Parameter Embeddings}

A general two-parameter map need not be additive. As in \cite[Proposition 2.7]{FrizSeeger2022}, its pointwise control therefore requires a condition on $\delta A$ in addition to its Besov norm. For $s<u<t$, write
$$
\ell_\vartheta(s,u,t)
=\bigl((u-s)\land(t-u)\bigr)^\vartheta
\bigl((u-s)\lor(t-u)\bigr)^{1-\vartheta},
\qquad 0<\vartheta\leq1/2.
$$

\begin{theorem}[Refined Two-Parameter Embedding]
\label{thm: Refined Two Parameter Embedding}
    Let $1<p<\infin$, $0<q\leq\infin$, and $w\in A_p$. Let $\omega:(0,T]\to(0,\infin)$ be nondecreasing, with $\omega(h)\to0$ as $h\downarrow0$, and suppose that
    $$
    \zeta_{\omega,p,w}(x,h)=\frac{\omega(h)}{w(B_h(x))^{1/p}}
    $$
    satisfies \eqref{eq: Control Almost Monotonicity}--\eqref{eq: Control Dini Bound}. Let $A:\triangle_2(0,T)\to V$ be measurable, with
    $$
    N=\|A\|_{\mathbb B^\omega_{p,q}(w)}<\infin.
    $$
    Suppose that, for some $M\geq0$ and $0<\vartheta\leq1/2$,
    $$
    |\delta A_{s,u,t}|
    \leq M\zeta_{\omega,p,w}(s,\ell_\vartheta(s,u,t)),
    \qquad 0\leq s<u<t\leq T.
    $$
    Then $A$ has a unique continuous representative on the closed simplex, vanishing on the diagonal, and
    $$
    |A_{s,t}|\leq C(N+M)
    \frac{\omega(t-s)}{w([s,t])^{1/p}},
    \qquad 0\leq s<t\leq T.
    $$
    The constant depends only on $p,q,[w]_{A_p},\vartheta$ and the control constants.
\end{theorem}
\begin{proof}
    We follow the Campanato argument for \cite[Proposition 2.7]{FrizSeeger2022}, using Lemma~\ref{lem: Localised Campanato Criterion}. Write $\zeta=\zeta_{\omega,p,w}$. Doubling of $\zeta$ and of the weighted measure implies that $\omega$ is doubling. Moreover,
    \begin{equation}
        \Omega_p^w(A,h)\lesssim N\omega(h),
        \qquad 0<h\leq T/2.
        \label{eq: Besov Small Scale Modulus Bound}
    \end{equation}
    For finite $q$, integrate the defining norm over $[h,2h]$ and use monotonicity of $\Omega_p^w(A,\cdot)$ and doubling of $\omega$; for $q=\infin$, the estimate is immediate.

    Extend $A$ antisymmetrically and put $\varrho(s,t)=|A_{s,t}|$. Fix $B=B_R(x)=[c,d]$, with $R\leq T/8$, and write $L=|B|$. H\"older's inequality and the inverse-weight bound from Theorem~2.10 give, for $0<h<L$,
    $$
    \int_c^{d-h}|A_{s,s+h}|\hspace{.2em}ds
    \leq[w]_{A_p}^{1/p}\frac{L}{w(B)^{1/p}}\Omega_p^w(A,h).
    $$
    Since $R\leq L\leq2R\leq T/4$, \eqref{eq: Besov Small Scale Modulus Bound} yields
    $$
    \frac1{L^2}\int_B\int_B|A_{s,t}|\hspace{.2em}ds\hspace{.2em}dt
    \lesssim\frac{N}{Lw(B)^{1/p}}\int_0^L\omega(h)\hspace{.2em}dh
    \lesssim N\zeta(x,R).
    $$
    This verifies the mean condition and local integrability. The identity defining $\delta A$ gives the approximate triangle inequality. Under antisymmetric extension, $|\delta A|$ is invariant under permutations of its arguments. The mixed scale of the two adjacent gaps in their increasing ordering is at most the mixed scale of $|a-b|$ and $|b-c|$, so \eqref{eq: Control Almost Monotonicity} verifies the condition for every ordering.

    Lemma~\ref{lem: Localised Campanato Criterion} supplies a full-measure set $\mathcal X$ on which
    $$
    |A_{s,t}|\lesssim(N+M)\zeta(s,|t-s|).
    $$
    By \eqref{eq: Control Power Decay} and the spatial comparison at scale $T$, $\sup_x\zeta(x,h)\to0$ as $h\downarrow0$. Thus $A$ is uniformly continuous towards the diagonal. Away from it, use
    $$
    |A_{s,t}-A_{s',t}|
    \leq|A_{s,s'}|+|\delta A_{s,s',t}|,
    \qquad s<s'<t,
    $$
    and the corresponding identity for the right endpoint. The defect bound makes both endpoint variations tend uniformly to zero. Hence $A$ extends continuously from $(\mathcal X\times\mathcal X)\cap\triangle_2(0,T)$ to the closed simplex.

    For each fixed $h>0$, both $s$ and $s+h$ lie in $\mathcal X$ for almost every basepoint $s$, so this replacement preserves the increment norms. On $\mathcal X\times\mathcal X$, the stated bound follows from $[s,t]\subset B_{t-s}(s)$. For arbitrary $s<t$, approximate the endpoints from within $[s,t]$ and use monotonicity of $\omega$, continuity of the representative, and absolute continuity of weighted mass. Uniqueness follows from density.
\end{proof}

For power regularity and the logarithmic moduli used below, local deweighting gives the following alternative. Its explicit decay assumption allows all positive integrability and scale indices.

\begin{theorem}[Two-Parameter Embedding in the Quasi-Banach Range]
\label{thm: Quasi Banach Two Parameter Embedding}
    Let $0<p,q\leq\infin$, $1\leq r<\infin$, and $w\in A_r$. Let $\omega:(0,T]\to(0,\infin)$ be nondecreasing and doubling, with
    $$
    \omega(2h)\leq D_\omega\omega(h),
    \qquad 0<2h\leq T.
    $$
    Suppose there are $d>r/p$ and $C_\omega\geq1$ such that
    \begin{equation}
        \omega(h)\leq C_\omega(h/L)^d\omega(L),
        \qquad 0<h\leq L\leq T.
        \label{eq: Modulus Positive Decay Gap}
    \end{equation}
    Let $A:\triangle_2(0,T)\to V$ be measurable, with $\|A\|_{\mathbb B^\omega_{p,q}(w)}<\infin$. Suppose that, for some $M\geq0$ and $0<\vartheta\leq1/2$,
    $$
    |\delta A_{s,u,t}|
    \leq M\frac{\omega(\ell_\vartheta(s,u,t))}
    {w(B_{\ell_\vartheta(s,u,t)}(s))^{1/p}},
    \qquad 0\leq s<u<t\leq T.
    $$
    Then $A$ has a unique continuous representative, vanishing on the diagonal, and
    $$
    |A_{s,t}|\leq C\left(
    \|A\|_{\mathbb B^\omega_{p,q}(w;[s,t])}+M\right)
    \frac{\omega(t-s)}{w([s,t])^{1/p}},
    \qquad 0\leq s<t\leq T.
    $$
    Here $C$ depends only on $p,q,r,[w]_{A_r},d,C_\omega,D_\omega,\vartheta$. In particular, for $\omega(h)=h^\gamma$, the assertion applies whenever $\gamma>r/p$, including $p<1$ or $q<1$.
\end{theorem}
\begin{proof}
    We apply \cite[Proposition 2.7]{FrizSeeger2022} after local deweighting. First suppose $p<\infin$ and put $u=p/r$. On an interval $J$, Lemma~\ref{lem: Weight Controls and Local Deweighting} gives
    $$
    \|A\|_{\mathbb B^\omega_{u,q}(J)}
    \leq c_J\|A\|_{\mathbb B^\omega_{p,q}(w;J)},
    \qquad
    |\delta A_{s,v,t}|
    \leq Mc_J\omega(\ell_\vartheta)\ell_\vartheta^{-1/u}
    $$
    for $s<v<t$ in $J$.

    Put $e=d-1/u>0$, $\chi(h)=\omega(h)h^{-1/u}$, and $\chi^*(h)=\sup_{0<v\leq h}\chi(v)$. Condition~\eqref{eq: Modulus Positive Decay Gap} gives
    $$
    \chi\leq\chi^*\leq C_\omega\chi,
    \qquad
    \chi^*(h)\leq C_\omega(h/L)^e\chi^*(L).
    $$
    The envelope $\chi^*$ is nondecreasing and doubling. With $a_0=1\land u$,
    $$
    \int_0^L\chi^*(h)^{a_0}\hspace{.2em}\frac{dh}{h}
    \leq\frac{C_\omega^{a_0}}{ea_0}\chi^*(L)^{a_0}.
    $$
    Therefore $\widetilde\omega(h)=h^{1/u}\chi^*(h)$ is equivalent to $\omega$ and satisfies the modulus hypotheses of the cited proposition. The displayed defect estimate supplies its mixed continuity assumption. Its conclusion, with constants uniform in $J$, is
    $$
    |A_{s,t}|\leq Cc_J\left(
    \|A\|_{\mathbb B^\omega_{p,q}(w;J)}+M\right)
    \omega(t-s)|t-s|^{-1/u},
    \qquad s,t\in J.
    $$
    Apply this first on $[0,T]$ to obtain the continuous representative. The local representatives agree with it by continuity and almost-everywhere agreement. Taking $J=[s,t]$ cancels the powers $|J|^{r/p}$ and $|t-s|^{-1/u}$ and proves the assertion. As in the preceding proof, the construction in \cite[Proposition 2.7]{FrizSeeger2022} preserves the increment norms at every fixed length.

    For $p=\infin$, set $u=\infin$, $c_J=1$, and $\chi=\omega$. The same argument applies, with the weight factors omitted. The unweighted proposition includes all positive $u,q$, so no additional Banach-range restriction is needed.
\end{proof}

The preceding theorem applies to
$$
\omega(h)=h^\gamma(1+\log(T/h))^b,
\qquad \gamma>r/p,\quad b\geq0,
$$
up to replacement by an equivalent nondecreasing modulus. Indeed, for $r/p<d<\gamma$ and $x=\log(L/h)\geq0$,
$$
\frac{\omega(h)}{\omega(L)}
\leq(h/L)^d e^{-(\gamma-d)x}(1+x)^b
\lesssim(h/L)^d.
$$
Thus $\widehat\omega(h)=\sup_{0<v\leq h}\omega(v)$ is comparable to $\omega$, is nondecreasing and doubling, and satisfies \eqref{eq: Modulus Positive Decay Gap}.

For the critical sewing exponent $\gamma_c=1\lor r/p$, the required gap $\gamma_c-r/p$ is positive when $p>r$. If $r>1$, openness supplies $1<r_-<r$ with $w\in A_{r_-}$, and $\gamma_c>r_-/p$, so the embedding applies with $r_-$ while keeping $\gamma_c$ fixed. When $p<r$, this also gives $1\lor r_-/p<\gamma_c$, making sewing subcritical at that smoothness. When $p=r>1$, the sewing exponent remains $1$: openness restores the embedding gap but leaves the exponent-one sewing endpoint. For $r=1$ and $p\leq1$, the stated theorem gives no endpoint conclusion.

These observations concern the embedding assumptions; the critical scale-index conditions are imposed in \S 4. In particular, if a critical remainder satisfies a power bound in $\mathbb B^{\gamma_c}_{p,\infin}(w)$ and the corresponding mixed defect estimate, then, whenever a positive embedding gap is available,
$$
|A_{s,t}|\lesssim\left(
\|A\|_{\mathbb B^{\gamma_c}_{p,\infin}(w)}+M\right)
\frac{|t-s|^{\gamma_c}}{w([s,t])^{1/p}}.
$$

\subsection{Interpolation and Integrability Regain}

Sewing controls a remainder at the integrability index of its defect. The following weighted versions of \cite[Lemmas 2.7 and 2.8]{FrizSeeger2022} combine that bound with pointwise control to increase the integrability index. The interpolation uses the same measure $w_h(s)\hspace{.2em}ds$ at each fixed scale and requires no further Muckenhoupt estimate.

\begin{lemma}[Interpolation]
\label{lem: Weighted Interpolation}
    Let $A:\triangle_2(0,T)\to V$ be measurable, let $w$ be a positive integrable weight, and let $0<p<P\leq\infin$, $0<q\leq\infin$. For finite $P$, put $\theta=p/P$.
    \begin{enumerate}[label=(\alph*)]
        \item If $P<\infin$, then
        $$
        \Omega_P^w(A,\tau)
        \leq\Omega_\infin(A,\tau)^{1-\theta}
        \Omega_p^w(A,\tau)^\theta.
        $$

        \item Suppose $P<\infin$, $0<\alpha<\gamma$, and
        $$
        \Omega_\infin(A,\tau)\leq H\tau^\delta,
        \qquad \|A\|_{\mathbb B^\gamma_{p,q}(w)}<\infin,
        \qquad H\geq0,\quad\delta\geq0.
        $$
        If $\kappa=\delta(1-\theta)+\gamma\theta-\alpha>0$, then
        $$
        \|A\|_{\mathbb B^\alpha_{P,q}(w)}
        \leq C_{\kappa,q,\theta}T^\kappa H^{1-\theta}
        \|A\|_{\mathbb B^\gamma_{p,q}(w)}^\theta,
        $$
        where
        $$
        C_{\kappa,q,\theta}
        =\left(\frac{1-\theta}{\kappa q}\right)^{(1-\theta)/q}
        \quad(q<\infin),\qquad
        C_{\kappa,\infin,\theta}=1.
        $$
        For $P=\infin$ and $\delta>\alpha$, the direct estimate is
        $$
        \|A\|_{\mathbb B^\alpha_{\infin,q}}
        \leq
        \begin{cases}
            HT^{\delta-\alpha}/((\delta-\alpha)q)^{1/q},&q<\infin,\\
            HT^{\delta-\alpha},&q=\infin.
        \end{cases}
        $$

        \item More generally, let $\alpha>0$ and suppose $P<\infin$,
        $$
        \Omega_\infin(A,\tau)\leq H\rho(\tau),
        \qquad \|A\|_{\mathbb B^\omega_{p,q}(w)}<\infin,
        $$
        where $H\geq0$, $\rho$ is positive and measurable, and $\omega$ is an admissible modulus. Set
        $$
        K(\tau)=\tau^{-\alpha}\rho(\tau)^{1-\theta}\omega(\tau)^\theta.
        $$
        Then
        $$
        \|A\|_{\mathbb B^\alpha_{P,q}(w)}
        \leq H^{1-\theta}
        \|K\|_{L^{q/(1-\theta)}((0,T),d\tau/\tau)}
        \|A\|_{\mathbb B^\omega_{p,q}(w)}^\theta,
        $$
        whenever the right-hand side is finite, with the supremum convention for $q=\infin$. For $P=\infin$,
        $$
        \|A\|_{\mathbb B^\alpha_{\infin,q}}
        \leq H\|\tau^{-\alpha}\rho(\tau)\|_{L^q((0,T),d\tau/\tau)}.
        $$
    \end{enumerate}
\end{lemma}
\begin{proof}
    At a fixed increment length $h$,
    $$
    \int_0^{T-h}|A_{s,s+h}|^P w_h(s)\hspace{.2em}ds
    \leq\left(\operatorname*{ess\,sup}_{0\leq s\leq T-h}|A_{s,s+h}|\right)^{P-p}
    \int_0^{T-h}|A_{s,s+h}|^p w_h(s)\hspace{.2em}ds.
    $$
    Taking $P$th roots and then the supremum over $h\leq\tau$ proves (a), for every $p>0$.

    For (b), put $G(\tau)=\tau^{-\gamma}\Omega_p^w(A,\tau)$. Part (a) gives
    $$
    \tau^{-\alpha}\Omega_P^w(A,\tau)
    \leq H^{1-\theta}\tau^\kappa G(\tau)^\theta.
    $$
    For finite $q$, apply H\"older to the $q$th power with exponents $1/(1-\theta)$ and $1/\theta$, and evaluate
    $$
    \int_0^T\tau^{\kappa q/(1-\theta)}\hspace{.2em}\frac{d\tau}{\tau}
    =\frac{1-\theta}{\kappa q}T^{\kappa q/(1-\theta)}.
    $$
    Taking $q$th roots gives the stated constant. For $q=\infin$, take suprema. The argument remains valid for $q<1$, since H\"older is applied to the nonnegative integral.

    For (c), replace $G$ by $G_\omega(\tau)=\Omega_p^w(A,\tau)/\omega(\tau)$ to obtain
    $$
    \tau^{-\alpha}\Omega_P^w(A,\tau)
    \leq H^{1-\theta}K(\tau)G_\omega(\tau)^\theta,
    $$
    and repeat the same H\"older or supremum estimate. The assertions for $P=\infin$ follow directly by integrating the assumed pointwise modulus bound.
\end{proof}

Let $w\in A_r$, $0<p<P\leq\infin$, $0<q\leq\infin$, $\alpha>0$, and $\gamma>r/p$. Put
$$
\theta=p/P,\qquad
Q_w=[w]_{A_r}^{1/p}\left(\frac{T^r}{w([0,T])}\right)^{1/p},
\qquad
\Theta=\gamma-\alpha-r\left(\frac1p-\frac1P\right),
$$
with $\theta=0$ when $P=\infin$.

First suppose $\omega(h)=h^\gamma\ell(h)$ is an admissible modulus, up to equivalence, and
$$
|A_{s,t}|\leq N\frac{|t-s|^\gamma\ell(t-s)}{w([s,t])^{1/p}},
\qquad
\|A\|_{\mathbb B^\omega_{p,q}(w)}\leq N.
$$
The lower-mass estimate gives
$$
\Omega_\infin(A,h)\leq NQ_w
\sup_{0<v\leq h}v^{\gamma-r/p}\ell(v).
$$
If this supremum is bounded by a constant times $h^{\gamma-r/p}\ell(h)$, as for the logarithmic losses above, Lemma~\ref{lem: Weighted Interpolation}(c) applies with $\rho(h)=h^{\gamma-r/p}\ell(h)$. Since
$$
h^{-\alpha}\rho(h)^{1-\theta}\omega(h)^\theta
=h^\Theta\ell(h),
$$
it yields
$$
\|A\|_{\mathbb B^\alpha_{P,q}(w)}
\lesssim NQ_w^{1-\theta}
\|h^\Theta\ell(h)\|_{L^{q/(1-\theta)}((0,T),dh/h)}.
$$
For $P=\infin$, use the direct estimate in the same lemma. When $\Theta>0$, the Hardy bounds on $\ell$ used in \S 4 give
$$
\|A\|_{\mathbb B^\alpha_{P,q}(w)}
\lesssim NQ_w^{1-\theta}T^\Theta\ell(T).
$$

For the critical application, a power estimate at scale index $\infin$ can be used directly. Suppose instead that
$$
\Omega_p^w(A,h)\leq Nh^\gamma,
\qquad
|A_{s,t}|\leq N\frac{|t-s|^\gamma}{w([s,t])^{1/p}}.
$$
Then $\Omega_\infin(A,h)\leq NQ_wh^{\gamma-r/p}$, and interpolation at each scale gives
$$
\Omega_P^w(A,h)
\leq NQ_w^{1-\theta}
h^{\gamma-r(1/p-1/P)}
=NQ_w^{1-\theta}h^{\alpha+\Theta}.
$$
If $\Theta>0$, integration over scales yields
$$
\|A\|_{\mathbb B^\alpha_{P,q}(w)}
\leq NQ_w^{1-\theta}
\begin{cases}
    T^\Theta/(\Theta q)^{1/q},&q<\infin,\\
    T^\Theta,&q=\infin.
\end{cases}
$$
This gives integrability regain without a logarithmic factor and does not require a finite-$q$ remainder bound at the original smoothness $\gamma$. If a smaller admissible weight index $r_-$ is used to obtain the embedding gap, replace $r$ by $r_-$ in $Q_w$ and $\Theta$, keeping $\gamma$ fixed.

\section{A Weighted Besov Sewing Lemma}
In this section we generalise, to the weighted setting, the sewing lemma of \cite{FrizSeeger2022}, which in turn generalises those of \cite{Gubinelli2004Controlling}, \cite{FeyeldeLaPradelle2006Curvilinear}, \cite{FeyeldeLaPradelleMokobodzki2008Noncommutative}. Our approach mirrors that of Friz and Seeger: (i) we control dyadic Riemann-type sums of a two-parameter map $\Xi$, (ii) prove ``subcritical'' and ``critical'' sewing results associated with different parameter regimes, and (iii) use the interpolation result of \S 3 to prove an ``integrability regain'' theorem, which ensures that we can place the sewn remainder in a function space appropriate for Picard iteration arguments.

Given some $\Xi: \bigtriangleup_2(0, T) \rarr V$ and a partition $\pi = \{ 0 = \tau_0 < \tau_1 < ... < \tau_m = 1 \}$ we define the corresponding Riemann sum $I_\pi \Xi_{s, t}$ and remainder $R_\pi$ with
$$
I_\pi \Xi_{s, t} = \sum_{i = 1}^{m} \Xi_{s + \tau_{i - 1}(t-s), s + \tau_i(t - s)}, \qquad R_\pi \Xi = I_\pi \Xi - \Xi,
$$ and we set $\pi_n = \{ k2^{-n}: 0 \leq k \leq 2^{n} \}$ to be the $n$th level dyadic partition.

\begin{lemma}[Subdivision and Dyadic Bounds]
\label{lem:Dyadic Bounds for Sewing}
    Let $w \in A_r$ and $0 < p_2 < \infin$, and set
    $$
    a_2 = 1 \land p_2, \qquad \nu_2 = 1 \lor \frac{r}{p_2}.
    $$
    Let $V$ be a finite-dimensional normed vector space and let $\Xi: \bigtriangleup_2(0, T) \rarr V$ be measurable. Then the following hold:
    \begin{enumerate}[label=(\alph*)]
        \item\textup{(Subdivision Bounds)} Let
        $$
        \pi = \{0 = \tau_0 < ... < \tau_N = 1\}, \qquad \lambda_i = \tau_i - \tau_{i - 1}.
        $$
        For every measurable $B: \bigtriangleup_2(0, T) \rarr V$ and every $0 < h \leq T$,
        $$
        \Omega^w_{p_2}(I_\pi B, h)^{p_2} \leq [w]_{A_r} \sum_{i = 1}^{N} \lambda_i^{1 - p_2 \nu_2} \Omega^w_{p_2}(B, \lambda_i h)^{p_2}.
        $$
        The same estimate holds for any subset of the summands defining $I_\pi B$, with the corresponding terms omitted from the right-hand side.

        \item\textup{(Dyadic Bounds)} For every integer $\ell \geq 0$ and every $0 < h \leq T$,
        $$
        \Omega^w_{p_2}(R_{\pi_{\ell + 1}} \Xi - R_{\pi_{\ell}} \Xi, h) \lesssim 2^{\ell \nu_2} \overline{\Omega}^w_{p_2}(\delta \Xi, 2^{-\ell} h).
        $$
        Consequently, for integers $0 \leq m < n$,
        $$
        \Omega^w_{p_2}(R_{\pi_n} \Xi - R_{\pi_m} \Xi, h)^{a_2} \lesssim \sum_{\ell = m}^{n - 1} 2^{\ell a_2 \nu_2} \overline{\Omega}^w_{p_2}(\delta \Xi, 2^{-\ell} h)^{a_2}.
        $$

        \item\textup{(Integral Bound)} Let $0 < b \leq a_2$. For integers $0 \leq m < n$ and $0 < h \leq T$ satisfying $2^{1 - m} h \leq T$,
        $$
        \Omega^w_{p_2}(R_{\pi_n} \Xi - R_{\pi_m} \Xi, h) \lesssim h^{\nu_2} \left[ \int^{2^{1 - m} h}_{0} \left( \frac{\overline{\Omega}^w_{p_2}(\delta \Xi, \tau)}{\tau^{\nu_2}} \right)^{b} \hspace{.2em} \frac{d\tau}{\tau}  \right]^{1/b}.
        $$
    \end{enumerate}
    The implicit constants depend only on $p_2, r, [w]_{A_r}$, and, in (c), additionally on $b$.
\end{lemma}
\begin{proof}
    Set
    $$
    p := p_2, \qquad a := 1 \land p, \qquad \nu := 1 \lor \frac{r}{p},
    $$
    so $p\nu = p \lor r$. The dyadic telescoping and sum-to-integral arguments follow \cite[Lemma 3.1]{FrizSeeger2022}; the weighted input is the joint partition estimate of Lemma~\ref{lem: Shift and Scale Bounds for Muckenhoupt Weights}(v).

    For (a), fix $0 < u \leq h$ and $s \in [0, T-u]$. The intervals
    $$
    I = [s, s+u], \qquad I_i = [s+\tau_{i-1}u, s+\tau_i u]
    $$
    partition $I$, with $|I_i| = \lambda_i u$, and
    $$
    \langle w \rangle_I = w_u(s), \qquad \langle w \rangle_{I_i} = w_{\lambda_i u}(s+\tau_{i-1}u).
    $$
    Applying Lemma~\ref{lem: Shift and Scale Bounds for Muckenhoupt Weights}(v) and integrating over $s$ gives
    \begin{equation}
        \int_0^{T-u} |(I_\pi B)_{s,s+u}|^p w_u(s) \hspace{.2em} ds \leq [w]_{A_r} \sum_{i=1}^N \lambda_i^{1-p\nu} \int_0^{T-u} |B_{s+\tau_{i-1}u,s+\tau_i u}|^p w_{\lambda_i u}(s+\tau_{i-1}u) \hspace{.2em} ds. \label{eq: Sewing Bound Lemma Intermediate 1}
    \end{equation}
    With $v=s+\tau_{i-1}u$, the integration interval is contained in $[0,T-\lambda_i u]$, so
    \begin{equation}
        \int_0^{T-u} |B_{s+\tau_{i-1}u,s+\tau_i u}|^p w_{\lambda_i u}(s+\tau_{i-1}u) \hspace{.2em} ds \leq \int_0^{T-\lambda_i u} |B_{v,v+\lambda_i u}|^p w_{\lambda_i u}(v) \hspace{.2em} dv \leq \Omega_p^w(B,\lambda_i h)^p. \label{eq: Sewing Bound Lemma Intermediate 2}
    \end{equation}
    Substituting \eqref{eq: Sewing Bound Lemma Intermediate 2} into \eqref{eq: Sewing Bound Lemma Intermediate 1} and taking the supremum over $0<u\leq h$ proves (a). The subset assertion follows by setting the omitted $z_i$ equal to zero in the same partition estimate.

    For (b), define the two-parameter map
    $$
    B_{s,t} := \delta\Xi_{s,(s+t)/2,t}.
    $$
    The dyadic telescoping identity gives
    $$
    (R_{\pi_{\ell+1}}\Xi-R_{\pi_\ell}\Xi)_{s,t} = -(I_{\pi_\ell}B)_{s,t}.
    $$
    Applying (a) with $N=2^\ell$ and $\lambda_i=2^{-\ell}$, and using $\Omega_p^w(B,\tau)\leq\overline{\Omega}_p^w(\delta\Xi,\tau)$, yields
    $$
    \Omega_p^w(R_{\pi_{\ell+1}}\Xi-R_{\pi_\ell}\Xi,h)^p \leq [w]_{A_r} 2^{\ell p\nu} \overline{\Omega}_p^w(\delta\Xi,2^{-\ell}h)^p.
    $$
    Taking $p$th roots proves the first bound. Telescoping over $\ell=m,\ldots,n-1$ and using the triangle inequality when $p\geq1$, or $p$-subadditivity when $p<1$, gives
    $$
    \Omega_p^w(R_{\pi_n}\Xi-R_{\pi_m}\Xi,h)^a \lesssim \sum_{\ell=m}^{n-1} 2^{\ell a\nu} \overline{\Omega}_p^w(\delta\Xi,2^{-\ell}h)^a,
    $$
    proving the second bound.

    For (c), write $G(\tau):=\overline{\Omega}_p^w(\delta\Xi,\tau)$. Since $b/a\leq1$, part (b) and subadditivity give
    $$
    \Omega_p^w(R_{\pi_n}\Xi-R_{\pi_m}\Xi,h)^b \lesssim \sum_{\ell=m}^{n-1} 2^{\ell b\nu} G(2^{-\ell}h)^b.
    $$
    By monotonicity of $G$,
    $$
    2^{\ell b\nu}G(2^{-\ell}h)^b \lesssim h^{b\nu} \int_{2^{-\ell}h}^{2^{1-\ell}h} \left(\frac{G(\tau)}{\tau^\nu}\right)^b \hspace{.2em} \frac{d\tau}{\tau}.
    $$
    These integration intervals have disjoint interiors and lie in $(0,2^{1-m}h]\subseteq(0,T]$. Summing and taking $b$th roots therefore gives
    $$
    \Omega_p^w(R_{\pi_n}\Xi-R_{\pi_m}\Xi,h) \lesssim h^\nu \left[\int_0^{2^{1-m}h} \left(\frac{\overline{\Omega}_p^w(\delta\Xi,\tau)}{\tau^\nu}\right)^b \hspace{.2em} \frac{d\tau}{\tau}\right]^{1/b},
    $$
    as required.
\end{proof}

\begin{remark}[Weight Tax]
    The subdivision estimate in part (a) of the preceding lemma gives the dyadic coefficient $2^{\ell \nu_2}$. Comparing the parent and child weights separately and then
    applying the triangle or quasi-triangle inequality instead gives
    $$
    2^{\ell(1 \lor 1/p_2)} 2^{\ell(r-1)/p_2},
    $$
    corresponding to the potentially larger exponent $(1 \lor \frac{1}{p_2}) + \frac{r-1}{p_2}$. The joint estimate thus yields the coefficient $2^{-\ell(\gamma-\nu_2)}$ in subcritical sewing, leading to the sewing threshold $\gamma > \nu_2$. At $\gamma = \nu_2$, this geometric decay disappears, and summability must instead come from the scale integrability of the defect, as in critical sewing. In particular, when $p_2 \geq r$, the resulting threshold is $1$, just as in the unweighted Banach range.
\end{remark}

\subsection{Subcritical Weighted Sewing}
\begin{theorem}[Subcritical Weighted Sewing]
\label{thm:Subcritical Weighted Sewing}
    Let $V$ be a finite-dimensional normed vector space, and let $1 \leq r < \infin$ and $w \in A_r$. Suppose
    $$
    0 < p_1, p_2, q_1, q_2 \leq \infin, \qquad 0 < p_2 < \infin, \qquad 0 < \alpha < 1,
    $$
    and set
    $$
    \nu_2 = 1 \lor \frac{r}{p_2}, \qquad \gamma > \nu_2, \qquad \epsilon = \gamma - \nu_2.
    $$
    Write
    $$
    p_* = p_1 \land p_2, \qquad q_* = q_1 \lor q_2, \qquad W = w([0, T]).
    $$
    Let $\Xi: \bigtriangleup_2(0, T) \rarr V$ be measurable and satisfy
    $$
    \Xi \in \mathbb{B}^\alpha_{p_1, q_1}(w), \qquad \delta \Xi \in \overline{\mathbb{B}}^\gamma_{p_2, q_2}(w).
    $$
    Then there exists a measurable path $\mathcal{I} \Xi \in B^\alpha_{p_*, q_*}(w)$ along with a sewn two-parameter remainder path $R \Xi \in \mathbb{B}^\gamma_{p_2, q_2}(w)$ satisfying the following properties.
    \begin{enumerate}[label=(\alph*)]
        \item\textup{(Increment Identity)} For every $0 < h \leq T$,
        $$
        \Omega^w_{p_2}(\delta \mathcal{I} \Xi - \Xi - R\Xi, h) = 0.
        $$
        In particular, one may choose the representative
        $$
        (R \Xi)_{s, t} = (\mathcal{I} \Xi)_t - (\mathcal{I} \Xi)_s - \Xi_{s, t}.
        $$

        \item\textup{(Dyadic Convergence)} For every integer $n \geq 0$,
        $$
        \| I_{\pi_n} \Xi - \delta \mathcal{I} \Xi \|_{\mathbb{B}^\gamma_{p_2, q_2}(w)} = \|R_{\pi_n} \Xi - R \Xi \|_{\mathbb{B}^\gamma_{p_2, q_2}(w)} \lesssim 2^{-n \epsilon} \| \delta \Xi \|_{\overline{\mathbb{B}}^\gamma_{p_2, q_2}(w)}.
        $$

        \item\textup{(Arbitrary Convergence)} Let
        $$
        \pi = \{0 = \tau_0 < ... < \tau_N = 1\}.
        $$
        For every $q_2 \leq s \leq \infin$ and every $\pi$ with $\| \pi \| \leq 1/2$,
        $$
        \|I_\pi \Xi - \delta \mathcal{I} \Xi \|_{\mathbb{B}^\gamma_{p_2, s}(w)} = \| R_\pi \Xi - R \Xi \|_{\mathbb{B}^\gamma_{p_2, s}(w)} \lesssim \| \pi \|^\epsilon \| \delta \Xi \|_{\overline{\mathbb{B}}^\gamma_{p_2, q_2}(w)}.
        $$
        Thus convergence along arbitrary partition holds, in particular, in the remainder space $\mathbb{B}^\gamma_{p_2, q_2}(w)$.

        \item\textup{(Path Bounds)} The remainder satisfies
        $$
        \| R \Xi \|_{\mathbb{B}^\gamma_{p_2, q_2}(w)} \lesssim \| \delta \Xi \|_{\overline{\mathbb{B}}^\gamma_{p_2, q_2}(w)},
        $$
        and the sewn path satisfies
        $$
        [\mathcal{I} \Xi]_{B^\alpha_{p_*, q_*}(w)} \lesssim W^{1/p_* - 1/p_1} \| \Xi \|_{\mathbb{B}^\alpha_{p_1, q_1}(w)} + W^{1/p_* - 1/p_2} T^{\gamma - \alpha} \| \delta \Xi \|_{\overline{\mathbb{B}}^\gamma_{p_2, q_2}(w)}.
        $$
    \end{enumerate}
    The sewn increment $\delta \mathcal{I} \Xi$ is unique modulo zero-increment seminorms among measurable additive maps $A$ satisfying $A - \Xi \in \mathbb{B}^\gamma_{p_2, q_2}(w)$.
\end{theorem}
\begin{proof}
    Set
    $$
    p := p_2, \qquad q := q_2, \qquad \nu := 1 \lor \frac{r}{p}, \qquad \epsilon := \gamma-\nu, \qquad b := 1 \land p \land q,
    $$
    and write
    $$
    G(h) := h^{-\gamma}\overline{\Omega}_p^w(\delta\Xi,h), \qquad D := \|\delta\Xi\|_{\overline{\mathbb{B}}^\gamma_{p,q}(w)}.
    $$
    We construct the remainder in the weighted space and use deweighting together with \cite[Theorem 3.1]{FrizSeeger2022} to identify the limit as an additive increment.

    By Lemma~\ref{lem:Dyadic Bounds for Sewing}(b) and $b$-subadditivity,
    $$
    \left[h^{-\gamma}\Omega_p^w(R_{\pi_n}\Xi-R_{\pi_m}\Xi,h)\right]^b \lesssim \sum_{\ell=m}^{n-1} 2^{-\ell b\epsilon}G(2^{-\ell}h)^b.
    $$
    Taking the $L^{q/b}(dh/h)$-norm, with the supremum convention when $q=\infin$, gives
    \begin{align*}
        \|R_{\pi_n}\Xi-R_{\pi_m}\Xi\|_{\mathbb{B}^\gamma_{p,q}(w)}^b
        &\lesssim \sum_{\ell=m}^{n-1} 2^{-\ell b\epsilon}\|G(2^{-\ell}\cdot)\|_{L^q((0,T],dh/h)}^b
        \\
        &\leq D^b \sum_{\ell=m}^{n-1}2^{-\ell b\epsilon}
        \lesssim 2^{-mb\epsilon}D^b.
    \end{align*}
    Completeness therefore yields a limit $R\Xi\in\mathbb{B}^\gamma_{p,q}(w)$ satisfying
    $$
    \|R_{\pi_n}\Xi-R\Xi\|_{\mathbb{B}^\gamma_{p,q}(w)} \lesssim 2^{-n\epsilon}D.
    $$
    Since $R_{\pi_0}\Xi=0$, this also gives $\|R\Xi\|_{\mathbb{B}^\gamma_{p,q}(w)}\lesssim D$.

    By Theorem 2.10,
    $$
    \Xi\in\mathbb{B}^\alpha_{p_1/r,q_1}, \qquad \delta\Xi\in\overline{\mathbb{B}}^\gamma_{p/r,q},
    $$
    with $\infin/r=\infin$. Since $\gamma>1\lor r/p$, the unweighted sewing theorem and its dyadic convergence argument in \cite[Theorem 3.1]{FrizSeeger2022} provide a measurable path $F$ such that
    $$
    \|I_{\pi_n}\Xi-\delta F\|_{\mathbb{B}^\gamma_{p/r,q}} \rarr 0.
    $$
    Deweighting the convergence already proved shows that $\Xi+R\Xi$ is the same limit. Consequently,
    $$
    \Omega_{p/r}(\delta F-\Xi-R\Xi,h)=0, \qquad 0<h\leq T.
    $$
    Since $w_h$ is positive and finite, the weighted and unweighted slice measures have the same null sets. Hence
    $$
    \Omega_p^w(\delta F-\Xi-R\Xi,h)=0, \qquad 0<h\leq T.
    $$
    Set $\mathcal{I}\Xi:=F$ and choose $R\Xi=\delta F-\Xi$. This proves (a), and the preceding convergence estimate proves (b).

    We next prove the path bound. Write $P:=p_*$ and $Q:=q_*$. The finite-measure inequalities and $\int_0^{T-h}w_h(s)\hspace{.2em}ds\leq W$ give
    $$
    \omega_P^w(F,h) \lesssim W^{1/P-1/p_1}\Omega_{p_1}^w(\Xi,h)+W^{1/P-1/p}\Omega_p^w(R\Xi,h).
    $$
    For any nondecreasing $H$, $\beta>0$, and $0<u\leq v\leq\infin$, monotonicity on $[h,2h]$ gives
    $$
    \|h^{-\beta}H(h)\|_{L^v((0,T/2],dh/h)} \lesssim \|h^{-\beta}H(h)\|_{L^u((0,T],dh/h)}.
    $$
    Applying this to the preceding modulus bound, and using $\gamma>\alpha$, yields
    $$
    \|h^{-\alpha}\omega_P^w(F,h)\|_{L^Q((0,T/2],dh/h)} \lesssim W^{1/P-1/p_1}\|\Xi\|_{\mathbb{B}^\alpha_{p_1,q_1}(w)}+W^{1/P-1/p}T^{\gamma-\alpha}D.
    $$
    Splitting each increment into two equal children and applying the parent-child weight comparison gives $\omega_P^w(F,h)\lesssim\omega_P^w(F,h/2)$. Thus the same estimate holds over $(0,T]$, proving the seminorm bound in (d).

    For completeness, we verify the zeroth-order condition $F\in L^P(w)$. For an interval $I$ of length $L$, put $I_h:=\{s:s,s+h\in I\}$. Comparing pairs separated by at least $L/4$, and inserting a third point for closer pairs, gives the usual local oscillation estimate. Together with Lemma~\ref{lem: Shift and Scale Bounds for Muckenhoupt Weights}(i), this yields
    $$
    \frac{w(I)}{L}\inf_{c\in V}\int_I|F(t)-c|^P\hspace{.2em}dt \lesssim \frac{1}{L}\int_{L/4}^L\int_{I_h}|\delta F_{s,s+h}|^P w_h(s)\hspace{.2em}ds\hspace{.2em}dh.
    $$
    In particular, $F\in L^P(I)$ whenever $L\leq T/2$. Let $\mathcal{D}_j$ be the dyadic partition of $[0,T]$ into intervals of length $h_j=2^{-j}T$, and on each $I\in\mathcal{D}_j$ choose a constant $c_I$ minimizing the local $L^P$ error. Define $F_j=c_I$ on $I$. Comparing the constants on each parent and its children, and summing the preceding estimate, gives
    $$
    \|F_{j+1}-F_j\|_{L^P(w)}^P \lesssim \sum_{I\in\mathcal{D}_j}\frac{w(I)}{|I|}\int_I|F-c_I|^P \lesssim \omega_P^w(F,h_j)^P.
    $$
    The seminorm bound gives $\omega_P^w(F,h_j)\lesssim h_j^\alpha[F]_{B^\alpha_{P,Q}(w)}$, so these differences are summable to the power $1\land P$. Each $F_j$ belongs to $L^P(w)$, and Lebesgue differentiation gives $F_j\to F$ almost everywhere. Completeness of $L^P(w)$ therefore proves $F\in L^P(w)$.

    For (c), fix $q\leq v\leq\infin$ and a partition $\pi$ with $\|\pi\|\leq1/2$. Write
    $$
    \lambda_i:=\tau_i-\tau_{i-1}, \qquad \mathcal{J}_k:=\{i:2^{-k-1}<\lambda_i\leq2^{-k}\},
    $$
    and choose $K\geq1$ such that $2^{-K-1}<\|\pi\|\leq2^{-K}$. Let $B_k$ be the sum of the increments of $R\Xi$ over the children indexed by $\mathcal{J}_k$. Since $\#\mathcal{J}_k\leq2^{k+1}$, Lemma~\ref{lem:Dyadic Bounds for Sewing}(a) gives
    $$
    \Omega_p^w(B_k,h)\lesssim2^{k\nu}\Omega_p^w(R\Xi,2^{-k}h).
    $$
    Consequently, the scale embedding above and $2^{-k}T\leq T/2$ imply
    $$
    \|B_k\|_{\mathbb{B}^\gamma_{p,v}(w)} \lesssim 2^{-k\epsilon}\|h^{-\gamma}\Omega_p^w(R\Xi,h)\|_{L^v((0,2^{-k}T],dh/h)} \lesssim 2^{-k\epsilon}D.
    $$
    Additivity gives $\delta F-I_\pi\Xi=\sum_{k\geq K}B_k$. With $b_v:=1\land p\land v$, we obtain
    $$
    \|\delta F-I_\pi\Xi\|_{\mathbb{B}^\gamma_{p,v}(w)}^{b_v} \lesssim \sum_{k\geq K}2^{-kb_v\epsilon}D^{b_v} \lesssim \|\pi\|^{b_v\epsilon}D^{b_v}.
    $$
    Taking $b_v$th roots proves (c), including the asserted equality since $I_\pi\Xi-\delta F=R_\pi\Xi-R\Xi$.

    Finally, any other admissible additive map has the form $\delta\widetilde F$ modulo zero-increment seminorms by Proposition 2.15. Its difference from $\delta F$ belongs to $\mathbb{B}^\gamma_{p,q}(w)$. Since $\gamma>\nu$, Proposition 2.13 implies that $\widetilde F-F$ has a constant version. This proves uniqueness.
\end{proof}

\begin{remark}[Convergence along Arbitrary Partitions]
    Theorem~\ref{thm:Subcritical Weighted Sewing}(c) gives convergence in the original remainder space $\mathbb{B}^\gamma_{p_2, q_2}(w)$ for every $0 < q_2 \leq \infin$, with rate $\|\pi\|^{\gamma-\nu_2}$. Grouping child intervals of comparable length makes the summation geometric under the sole smoothness condition $\gamma > \nu_2$, including in the quasi-Banach range.
\end{remark}

\subsection{Critical Weighted Sewing}
At the boundary
$$
    \gamma_c:=\nu(p_2,r)=1\vee\frac r{p_2},
$$
the geometric decay used in subcritical sewing disappears.
The condition $0<q_2\le1\wedge p_2$ still gives summable dyadic
tails and a remainder in the little space
$\mathbb B^{\gamma_c}_{p_2,\infty;\circ}(w)$.
We use loss functions to describe its finite-scale-index
regularity.

For each $0<\sigma<\infty$, fix a positive, nonincreasing function
$\ell_\sigma:(0,T]\to(0,\infty)$ satisfying
$$
    \int_0^T\ell_\sigma(h)^{-\sigma}\frac{dh}{h}<\infty
$$
and, for every $\eta,\zeta>0$ and $0<\delta\le T$,
$$
    \int_0^\delta
        \ell_\sigma(h)^\eta h^\zeta\frac{dh}{h}
    \lesssim_{\sigma,\eta,\zeta}
    \ell_\sigma(\delta)^\eta\delta^\zeta.
$$
Set $\ell_\infty\equiv1$. A standard choice is
$$
    \ell_\sigma(h)
    =
    \left(1+\log\frac Th\right)^{1/\sigma+\epsilon},
    \qquad \epsilon>0,
    \qquad 0<\sigma<\infty.
$$
To make the modulus convention explicit, define
$$
    \omega_\sigma(h)
    :=
    \sup_{0<u\le h}u^{\gamma_c}\ell_\sigma(u).
$$
The preceding assumptions imply
$$
    \omega_\sigma(h)\asymp
    h^{\gamma_c}\ell_\sigma(h),
    \qquad
    \omega_\infty(h)=h^{\gamma_c}.
$$
We also write
$$
    \Lambda_\sigma
    :=
    \left(
        \int_0^T\ell_\sigma(h)^{-\sigma}\frac{dh}{h}
    \right)^{1/\sigma},
    \qquad
    \Lambda_\infty:=1.
$$

When $r>1$ and $r>p_2$, openness of the Muckenhoupt classes
makes this boundary subcritical after choosing a smaller
admissible weight index; the following corollary records the
resulting stronger conclusions. When $r\le p_2$, the boundary
is $\gamma_c=1$ and remains a genuine sewing endpoint.
The case $r=1$, $p_2<1$ also retains a genuine endpoint,
at $\gamma_c=1/p_2$.

\begin{theorem}[Critical Weighted Sewing]
\label{thm:Critical Weighted Sewing}
    Let $V$ be a finite-dimensional normed vector space, and let $1 \leq r < \infin$ and $w \in A_r$. Suppose
    $$
    0 < p_1, q_1 \leq \infin, \qquad 0 < p_2 < \infin, \qquad 0 < q_2 \leq 1 \land p_2, \qquad 0 < \alpha < 1.
    $$
    Set
    $$
    \gamma_c = 1 \lor \frac{r}{p_2}, \qquad p_* = p_1 \land p_2, \qquad q_* = q_1 \lor q_2, \qquad W = w([0, T]).
    $$
    Let $\Xi: \bigtriangleup_2(0, T) \rarr V$ be measurable and satisfy
    $$
    \Xi \in \mathbb{B}^\alpha_{p_1, q_1}(w), \qquad D_c(\Xi) := \| \delta \Xi \|_{\overline{\mathbb{B}}^{\gamma_c}_{p_2, q_2}(w)} + T^{-\gamma_c} \overline{\Omega}^w_{p_2}(\delta \Xi, T) < \infin.
    $$
    Then there exists a measurable path $\mathcal{I} \Xi \in B^\alpha_{p_*, q_*}(w)$ along with a sewn two-parameter remainder path $R \Xi$, with
    $$
    R \Xi \in \mathbb{B}^{\gamma_c}_{p_2, \infin; \circ}(w) \cap \bigcap_{0 < \sigma < \infin} \mathbb{B}^{\omega_\sigma}_{p_2, \sigma}(w),
    $$
    satisfying the following properties.
    \begin{enumerate}[label=(\alph*)]
        \item\textup{(Increment Identity)} For every $0 < h \leq T$,
        $$
        \Omega^w_{p_2}(\delta \mathcal{I} \Xi - \Xi - R\Xi, h) = 0.
        $$
        In particular, one may choose the representative
        $$
        (R \Xi)_{s, t} = (\mathcal{I} \Xi)_t - (\mathcal{I} \Xi)_s - \Xi_{s, t}.
        $$

        \item\textup{(Dyadic Convergence)} For every $0 < \sigma \leq \infin$,
        $$
        \lim_{n \rarr \infin} \|I_{\pi_n} \Xi - \delta \mathcal{I} \Xi \|_{\mathbb{B}^{\omega_\sigma}_{p_2, \sigma}(w)} = \lim_{n \rarr \infin} \| R_{\pi_n} \Xi - R \Xi \|_{\mathbb{B}^{\omega_\sigma}_{p_2, \sigma}(w)} = 0.
        $$

        \item\textup{(Arbitrary Convergence)} For every $0 < \sigma \leq \infin$,
        $$
        \lim_{\| \pi \| \rarr 0} \|I_\pi \Xi - \delta \mathcal{I} \Xi \|_{\mathbb{B}^{\omega_\sigma}_{p_2, \sigma}(w)} = \lim_{\| \pi \| \rarr 0} \| R_\pi \Xi - R \Xi \|_{\mathbb{B}^{\omega_\sigma}_{p_2, \sigma}(w)} = 0.
        $$
        In particular, taking $\sigma = \infin$ gives convergence in the space $\mathbb{B}^{\gamma_c}_{p_2, \infin}(w)$.

        \item\textup{(Path Bounds)} The remainder satisfies
        $$
        \| R \Xi \|_{\mathbb{B}^{\gamma_c}_{p_2, \infin}(w)} \lesssim D_c(\Xi),
        $$
        and, for every $0 < \sigma < \infin$,
        $$
        \|R \Xi \|_{\mathbb{B}^{\omega_\sigma}_{p_2, \sigma}(w)} \lesssim \Lambda_\sigma D_c(\Xi).
        $$
        The sewn path satisfies
        $$
        [\mathcal{I} \Xi]_{B^\alpha_{p_*, q_*}(w)} \lesssim W^{1/p_* - 1/p_1} \| \Xi \|_{\mathbb{B}^\alpha_{p_1, q_1}(w)} + W^{1/p_* - 1/p_2} T^{\gamma_c - \alpha} D_c(\Xi).
        $$
    \end{enumerate}
    The sewn increment $\delta \mathcal{I} \Xi$ is unique modulo zero-increment seminorms among measurable additive maps $A$ satisfying $A - \Xi \in \mathbb{B}^{\gamma_c}_{p_2, \infin; \circ}(w)$.
\end{theorem}
\begin{proof}
    Set $p = p_2$, $q = q_2$, $\nu = \gamma_c$, and $D = D_c(\Xi)$, and write
    $$
    G(h) = \overline{\Omega}^w_p(\delta \Xi, h), \qquad H(a) = \left( \int_0^a \left( \frac{G(\tau)}{\tau^\nu} \right)^q \frac{d\tau}{\tau} \right)^{1/q}.
    $$
    Then $H(T) \leq D$ and $H(a) \rarr 0$ as $a \rarr 0^+$. We follow the dyadic construction of \cite[proof of Theorem 3.2]{FrizSeeger2022}, using Lemma~\ref{lem:Dyadic Bounds for Sewing}.

    Since $q \leq 1 \land p$, Lemma~\ref{lem:Dyadic Bounds for Sewing}(c) gives, for $n > m \geq 1$,
    $$
    \Omega^w_p(R_{\pi_n}\Xi - R_{\pi_m}\Xi, h) \lesssim h^\nu H(2^{1-m}h).
    $$
    Consequently,
    $$
    \|R_{\pi_n}\Xi - R_{\pi_m}\Xi\|_{\mathbb{B}^\nu_{p,\infin}(w)} \lesssim H(2^{1-m}T) \rarr 0.
    $$
    Moreover, monotonicity of $G$ gives
    $$
    \sup_{0 < h \leq T} \frac{G(h)}{h^\nu} \lesssim H(T) + T^{-\nu}G(T) \leq D.
    $$
    Part (b) of the same lemma therefore gives $\|R_{\pi_1}\Xi\|_{\mathbb{B}^\nu_{p,\infin}(w)} \lesssim D$. Completeness yields a limit $R\Xi \in \mathbb{B}^\nu_{p,\infin}(w)$ with
    $$
    \|R\Xi\|_{\mathbb{B}^\nu_{p,\infin}(w)} \lesssim D.
    $$
    Taking $m = 0$ in part (c) and passing to the limit gives
    $$
    \frac{\Omega^w_p(R\Xi,h)}{h^\nu} \lesssim H(2h), \qquad 0 < h \leq T/2.
    $$
    Thus $R\Xi \in \mathbb{B}^\nu_{p,\infin;\circ}(w)$.

    For every two-parameter map $B$ and every $0 < \sigma \leq \infin$, the definitions give
    $$
    \|B\|_{\mathbb{B}^{\omega_\sigma}_{p,\sigma}(w)} \leq \Lambda_\sigma \|B\|_{\mathbb{B}^\nu_{p,\infin}(w)}.
    $$
    Applying this to $R\Xi$ and $R_{\pi_n}\Xi-R\Xi$ proves the asserted remainder bounds and convergence in all the stated spaces.

    We next identify $\Xi+R\Xi$ as an additive increment. If $r > 1$ and $p < r$, openness gives some $\rho \in (1,r)$ with $w \in A_\rho$, and $\nu > 1 \lor \rho/p$. Theorem~\ref{thm:Subcritical Weighted Sewing}, applied with weight index $\rho$, then supplies a measurable path $F$ whose increments are the dyadic limit. Otherwise, set $u = p/r$. We have $\nu = 1 \lor 1/u$ and $q \leq 1 \land u$, so Theorem 2.10 and \cite[Theorem 3.2]{FrizSeeger2022}, applied with integrability indices $p_1/r$ and $u$, again supply such a path $F$ in the unweighted setting. Deweighting the convergence established above identifies the same limit. In either case, uniqueness of the limit on each increment slice and positivity of $w_h$ give
    $$
    \Omega^w_p(\delta F-\Xi-R\Xi,h) = 0, \qquad 0 < h \leq T.
    $$
    Set $\mathcal{I}\Xi = F$ and choose $R\Xi = \delta F-\Xi$ pointwise. This proves (a) and (b).

    For arbitrary partitions, define
    $$
    \eta(a) = \sup_{0 < u \leq a} \frac{\Omega^w_p(R\Xi,u)}{u^\nu}.
    $$
    The little-space property gives $\eta(a) \rarr 0$ as $a \rarr 0^+$. Let $\lambda_i = \tau_i-\tau_{i-1}$. Additivity and Lemma~\ref{lem:Dyadic Bounds for Sewing}(a) yield
    $$
    \Omega^w_p(\delta F-I_\pi\Xi,h)^p = \Omega^w_p(I_\pi R\Xi,h)^p \leq [w]_{A_r} \sum_i \lambda_i^{1-p\nu} \Omega^w_p(R\Xi,\lambda_i h)^p \leq [w]_{A_r} h^{p\nu} \eta(\|\pi\|h)^p,
    $$
    where we used $\sum_i\lambda_i = 1$. Hence
    $$
    \|\delta F-I_\pi\Xi\|_{\mathbb{B}^\nu_{p,\infin}(w)} \leq [w]_{A_r}^{1/p}\eta(\|\pi\|T) \rarr 0.
    $$
    The preceding embedding into $\mathbb{B}^{\omega_\sigma}_{p,\sigma}(w)$ gives (c), since $R_\pi\Xi-R\Xi = I_\pi\Xi-\delta F$.

    Since $\nu > \alpha$, the remainder bound also gives
    $$
    \|R\Xi\|_{\mathbb{B}^\alpha_{p,q_*}(w)} \lesssim T^{\nu-\alpha}D.
    $$
    The spatial and scale embeddings used in the proof of Theorem~\ref{thm:Subcritical Weighted Sewing}, together with the same argument for membership in $L^{p_*}(w)$, therefore yield $F \in B^\alpha_{p_*,q_*}(w)$ and
    $$
    [F]_{B^\alpha_{p_*,q_*}(w)} \lesssim W^{1/p_*-1/p_1}\|\Xi\|_{\mathbb{B}^\alpha_{p_1,q_1}(w)} + W^{1/p_*-1/p}T^{\nu-\alpha}D.
    $$
    This completes (d).

    Finally, let $A$ be another measurable additive map satisfying the stated little-space condition. Proposition 2.15 gives $A = \delta \widetilde{F}$ modulo zero-increment seminorms, and
    $$
    \omega^w_p(\widetilde{F}-F,h) = o(h^\nu).
    $$
    Proposition 2.13 implies that $\widetilde{F}-F$ has a constant version, proving uniqueness.
\end{proof}

\begin{corollary}[Critical Sewing by Openness]
\label{cor: Critical Sewing by Opennesss}
    Let $1 < r < \infin$ and $w \in A_r$, and suppose
    $$
    0 < p_2 < r, \qquad 0 < p_1, q_1, q_2 \leq \infin, \qquad 0 < \alpha < 1.
    $$
    Set $\gamma_c = \frac{r}{p_2}$. Let $\Xi: \bigtriangleup_2(0, T) \rarr V$ be a measurable map with
    $$
    \Xi \in \mathbb{B}^\alpha_{p_1, q_1}(w), \qquad \delta \Xi \in \overline{\mathbb{B}}^{\gamma_c}_{p_2, q_2}(w).
    $$
    Then there exists an $r_- \in (1, r)$ with $w \in A_{r_-}$ and such that all the conclusions of Theorem~\ref{thm:Subcritical Weighted Sewing} hold with weight index $r_-$ and defect smoothness $\gamma_c$. The convergence exponent then becomes
    $$
    \epsilon_- = \gamma_c - \left( 1 \lor \frac{r_-}{p_2} \right) > 0,
    $$
    and, in particular, $R \Xi \in \mathbb{B}^{\gamma_c}_{p_2, q_2}(w)$ for arbitrary $0 < q_2 \leq \infin$.
\end{corollary}
\begin{proof}
    By Lemma~\ref{lem: Standard Muckenhoupt Weights Properties}(iii), there exists $r_- \in (1,r)$ such that $w \in A_{r_-}$. Since $p_2 < r$ and $r_- < r$, we have
    $$
    1 \lor \frac{r_-}{p_2} < \frac{r}{p_2} = \gamma_c.
    $$
    Theorem~\ref{thm:Subcritical Weighted Sewing} therefore applies with weight index $r_-$ and defect smoothness $\gamma_c$, giving all the asserted conclusions.
\end{proof}

\subsection{Integrability Regain}
It's important that the ``integral path'' increments $\delta \mathcal{I} \Xi$ belongs to the same space as the initial model $\Xi$ (this space being $\mathbb{B}^\alpha_{p_1, q_1}(w)$), since then one can close fixed point arguments in the appropriate weighted spaces. Part (d) of Theorems~\ref{thm:Subcritical Weighted Sewing} and \ref{thm:Critical Weighted Sewing} show that this is the case when $p_1 \leq p_2$ and $q_1 \geq q_2$. The case $p_1 > p_2$ is important however, since $\Xi$ is typically a tensor product of increments, and successive tensor products reduce integrability (this also typically yields $q_1 \geq q_2$, so this parameter doesn't pose a problem). We thus address this regime separately here, using the pointwise embeddings and interpolation estimates of $\S 3$.

\begin{theorem}[Integrability Regain]
\label{thm:Integrability Regain for Sewing}
    Let $V$ be a finite-dimensional vector space, $1 \leq r < \infin$, and $w \in A_r$. Suppose
    $$
    0 < p_2 < p_1 \leq \infin, \qquad 0 < q_1, q_2 \leq \infin, \qquad 0 < \alpha < 1,
    $$
    and let $\Xi: \bigtriangleup_2(0, T) \rarr V$ be measurable with $\Xi \in \mathbb{B}^\alpha_{p_1, q_1}(w)$.
    Set
    $$
    \eta = \frac{1}{p_2} - \frac{1}{p_1}, \qquad q_* = q_1 \lor q_2,
    $$
    and for $\beta \geq 0, \rho \geq 1$, set
    $$
    \Theta_{\beta, \rho} = \beta - \alpha - \rho \eta, \qquad Q_{w, \rho} = \left( \frac{T^\rho}{w([0, T])} \right)^{\eta}.
    $$
    For $s < u < t$ set
    $$
    a = (u - s) \land (t - u), \qquad b = (u - s) \lor (t - u).
    $$
    Denote by $\textup{C}_\beta$ the condition that for some $0 < \vartheta \leq 1/2$ and $M \geq 0$ we have
    $$
    |\delta \Xi_{s, u, t}| \leq M \frac{(a^\vartheta b^{1 - \vartheta})^{\beta}}{w(B_{a^\vartheta b^{1 - \vartheta}}(s))^{1/p_2}}.
    $$
    Suppose one of the following two regimes hold.
    \begin{enumerate}[label=(\roman*)]
        \item\textup{(Subcritical)} For some $\gamma > 1 \lor \frac{r}{p_2}$ we have
        $$
        \delta \Xi \in \overline{\mathbb{B}}^\gamma_{p_2, q_2}(w), \qquad C_\gamma, \qquad \Theta_{\gamma, r} > 0.
        $$
        
        \item\textup{(Critical)} For $\gamma_c = 1 \lor \frac{r}{p_2}$ we have
        $$
        \delta \Xi \in \overline{\mathbb{B}}^{\gamma_c}_{p_2, q_2}(w), \qquad C_{\gamma_c}
        $$
        along with the condition $0 < q_2 \leq 1 \land p_2$ and the existence of some $\rho \in [1, r]$ such that $w \in A_\rho$, $\gamma_c > \frac{\rho}{p_2}$, and $\Theta_{\gamma_c, \rho} > 0$.
    \end{enumerate}
    In either case, some sewing lemma applies; let $\mathcal{I} \Xi, R \Xi$ be the resulting sewn and remainder paths, and set $N_\beta = \| \delta \Xi \|_{\overline{\mathbb{B}}^\beta_{p_2, q_2}(w)} + M$. Then we have the following.
    \begin{enumerate}[label=(\alph*)]
        \item\textup{(Pointwise Remainder Estimate)} The remainder has a continuous representative, also denoted by $R \Xi$, with
        $$
        \sup_{0 \leq s < t \leq T} \frac{|R \Xi_{s, t}| w([s, t])^{1/p_2}}{|t - s|^\beta} \lesssim N_\beta.
        $$

        \item\textup{(Integrability Regain)} We have
        $$
        R \Xi \in \mathbb{B}^\alpha_{p_1, q_2}(w), \qquad \|R \Xi \|_{\mathbb{B}^\alpha_{p_1, q_2}(w)} \lesssim Q_{w, \rho} T^{\Theta_{\beta, \rho}} N_\beta.
        $$

        \item\textup{(Sewn Path Regularity)} We have
        $$
        \mathcal{I} \Xi \in B^\alpha_{p_1, q_*}(w), \qquad [\mathcal{I} \Xi]_{B^\alpha_{p_1, q_*}(w)} \lesssim \| \Xi \|_{\mathbb{B}^\alpha_{p_1, q_1}(w)} + Q_{w, \rho} T^{\Theta_{\beta, \rho}} N_\beta. 
        $$
    \end{enumerate}
    When sewing is obtained through Corollary~\ref{cor: Critical Sewing by Opennesss}, then we apply the conditions (i) with $r_-$ in place of $r$ and $\gamma = \gamma_c$; $q_2$ is then unrestricted.
\end{theorem}
\begin{proof}
    We follow \cite[proof of Theorem 3.3]{FrizSeeger2022}, using the weighted embedding and interpolation results of $\S 3$. Write $R = R\Xi$, $p = p_2$, $P = p_1$, and $\theta = p/P$, with $\theta = 0$ when $P = \infin$. In case (i), take $(\beta,\rho) = (\gamma,r)$; in case (ii), take $\beta = \gamma_c$ and $\rho$ as in the hypothesis. In both cases,
    $$
    \beta > \frac{\rho}{p}, \qquad \Theta_{\beta,\rho} = \beta-\alpha-\frac{\rho}{p}(1-\theta) > 0.
    $$

    First, the lower-mass estimate and $\textup{C}_\beta$ give
    $$
    |\delta\Xi_{s,u,t}| \lesssim M \left( \frac{T^\rho}{W} \right)^{1/p} (t-s)^{\beta-\rho/p}, \qquad W = w([0,T]).
    $$
    Together with Lemma~\ref{lem: Shift and Scale Bounds for Muckenhoupt Weights}(iii), this implies
    $$
    T^{-\beta}\overline{\Omega}^w_p(\delta\Xi,T) \lesssim M.
    $$
    Thus, in the critical case, $D_c(\Xi) \lesssim N_\beta$. Theorems~\ref{thm:Subcritical Weighted Sewing} and~\ref{thm:Critical Weighted Sewing} consequently give
    $$
    \|R\|_{\mathbb{B}^\beta_{p,q_2}(w)} \lesssim N_\beta \quad \textup{in case (i)}, \qquad \|R\|_{\mathbb{B}^\beta_{p,\infin}(w)} \lesssim N_\beta \quad \textup{in case (ii)}.
    $$

    Choose initially $R = \delta\mathcal{I}\Xi-\Xi$, so that $\delta R = -\delta\Xi$. Theorem 3.7, applied with weight index $\rho$, modulus $h^\beta$, and scale index $q_2$ in case (i) or $\infin$ in case (ii), supplies a continuous representative satisfying
    $$
    |R_{s,t}| \lesssim N_\beta \frac{|t-s|^\beta}{w([s,t])^{1/p}}.
    $$
    This proves (a). The lower-mass estimate then yields
    $$
    \Omega_\infin(R,h) \lesssim N_\beta \left( \frac{T^\rho}{W} \right)^{1/p} h^{\beta-\rho/p}.
    $$

    In case (i), apply Lemma 3.8 with $\delta = \beta-\rho/p$ and $\gamma = \beta$. Since
    $$
    \delta(1-\theta)+\beta\theta-\alpha = \Theta_{\beta,\rho}, \qquad \left( \frac{T^\rho}{W} \right)^{(1-\theta)/p} = Q_{w,\rho},
    $$
    we obtain
    $$
    \|R\|_{\mathbb{B}^\alpha_{P,q_2}(w)} \lesssim Q_{w,\rho} T^{\Theta_{\beta,\rho}} N_\beta.
    $$
    When $P = \infin$, the same estimate follows by integrating the preceding pointwise modulus bound directly.

    In case (ii), the endpoint remainder bound gives $\Omega^w_p(R,h) \lesssim N_\beta h^\beta$. Interpolating this with the bound on $\Omega_\infin(R,h)$ gives
    $$
    \Omega^w_P(R,h) \lesssim N_\beta Q_{w,\rho} h^{\beta-\rho(1/p-1/P)},
    $$
    again with the direct interpretation when $P = \infin$. Since $\Theta_{\beta,\rho} > 0$, integration against $dh/h$ yields
    $$
    \|R\|_{\mathbb{B}^\alpha_{P,q_2}(w)} \lesssim N_\beta Q_{w,\rho} \left( \int_0^T h^{q_2\Theta_{\beta,\rho}} \frac{dh}{h} \right)^{1/q_2} \lesssim Q_{w,\rho} T^{\Theta_{\beta,\rho}} N_\beta.
    $$
    This proves (b) in both regimes.

    Finally, $\delta\mathcal{I}\Xi = \Xi+R$ and the path-space argument in the proof of Theorem~\ref{thm:Subcritical Weighted Sewing} give $\mathcal{I}\Xi \in B^\alpha_{p_1,q_*}(w)$ and
    $$
    [\mathcal{I}\Xi]_{B^\alpha_{p_1,q_*}(w)} \lesssim \|\Xi\|_{\mathbb{B}^\alpha_{p_1,q_1}(w)} + \|R\|_{\mathbb{B}^\alpha_{p_1,q_2}(w)}.
    $$
    Applying (b) proves (c). The openness case follows from case (i) with $r_-$ in place of $r$ and $\gamma = \gamma_c$.
\end{proof}

\section{Young Integration}
In this section we use the sewing machinery of $\S 4$ to give meaning to integrals associated with a bounded bilinear map $B: V_0 \times V_1 \rarr V_2$ between finite-dimensional normed spaces. These integrals take the form
$$
\int^t_s B(f_u, dg_u),
$$
where $f \in B^{\alpha_0}_{p_0, q_0}(w; V_0)$ and $g \in B^{\alpha_1}_{p_1, q_1}(w; V_1)$. We assume $w \in A_r$ for some $1 \leq r < \infin$ and set
$$
\frac{1}{p_2} = \frac{1}{p_0} + \frac{1}{p_1},
$$
with the convention $1/\infin = 0$. The bilinear formulation includes scalar multiplication and the evaluation map $B(A,x) = Ax$ used in the differential equations below.

As in the standard theory of Young integration, the two-parameter local model we take is $\Xi_{s,t} = B(f_s,\delta g_{s,t})$. For sewing, one of our requirements is sufficient three-parameter weighted Besov regularity of the additive defect $\delta\Xi_{s,u,t}$. A simple calculation shows that
$$
\delta\Xi_{s,u,t} = -B(\delta f_{s,u},\delta g_{u,t}),
$$
so the problem can be rephrased as proving sufficient regularity of this bilinear product. At a fixed increment length $0 < h < T$, this requires a bound on
$$
\sup_{0 < \theta < 1} \|B(\delta f_{\cdot,\cdot+\theta h},\delta g_{\cdot+\theta h,\cdot+h})\|_{L^{p_2}(w_h)},
$$
where $L^p(w_h)$ denotes $L^p([0,T-h],w_h(s)\,ds)$.

In the unweighted case, Hölder's inequality yields the appropriate product estimate (see \cite[proof of Theorem 4.1]{FrizSeeger2022}). However, the nonuniformity in comparing a weight on a parent interval to the weights on two arbitrary subintervals necessitates an additional step. To see this, suppose $0 < \theta \leq 1/2$ and write $u = s+\theta h$, $t = s+h$. Thus $\delta f_{s,u}$ is the ``short'' increment and $\delta g_{u,t}$ is the ``long'' one. The generalised Hölder inequality yields
$$
\|B(\delta f_{\cdot,\cdot+\theta h},\delta g_{\cdot+\theta h,\cdot+h})\|_{L^{p_2}(w_h)} \leq \|B\| \|\delta f_{\cdot,\cdot+\theta h}\|_{L^{p_0}(w_h)} \|\delta g_{\cdot+\theta h,\cdot+h}\|_{L^{p_1}(w_h)}.
$$
Since the long increment has length between $h/2$ and $h$, the shifting and scaling bounds give
$$
\|\delta g_{\cdot+\theta h,\cdot+h}\|_{L^{p_1}(w_h)} \lesssim \omega^w_{p_1}(g,h).
$$
For the short increment, however, the parent-child estimate for $A_r$ weights gives $w_h(s) \lesssim \theta^{-(r-1)}w_{\theta h}(s)$. Thus
$$
\|\delta f_{\cdot,\cdot+\theta h}\|_{L^{p_0}(w_h)} \lesssim \theta^{-(r-1)/p_0}\omega^w_{p_0}(f,\theta h),
$$
and so
$$
\|B(\delta f_{\cdot,\cdot+\theta h},\delta g_{\cdot+\theta h,\cdot+h})\|_{L^{p_2}(w_h)} \lesssim \|B\| \theta^{-(r-1)/p_0}\omega^w_{p_0}(f,\theta h)\omega^w_{p_1}(g,h).
$$
The factor $\theta^{-(r-1)/p_0}$ may diverge as $\theta \rarr 0$, so this estimate does not yet give a bound uniform in $\theta$. However, under the condition $\alpha_0 > (r-1)/p_0$, the Besov decay of the shorter increment compensates for the singular factor introduced during the parent-child comparison. The case $1/2 \leq \theta < 1$ is analogous, with $g$ supplying the shorter increment and the corresponding condition $\alpha_1 > (r-1)/p_1$. Lemma~\ref{lem:Shorter Increment Besov Decay} makes this compensation precise, and Lemma~\ref{lem: Young Intermediate 2} applies it to obtain the defect regularity required for sewing.

\begin{lemma}
\label{lem:Shorter Increment Besov Decay}
    Let $\omega: (0, T] \rarr [0, \infin)$ be a nondecreasing map and let $\kappa \geq 0$. Set
    $$
    \omega^{\sharp, \kappa}(\tau) = \sup_{0 < \sigma \leq \tau} \left( \frac{\tau}{\sigma} \right)^{\kappa} \omega(\sigma).
    $$
    If $\alpha > \kappa$ and $0 < q \leq \infin$, then
    $$
    \| \tau^{-\alpha} \omega^{\sharp, \kappa}(\tau) \|_{L^q((0, T], d\tau / \tau)} \lesssim \| \tau^{-\alpha} \omega(\tau) \|_{L^q((0, T], d\tau / \tau)}
    $$
\end{lemma}
\begin{proof}
    Write
    $$
    F(\tau) = \tau^{-\alpha}\omega(\tau), \qquad G(\tau) = \tau^{-\alpha}\omega^{\sharp,\kappa}(\tau), \qquad d = \alpha-\kappa > 0.
    $$
    We may assume that $\|F\|_{L^q(d\tau/\tau)} < \infin$. Fix $0 < \sigma \leq \tau \leq T$ and choose $j \geq 0$ with $2^{-j-1}\tau < \sigma \leq 2^{-j}\tau$. Since $\omega$ is nondecreasing and $\kappa \geq 0$,
    $$
    \left(\frac{\tau}{\sigma}\right)^\kappa \omega(\sigma) \leq 2^{(j+1)\kappa}\omega(2^{-j}\tau).
    $$
    Taking the supremum over $\sigma$ and multiplying by $\tau^{-\alpha}$ gives
    $$
    G(\tau) \leq 2^\kappa \sup_{j \geq 0} 2^{-jd}F(2^{-j}\tau).
    $$

    If $0 < q < \infin$, then $(\sup_j a_j)^q \leq \sum_j a_j^q$ for nonnegative $a_j$. Tonelli's theorem and the change of variables $u = 2^{-j}\tau$ therefore give
    \begin{align*}
        \|G\|_{L^q(d\tau/\tau)}^q
        &\leq 2^{\kappa q}\sum_{j=0}^{\infin}2^{-jdq}\int_0^T F(2^{-j}\tau)^q \frac{d\tau}{\tau} \\
        &= 2^{\kappa q}\sum_{j=0}^{\infin}2^{-jdq}\int_0^{2^{-j}T} F(u)^q \frac{du}{u} \\
        &\leq \frac{2^{\kappa q}}{1-2^{-dq}}\|F\|_{L^q(d\tau/\tau)}^q,
    \end{align*}
    where the geometric series converges because $dq > 0$. Taking $q$th roots proves the claim.

    If $q = \infin$, taking essential suprema instead gives
    $$
    \|G\|_{L^\infin(d\tau/\tau)} \leq 2^\kappa \sup_{j \geq 0} 2^{-jd}\|F(2^{-j}\cdot)\|_{L^\infin((0,T],d\tau/\tau)} \leq 2^\kappa \|F\|_{L^\infin(d\tau/\tau)}.
    $$
\end{proof}

\begin{lemma}
\label{lem: Young Intermediate 2}
    Let $1 \leq r < \infin$ and $w \in A_r$, and let $B: V_0 \times V_1 \rarr V_2$ be a bounded bilinear map of finite-dimensional normed spaces. Let
    $$
    0 < p_0, p_1, q_0, q_1 \leq \infin, \qquad \frac{1}{p_2} = \frac{1}{p_0} + \frac{1}{p_1}, \qquad \frac{1}{q_2} = \frac{1}{q_0} + \frac{1}{q_1},
    $$ with $p_2 < \infin$. Suppose
    $$
    0 < \alpha_i < 1, \qquad \alpha_i > \frac{r-1}{p_i}, \qquad i = 0, 1,
    $$ where we set $\frac{1}{\infin} = 0$. For
    $$
    f \in B^{\alpha_0}_{p_0, q_0}(w; V_0), \qquad g \in B^{\alpha_1}_{p_1, q_1}(w; V_1)
    $$ we define
    $$
    C_{s, u, t} = B(\delta f_{s, u}, \delta g_{u, t}).
    $$
    Then
    $$
    \| C \|_{\overline{\mathbb{B}}^{\alpha_0 + \alpha_1}_{p_2, q_2}(w)} \lesssim \|B\| [f]_{B^{\alpha_0}_{p_0, q_0}(w)} [g]_{B^{\alpha_1}_{p_1, q_1}(w)}.
    $$
\end{lemma}
\begin{proof}
    Set
    $$
    \kappa_i = \frac{r-1}{p_i}, \qquad \omega_0(\tau) = \omega^w_{p_0}(f,\tau), \qquad \omega_1(\tau) = \omega^w_{p_1}(g,\tau),
    $$
    and define
    $$
    H_i(\tau) = \sup_{0 < \sigma \leq \tau}\left(\frac{\tau}{\sigma}\right)^{\kappa_i}\omega_i(\sigma), \qquad i = 0,1.
    $$
    Since $\alpha_i > \kappa_i$, Lemma~\ref{lem:Shorter Increment Besov Decay} gives
    $$
    \|\tau^{-\alpha_0}H_0(\tau)\|_{L^{q_0}(d\tau/\tau)} \lesssim [f]_{B^{\alpha_0}_{p_0,q_0}(w)}, \qquad \|\tau^{-\alpha_1}H_1(\tau)\|_{L^{q_1}(d\tau/\tau)} \lesssim [g]_{B^{\alpha_1}_{p_1,q_1}(w)}.
    $$

    Fix $0 < h < T$ and $0 < \theta < 1$, and set $a = \theta h$, $b = (1-\theta)h$. The parent-child comparisons of Lemma~\ref{lem: Shift and Scale Bounds for Muckenhoupt Weights}(i) give
    $$
    w_h(s) \leq [w]_{A_r}\theta^{-(r-1)}w_a(s), \qquad w_h(s) \leq [w]_{A_r}(1-\theta)^{-(r-1)}w_b(s+a).
    $$
    Writing $L^p(w_h) = L^p([0,T-h],w_h(s)\,ds)$ and changing variables in the second estimate, we obtain
    $$
    \|\delta f_{\cdot,\cdot+a}\|_{L^{p_0}(w_h)} \lesssim \theta^{-\kappa_0}\omega_0(a), \qquad \|\delta g_{\cdot+a,\cdot+h}\|_{L^{p_1}(w_h)} \lesssim (1-\theta)^{-\kappa_1}\omega_1(b).
    $$
    When either input exponent is infinite, the corresponding estimate follows directly by taking an essential supremum. Hölder's inequality therefore gives, in all cases,
    \begin{align*}
        \|C_{\cdot,\cdot+a,\cdot+h}\|_{L^{p_2}(w_h)}
        &\leq \|B\|\|\delta f_{\cdot,\cdot+a}\|_{L^{p_0}(w_h)}\|\delta g_{\cdot+a,\cdot+h}\|_{L^{p_1}(w_h)} \\
        &\lesssim \|B\|\theta^{-\kappa_0}\omega_0(\theta h)(1-\theta)^{-\kappa_1}\omega_1((1-\theta)h) \\
        &\leq \|B\|H_0(h)H_1(h).
    \end{align*}
    Since the $H_i$ are nondecreasing, taking suprema over $\theta$ and $h \leq \tau$ yields
    $$
    \overline{\Omega}^w_{p_2}(C,\tau) \lesssim \|B\|H_0(\tau)H_1(\tau).
    $$
    Finally, Hölder's inequality in the scale variable, with $1/q_2 = 1/q_0+1/q_1$, gives
    $$
    \|C\|_{\overline{\mathbb{B}}^{\alpha_0+\alpha_1}_{p_2,q_2}(w)} \lesssim \|B\|\|\tau^{-\alpha_0}H_0(\tau)\|_{L^{q_0}(d\tau/\tau)}\|\tau^{-\alpha_1}H_1(\tau)\|_{L^{q_1}(d\tau/\tau)} \lesssim \|B\|[f]_{B^{\alpha_0}_{p_0,q_0}(w)}[g]_{B^{\alpha_1}_{p_1,q_1}(w)}.
    $$
\end{proof}

\begin{remark}[Product and Sewing Conditions]
\label{rem:Young Conditions Imply Sewing Conditions}
    The hypotheses
    $$
        \alpha_i>\frac{r-1}{p_i},
        \qquad i=0,1,
    $$
    in Lemma~\ref{lem: Young Intermediate 2} ensure that the shorter-increment
    decay compensates for the parent--child weight comparison, uniformly
    over the splitting parameter.

    Writing
    $$
        \gamma:=\alpha_0+\alpha_1,
        \qquad
        \frac1{p_2}:=\frac1{p_0}+\frac1{p_1},
        \qquad
        \nu_2:=1\vee\frac r{p_2},
    $$
    subcritical sewing requires $\gamma>\nu_2$, while the boundary
    $\gamma=\nu_2$ is treated by the critical and openness results of
    Section~4. These sewing conditions do not, by themselves, imply
    the individual product hypotheses. For example,
    $$
        r=2,\qquad p_0=p_1=4,\qquad
        \alpha_0=\frac15,\qquad \alpha_1=\frac9{10}
    $$
    give $\gamma=11/10>\nu_2=1$, but
    $\alpha_0<(r-1)/p_0=1/4$.

    We therefore retain the product hypotheses explicitly in the
    weighted Young integral theorem below. They are automatically
    satisfied under the stronger continuity assumptions
    $\alpha_i>r/p_i$, $i=0,1$.
\end{remark}

\begin{theorem}[Weighted Young Integral]
\label{thm:Weighted Young Integral}
    Let $1 \leq r < \infin$ and $w \in A_r$, and let
    $$
    B: V_0 \times V_1 \rarr V_2
    $$
    be a bounded bilinear map of finite-dimensional normed vector spaces.
    Let
    $$
    0 < \alpha_0, \alpha_1 < 1, \qquad 0 < p_0, q_0, p_1, q_1 \leq \infin.
    $$
    Define
    $$
    \gamma = \alpha_0 + \alpha_1, \qquad \frac{1}{p_2} = \frac{1}{p_0} + \frac{1}{p_1}, \qquad \frac{1}{q_2} = \frac{1}{q_0} + \frac{1}{q_1}, \qquad \nu_2 = 1 \lor \frac{r}{p_2}.
    $$
    and assume $p_2 < \infin$. Also assume
    $$
    \alpha_i > \frac{r - 1}{p_i} \qquad \textup{ for } i = 0, 1.
    $$
    Let $f \in B^{\alpha_0}_{p_0, q_0}(w; V_0)$ and $g \in B^{\alpha_1}_{p_1, q_1}(w; V_1)$, and assume
    $$
    F_\bullet := \sup_{0 < h < T} \| f \|_{L^{p_0}([0, T - h], w_h)} < \infin.
    $$
    For concision we set
    $$
    F = [f]_{B^{\alpha_0}_{p_0, q_0}(w)}, \qquad G = [g]_{B^{\alpha_1}_{p_1, q_1}(w)}, \qquad W = w([0, T])
    $$
    For a partition $\pi = \{ 0 = \tau_0 < ... < \tau_N = 1 \}$ we define
    $$
    I^\pi_{s, t} = \sum_{j = 1}^{N} B(f_{s + \tau_{j - 1}(t - s)}, \delta g_{s + \tau_{j - 1}(t - s), s + \tau_j(t - s)}).
    $$
    \begin{enumerate}[label=(\alph*)]
        \item\textup{(Subcritical Sewing)} If $\gamma > \nu_2$ then there is a path $I \in B^{\alpha_1}_{p_2, q_1}(w; V_2)$ such that:
        \begin{enumerate}[label=(\roman*)]
            \item The remainder $R_{s, t} := \delta I_{s, t} - B(f_s, \delta g_{s, t})$ satisfies
            $$
            \|R\|_{\mathbb{B}^\gamma_{p_2, q_2}(w)} \lesssim \|B\| FG.
            $$

            \item We have the bound
            $$
            [I]_{B^{\alpha_1}_{p_2, q_1}(w)} \lesssim \|B\|(F_\bullet + T^{\alpha_0}F)G.
            $$

            \item For every $q_2 \leq \sigma \leq \infin$ and every partition $\pi$ with $\| \pi \| \leq 1/2$,
            $$
            \|I^\pi - \delta I\|_{\mathbb{B}^\gamma_{p_2, \sigma}(w)} \lesssim_\sigma \| \pi \|^{\gamma - \nu_2} \|B\| FG.
            $$
        \end{enumerate}

        \item\textup{(Critical Sewing)} If $\gamma = \gamma_c := \nu_2$ and $0 < q_2 \leq 1 \land p_2$, then there is a path $I \in B^{\alpha_1}_{p_2, q_1}(w; V_2)$ such that:
        \begin{enumerate}[label=(\roman*)]
            \item The remainder $R_{s, t} := \delta I - \Xi$ satisfies
            $$
            R \in \mathbb{B}^{\gamma_c}_{p_2, \infin; \circ}(w) \cap \bigcap_{0 < \sigma < \infin} \mathbb{B}^{\omega_\sigma}_{p_2, \sigma}(w), \qquad \|R\|_{\mathbb{B}^{\gamma_c}_{p_2, \infin}(w)} \lesssim \|B\|FG, \qquad \|R\|_{\mathbb{B}^{\omega_\sigma}_{p_2, \sigma}(w)} \lesssim \Lambda_\sigma \|B\|FG.
            $$

            \item We have the bound
            $$
            [I]_{B^{\alpha_1}_{p_2, q_1}(w)} \lesssim \|B\|(F_\bullet + T^{\alpha_0}F)G.
            $$
            
            \item For every $0 < \sigma \leq \infin$,
            $$
            \lim_{\| \pi \| \rarr 0} \|I^\pi - \delta I\|_{\mathbb{B}^{\omega_\sigma}_{p_2, \sigma}(w)} = 0.
            $$
            In particular, taking $\sigma = \infin$ gives convergence in $\mathbb{B}^{\gamma_c}_{p_2, \infin}(w)$.
        \end{enumerate}

        \item\textup{(Critical Sewing by Openness)} Suppose
        $$
        r > 1, \qquad p_2 < r, \qquad \gamma = \frac{r}{p_2},
        $$
        then there exists $r_- \in (1, r)$ such that $w \in A_{r_-}$. Set
        $$
        \epsilon_- := \gamma - \left(1 \lor \frac{r_-}{p_2} \right) > 0.
        $$
        All conclusions of (a) hold with weight index $r_-$ for arbitrary $0 < q_2 \leq \infin$, and with the updated convergence rate $\| \pi \|^{\epsilon_-}$.

        \item\textup{(Integrability Regain)} Suppose
        $$
        \alpha_i > \frac{r}{p_i} \hspace{1em} \textup{ for } i = 0, 1 \hspace{1em} \textup{ in } (a), (b), \qquad \textup{ and } \qquad \alpha_i > \frac{r_-}{p_i} \hspace{1em} \textup{ for } i = 0, 1 \hspace{1em} \textup{ in } (c).
        $$
        Then, using the continuous representatives of $f, g$, the path $I$ can be chosen with $I_0 = 0$, belongs to $B^{\alpha_1}_{p_1, q_1}(w; V_2)$, and
        $$
        [I]_{B^{\alpha_1}_{p_1, q_1}(w)} \lesssim \|B\| \left(|f_0| + \frac{T^{\alpha_0}}{W^{1/p_0}} F \right) G.
        $$
    \end{enumerate}
    In parts (a), (c) the sewn increment $\delta \mathcal{I} \Xi$ is unique modulo zero-increment seminorms among additive maps whose remainders belong to $\mathbb{B}^{\gamma}_{p_2, q_2}(w)$, and in part (b), uniqueness holds among additive maps whose remainders belong to $\mathbb{B}^{\gamma_c}_{p_2, \infin; \circ}(w)$.
\end{theorem}
\begin{proof}
    We follow \cite[proof of Theorem 4.1]{FrizSeeger2022}, using the weighted product estimate and sewing theorems. Set $\Xi_{s,t} = B(f_s,\delta g_{s,t})$. By bilinearity,
    $$
    \delta\Xi_{s,u,t} = -B(\delta f_{s,u},\delta g_{u,t}).
    $$
    For every $0 < h < T$, Hölder's inequality gives
    $$
    \|\Xi_{\cdot,\cdot+h}\|_{L^{p_2}(w_h)} \leq \|B\|\|f\|_{L^{p_0}([0,T-h],w_h)}\|\delta g_{\cdot,\cdot+h}\|_{L^{p_1}(w_h)} \leq \|B\|F_\bullet\omega^w_{p_1}(g,h).
    $$
    Consequently, Lemma~\ref{lem: Young Intermediate 2} yields
    $$
    \|\Xi\|_{\mathbb{B}^{\alpha_1}_{p_2,q_1}(w)} \lesssim \|B\|F_\bullet G, \qquad \|\delta\Xi\|_{\overline{\mathbb{B}}^\gamma_{p_2,q_2}(w)} \lesssim \|B\|FG.
    $$

    For (a), suppose $\gamma > \nu_2$. Since $q_2 \leq q_1$, Theorem~\ref{thm:Subcritical Weighted Sewing} gives $I \in B^{\alpha_1}_{p_2,q_1}(w)$ and $R = \delta I-\Xi$ satisfying
    $$
    \|R\|_{\mathbb{B}^\gamma_{p_2,q_2}(w)} \lesssim \|B\|FG, \qquad [I]_{B^{\alpha_1}_{p_2,q_1}(w)} \lesssim \|B\|(F_\bullet+T^{\alpha_0}F)G.
    $$
    Since $I^\pi = I_\pi\Xi$, its arbitrary-partition estimate gives (a)(iii).

    For (b), we verify the largest-scale bound required in $D_c(\Xi)$. Write $f_0 = f$ and $f_1 = g$ temporarily. Monotonicity gives $\omega^w_{p_i}(f_i,h) \lesssim h^{\alpha_i}[f_i]_{B^{\alpha_i}_{p_i,q_i}(w)}$ for $h \leq T/2$. Splitting increments in half and applying Lemma~\ref{lem: Shift and Scale Bounds for Muckenhoupt Weights}(i) gives
    $$
    \omega^w_{p_i}(f_i,h) \lesssim \omega^w_{p_i}(f_i,h/2),
    $$
    which extends the bound to all $h \leq T$. Thus
    $$
    \omega^w_{p_0}(f,h) \lesssim h^{\alpha_0}F, \qquad \omega^w_{p_1}(g,h) \lesssim h^{\alpha_1}G, \qquad 0 < h \leq T.
    $$
    Set $\kappa_i = (r-1)/p_i$. The weight comparisons in the proof of Lemma~\ref{lem: Young Intermediate 2} now give
    $$
    \|\delta\Xi_{\cdot,\cdot+\theta h,\cdot+h}\|_{L^{p_2}(w_h)} \lesssim \|B\|FG h^{\gamma_c}\theta^{\alpha_0-\kappa_0}(1-\theta)^{\alpha_1-\kappa_1} \leq \|B\|FG h^{\gamma_c}.
    $$
    Taking suprema yields $T^{-\gamma_c}\overline{\Omega}^w_{p_2}(\delta\Xi,T) \lesssim \|B\|FG$, and hence $D_c(\Xi) \lesssim \|B\|FG$. Theorem~\ref{thm:Critical Weighted Sewing}, again using $q_2 \leq q_1$, gives all the conclusions in (b).

    For (c), openness supplies $r_- \in (1,r)$ with $w \in A_{r_-}$. Since
    $$
    \gamma = \frac{r}{p_2} > 1 \lor \frac{r_-}{p_2},
    $$
    and the product hypotheses remain valid with $r_-$ in place of $r$, part (a) applies with this smaller weight index. The asserted uniqueness in (a)--(c) follows from the corresponding sewing theorems.

    For (d), set $\rho = r$ in (a), (b), and $\rho = r_-$ in (c). Then $w \in A_\rho$ and $\alpha_i > \rho/p_i$. Using the continuous representatives supplied by Theorem 3.3, the embedding and small-set estimates give
    $$
    \|f\|_\infin \lesssim |f_0|+\frac{T^{\alpha_0}}{W^{1/p_0}}F, \qquad \|\Xi\|_{\mathbb{B}^{\alpha_1}_{p_1,q_1}(w)} \leq \|B\|\|f\|_\infin G.
    $$
    If $p_0 = \infin$, then $p_2 = p_1$ and $F_\bullet \leq \|f\|_\infin$, so the preceding path estimates already give the desired bound. Suppose henceforth that $p_0 < \infin$, so $p_2 < p_1$.

    We verify the hypotheses of Theorem~\ref{thm:Integrability Regain for Sewing}. Set $\epsilon_i = \alpha_i-\rho/p_i > 0$. For $s < u < t$, write
    $$
    J = [s,t], \qquad h = t-s, \qquad a = (u-s)\land(t-u), \qquad b = (u-s)\lor(t-u).
    $$
    Choose $0 < \vartheta \leq 1/2$ with $\vartheta\gamma \leq \epsilon_0\land\epsilon_1$, and set $\ell = a^\vartheta b^{1-\vartheta}$. The localized embedding and small-set estimates give
    \begin{align*}
        |\delta\Xi_{s,u,t}|
        &\lesssim \|B\|FG\frac{(u-s)^{\alpha_0}(t-u)^{\alpha_1}}{w([s,u])^{1/p_0}w([u,t])^{1/p_1}} \\
        &\lesssim \|B\|FG\frac{h^{\rho/p_2}(u-s)^{\epsilon_0}(t-u)^{\epsilon_1}}{w(J)^{1/p_2}}
        \lesssim \|B\|FG\frac{\ell^\gamma}{w(J)^{1/p_2}},
    \end{align*}
    where the last inequality follows from $h \leq 2b$ and the choice of $\vartheta$. Since $\ell \leq h$, another weight comparison gives $w(B_\ell(s)) \leq w(B_h(s)) \lesssim w(J)$. Therefore
    $$
    |\delta\Xi_{s,u,t}| \lesssim \|B\|FG\frac{(a^\vartheta b^{1-\vartheta})^\gamma}{w(B_{a^\vartheta b^{1-\vartheta}}(s))^{1/p_2}},
    $$
    so $\textup{C}_\gamma$ holds with $M \lesssim \|B\|FG$.

    The remaining parameters satisfy
    $$
    \eta = \frac{1}{p_2}-\frac{1}{p_1} = \frac{1}{p_0}, \qquad \Theta_{\gamma,\rho} = \gamma-\alpha_1-\rho\eta = \alpha_0-\frac{\rho}{p_0} > 0.
    $$
    Also $\gamma > \rho/p_2$; in case (b), this forces $\gamma_c = 1 > \rho/p_2$, verifying the additional critical condition. Since $q_2 \leq q_1$ and $N_\gamma \lesssim \|B\|FG$, Theorem~\ref{thm:Integrability Regain for Sewing} gives
    \begin{align*}
        [I]_{B^{\alpha_1}_{p_1,q_1}(w)}
        &\lesssim \|\Xi\|_{\mathbb{B}^{\alpha_1}_{p_1,q_1}(w)}+\left(\frac{T^\rho}{W}\right)^{1/p_0}T^{\alpha_0-\rho/p_0}\|B\|FG \\
        &\lesssim \|B\|\left(|f_0|+\frac{T^{\alpha_0}}{W^{1/p_0}}F\right)G.
    \end{align*}
    Finally, $\alpha_1 > \rho/p_1$ supplies a continuous representative of $I$ by Theorem 3.3. Subtracting its value at zero ensures $I_0 = 0$ without changing its increments or the estimate.
\end{proof}

\begin{remark}[Critical Convergence and Integrability Regain]
    In part (a), convergence holds in the power spaces $\mathbb{B}^\gamma_{p_2,\sigma}(w)$ for $\sigma \geq q_2$. In part (b), convergence in $\mathbb{B}^{\gamma_c}_{p_2,\infin}(w)$ gives the finite-index conclusions through
    $$
    \|I^\pi-\delta I\|_{\mathbb{B}^{\omega_\sigma}_{p_2,\sigma}(w)} \leq \Lambda_\sigma \|I^\pi-\delta I\|_{\mathbb{B}^{\gamma_c}_{p_2,\infin}(w)}, \qquad 0 < \sigma < \infin.
    $$
    These finite-index remainder estimates use the loss modulus $\omega_\sigma$. The path estimates concern the lower smoothness $\alpha_1 < \gamma_c$. This positive gap restores finite scale integrability and gives the bound in (b)(ii) without a logarithmic factor. Under the additional assumptions in (d), pointwise control also permits a return to the integrability index $p_1$.
\end{remark}

We now consider differential equations
\begin{equation}
Y_t = \xi + \int^t_0 F(Y_s) \hspace{.2em} dX_s \label{(5.1)}
\end{equation}
where $X: [0, T] \rarr \R^d$, $F: \R^m \rarr \mathcal{L}(\R^d, \R^m)$, and this integral is interpreted in the weighted Young sense (as in Theorem~\ref{thm:Weighted Young Integral}). Specifically, we prove weighted analogues of \cite[Theorems 4.2, 4.3]{FrizSeeger2022}, which concern existence and uniqueness of solutions lying in the appropriate Besov space and local Lipschitz continuity of the Îto-Lyons map respectively.

\begin{theorem}[Existence and Uniqueness]
\label{thm: Young Existence and Uniqueness}
    Let $1 \leq r < \infin$ and $w \in A_r$, and suppose
    $$
    0 < p < \infin, \qquad 0 < q \leq \infin, \qquad 0 < \alpha < 1, \qquad 0 < \delta \leq 1,
    $$
    and let
    $$
    X \in B^\alpha_{p, q}(w; \R^d), \qquad F \in C^{1, \delta}_b(\R^m; \mathcal{L}(\R^d, \R^m)).
    $$
    Suppose
    $$
    \alpha > \frac{r}{p}, \qquad \delta \alpha > \frac{r}{p},
    $$
    and that one of the following regimes hold:
    $$
    \textup{(Subcritical)} \hspace{1em}  (1 + \delta)\alpha > 1, \qquad \textup{(Critical)} \hspace{1em} (1 + \delta)\alpha = 1, \quad 0 < q \leq 1 + \delta.
    $$
    Then for every $\xi \in \R^m$ there is a unique continuous path $Y \in B^\alpha_{p, q}(w; \R^m)$ such that
    $$
    Y_t = \xi + \int^t_0 F(Y_u) \hspace{.2em} dX_u
    $$
    for all $0 \leq t \leq T$. Moreover,
    $$
    [Y]_{B^\alpha_{p, q}(w)} \leq M
    $$
    where $M$ is some constant that depends only on the fixed initial parameters of the differential equation, not on the solution $Y$ we're consisdeirng.
\end{theorem}
\begin{proof}
    We follow the Picard argument of \cite[proof of Theorem 4.2]{FrizSeeger2022}. Set
    $$
    K = [X]_{B^\alpha_{p,q}(w)}, \qquad W = w([0,T]), \qquad \beta = \delta\alpha, \qquad L_F = \|DF\|_\infin, \qquad H_F = [DF]_{C^\delta}.
    $$
    Note that $p > 1$. For $J = [a,b] \subset [0,T]$, write
    $$
    L = b-a, \qquad W_J = w(J), \qquad [U]_J = [U]_{B^\alpha_{p,q}(w;J)}, \qquad \lambda_J = \frac{L^\alpha}{W_J^{1/p}}.
    $$
    The small-set estimate gives
    $$
    \lambda_J \leq [w]_{A_r}^{1/p}\left(\frac{T^r}{W}\right)^{1/p}L^{\alpha-r/p},
    $$
    so $\lambda_J \rarr 0$ uniformly in the position of $J$ as $L \rarr 0$. Theorem 3.3 also gives
    $$
    \|U-U_a\|_{\infin;J} \lesssim \lambda_J[U]_J.
    $$
    All constants below are uniform in $J$.

    Fix $\eta \in \R^m$ and define, using continuous representatives,
    $$
    \mathcal{X}_{J,\eta,R} = \{Y \in B^\alpha_{p,q}(w;J;\R^m) : Y_a = \eta,\ [Y]_J \leq R\}.
    $$
    Completeness of the Besov space and the preceding embedding show that this ball is complete for the metric $d_J(Y,Z) = [Y-Z]_J^{1\land q}$. Set
    $$
    (\mathcal{M}Y)_t = \eta+\int_a^t F(Y_s)\hspace{.2em}dX_s.
    $$

    We first show that $\mathcal{M}$ preserves the ball. Since $F$ is Lipschitz,
    $$
    [F(Y)]_J \leq L_F[Y]_J.
    $$
    Theorem~\ref{thm:Weighted Young Integral}(d) applies with both smoothness indices equal to $\alpha$: we have $2\alpha > 2r/p$, and $2\alpha > 1$ unless $\delta = 1$ and $\alpha = 1/2$, in which case $q \leq 2$. Hence
    $$
    [\mathcal{M}Y]_J \leq CK\left(\|F\|_\infin+L_F\lambda_J[Y]_J\right).
    $$
    Choose $R = 1+2CK\|F\|_\infin$. For sufficiently short $J$, uniformly in $\eta$, we have $CKL_F\lambda_J \leq 1/2$, and therefore $\mathcal{M}(\mathcal{X}_{J,\eta,R}) \subset \mathcal{X}_{J,\eta,R}$.

    We next prove contraction. Let $Y,Z \in \mathcal{X}_{J,\eta,R}$ and write $D = Y-Z$, $H = F(Y)-F(Z)$. The mean-value argument of \cite[Lemma 2.4]{FrizSeeger2022}, using weighted Hölder and Lemma~\ref{lem: Shift and Scale Bounds for Muckenhoupt Weights}(iii), gives
    $$
    [H]_{B^\beta_{p,q/\delta}(w;J)} \lesssim L_F L^{\alpha-\beta}[D]_J+H_F W_J^{(1-\delta)/p}\|D\|_{\infin;J}\left([Y]_J^\delta+[Z]_J^\delta\right).
    $$
    Here the first term uses the scale embedding from $B^\alpha_{p,q}$ to $B^\beta_{p,q/\delta}$; the second uses the finite-measure factor $W_J^{(1-\delta)/p}$. Since $D_a = 0$, the preceding pointwise embedding gives $\|D\|_{\infin;J} \lesssim \lambda_J[D]_J$.

    Apply Theorem~\ref{thm:Weighted Young Integral}(d) to $H$ and $X$, with smoothness indices $\beta,\alpha$ and scale indices $q/\delta,q$. Indeed, $\beta,\alpha > r/p$ and $\alpha+\beta > 2r/p$. At the critical boundary, $p/2 > r$ and the defect scale index is $q/(1+\delta) \leq 1$. Since $H_a = 0$, we obtain
    \begin{align*}
        [\mathcal{M}Y-\mathcal{M}Z]_J
        &\lesssim K\frac{L^\beta}{W_J^{1/p}}[H]_{B^\beta_{p,q/\delta}(w;J)} \\
        &\lesssim K\left(L_F\lambda_J+H_F R^\delta\lambda_J^{1+\delta}\right)[Y-Z]_J.
    \end{align*}
    Choose $L_0 > 0$ sufficiently small that the ball-preserving estimate holds and this last coefficient is at most $1/2$ whenever $|J| \leq L_0$. Banach's fixed-point theorem then supplies a local solution. Moreover, applying the ball-preserving estimate to any solution and absorbing its last term gives
    $$
    [Y]_J \leq 2CK\|F\|_\infin < R.
    $$
    Thus local uniqueness holds among all solutions in $B^\alpha_{p,q}(w;J)$.

    We continue these solutions using overlapping intervals. Choose an integer $N \geq 2$ such that $\ell = T/N$ satisfies $2\ell \leq L_0$, and put
    $$
    J_j = [j\ell,(j+2)\ell], \qquad j = 0,\ldots,N-2.
    $$
    Solve first on $J_0$ with initial value $\xi$, then successively on $J_j$ with initial value supplied by the preceding solution at $j\ell$. Local uniqueness makes the solutions agree on overlaps, giving a continuous path $Y$ on $[0,T]$ with $[Y]_{J_j} \leq R$.

    Every increment of length at most $\ell$ lies in some $J_j$, so
    $$
    \omega^w_p(Y,h) \leq \sum_{j=0}^{N-2}\omega^w_p(Y,h;J_j), \qquad 0 < h \leq \ell.
    $$
    This controls the small-scale part of the global Besov seminorm. For the remaining scales, Lemma~\ref{lem: Shift and Scale Bounds for Muckenhoupt Weights}(iii) gives $\omega^w_p(Y,h) \leq W^{1/p}\operatorname{osc}_{[0,T]}Y$. The local pointwise estimates also give
    $$
    \operatorname{osc}_{[0,T]}Y \lesssim NR\max_j\lambda_{J_j}.
    $$
    Consequently,
    $$
    [Y]_{B^\alpha_{p,q}(w)} \lesssim_N R+W^{1/p}\ell^{-\alpha}\operatorname{osc}_{[0,T]}Y \leq M,
    $$
    where $M$ depends only on the fixed data and is independent of $\xi$. Continuity also ensures $Y \in L^p(w)$.

    Locality and additivity of the Young integral now give
    $$
    Y_t = \xi+\int_0^t F(Y_s)\hspace{.2em}dX_s, \qquad 0 \leq t \leq T.
    $$
    Finally, any other solution belongs to the same local balls by the preceding a priori estimate, so successive applications of local uniqueness identify it with $Y$ on all of $[0,T]$.
\end{proof}

\begin{theorem}[Local Lipschitz Continuity of Itô-Lyons Map]
\label{thm: Local Lipschitz Continuity of Itô-Lyons Map}
    Assume the parameter and weight hypotheses of Theorem~\ref{thm: Young Existence and Uniqueness}, in either the subcritical or critical regime. For $i = 1, 2$ let
    $$
    \xi^i \in \R^m, \qquad X^i \in B^\alpha_{p, q}(w; \R^d), \qquad F^i \in C^{1, \delta}_{b}(\R^m; \mathcal{L}(\R^d, \R^m)),
    $$
    and let $Y^i$ be the corresponding solutions to
    $$
    Y^i = \xi^i + \int^t_0 F^i(Y^i_u) \hspace{.2em} dX^i_u.
    $$
    For every $M > 0$ there is a constant $C_M > 0$ such that whenever
    $$
    |\xi^i| + [X^i]_{B^\alpha_{p, q}(w)} + \|F^i\|_{C^{1, \delta}_{b}} \leq M,
    $$
    we have
    $$
    |Y^1_0 - Y^2_0| + [Y^1 - Y^2]_{B^\alpha_{p, q}(w)} \leq C_M \left(|\xi^1 - \xi^2| + [X^1 - X^2]_{B^\alpha_{p, q}(w)} + \|F^1 - F^2\|_{C^\delta_b} \right).
    $$
    Consequently the solution map $\Phi: (\xi, F, X) \mapsto Y$ is locally Lipschitz continuous as a map in
    $$
    \R^m \times C^{1, \delta}_{b}(\R^m; \mathcal{L}(\R^d, \R^m)) \times B^\alpha_{p, q}(w; \R^d) \rarr B^\alpha_{p, q}(w; \R^m).
    $$
\end{theorem}
\begin{proof}
    We follow \cite[proof of Theorem 4.3]{FrizSeeger2022}. Theorem~\ref{thm: Young Existence and Uniqueness} and its proof supply a uniform bound
    $$
    [Y^1]_{B^\alpha_{p,q}(w)}+[Y^2]_{B^\alpha_{p,q}(w)} \leq R_M.
    $$
    Set
    $$
    D = Y^1-Y^2, \qquad \Delta X = X^1-X^2, \qquad \Delta F = F^1-F^2, \qquad \beta = \delta\alpha,
    $$
    and write
    $$
    \mathcal{E} = [\Delta X]_{B^\alpha_{p,q}(w)}+\|\Delta F\|_{C^\delta_b}.
    $$
    For $J = [a,b]$, use the notation
    $$
    L = b-a, \qquad W_J = w(J), \qquad \lambda_J = \frac{L^\alpha}{W_J^{1/p}}, \qquad [U]_J = [U]_{B^\alpha_{p,q}(w;J)}, \qquad [U]_{\beta,J} = [U]_{B^\beta_{p,q/\delta}(w;J)}.
    $$
    As in the preceding proof,
    $$
    \lambda_J \lesssim \left(\frac{T^r}{w([0,T])}\right)^{1/p}L^{\alpha-r/p},
    $$
    uniformly in the position of $J$.

    By bilinearity of the Young integral, for $t \in J$,
    \begin{align*}
        D_t-D_a
        &= \int_a^t \Delta F(Y^1_u)\hspace{.2em}dX^1_u \\
        &\quad+\int_a^t \big(F^2(Y^1_u)-F^2(Y^2_u)\big)\hspace{.2em}dX^1_u
        +\int_a^t F^2(Y^2_u)\hspace{.2em}d\Delta X_u.
    \end{align*}
    Proposition 2.17 and the Lipschitz composition estimate give
    $$
    [\Delta F(Y^1)]_{\beta,J} \leq \|\Delta F\|_{C^\delta_b}W_J^{(1-\delta)/p}R_M^\delta, \qquad [F^2(Y^2)]_J \leq MR_M.
    $$
    Set $H = F^2(Y^1)-F^2(Y^2)$. The composition difference estimate from the preceding proof yields
    $$
    [H]_{\beta,J} \lesssim_M L^{\alpha-\beta}[D]_J+W_J^{(1-\delta)/p}R_M^\delta\|D\|_{\infin;J}.
    $$
    Moreover, Theorem 3.3 and the Lipschitz bound on $F^2$ give
    $$
    \|D\|_{\infin;J} \leq |D_a|+C\lambda_J[D]_J, \qquad |H_a| \leq M|D_a|.
    $$

    Apply Theorem~\ref{thm:Weighted Young Integral}(d) to the three integrals, using the smoothness pair $(\beta,\alpha)$ for the first two and $(\alpha,\alpha)$ for the third. The parameter checks are those in the preceding proof; in particular, at the critical boundary the first two integrals have defect scale index $q/(1+\delta) \leq 1$. Since
    $$
    \frac{L^\beta}{W_J^{1/p}}L^{\alpha-\beta} = \lambda_J, \qquad \frac{L^\beta}{W_J^{1/p}}W_J^{(1-\delta)/p} = \lambda_J^\delta,
    $$
    the preceding estimates and the quasi-triangle inequality give, whenever $\lambda_J \leq 1$,
    $$
    [D]_J \leq C_M\big(|D_a|+\mathcal{E}\big)+C_M\big(\lambda_J+\lambda_J^{1+\delta}\big)[D]_J.
    $$
    Choose $L_0 > 0$ sufficiently small that the last coefficient is at most $1/2$ and $\lambda_J \leq 1$ whenever $|J| \leq L_0$. Absorbing that term gives
    $$
    [D]_J \leq C_M\big(|D_a|+\mathcal{E}\big),
    $$
    and hence
    $$
    \|D\|_{\infin;J} \leq (1+C_M\lambda_J)|D_a|+C_M\lambda_J\mathcal{E}.
    $$

    Choose an equally spaced partition of $[0,T]$ with spacing $\ell$ satisfying $2\ell \leq L_0$. Iterating the last estimate over consecutive mesh intervals yields
    $$
    \|D\|_\infin \leq C_M\big(|D_0|+\mathcal{E}\big).
    $$
    Applying the local seminorm estimate on the overlapping intervals $J_j = [j\ell,(j+2)\ell]$ therefore gives
    $$
    [D]_{J_j} \leq C_M\big(|D_0|+\mathcal{E}\big).
    $$
    Every increment of length at most $\ell$ lies in one of these intervals. The global modulus argument from the preceding proof, with the larger scales controlled by $\omega^w_p(D,h) \leq 2W^{1/p}\|D\|_\infin$, consequently gives
    $$
    |D_0|+[D]_{B^\alpha_{p,q}(w)} \leq C_M\big(|D_0|+\mathcal{E}\big).
    $$
    Since $D_0 = \xi^1-\xi^2$, this is the asserted estimate.

    Finally, $\|D\|_{L^p(w)} \leq W^{1/p}\|D\|_\infin$ supplies the corresponding full Besov quasi-norm bound. Together with the continuous embedding $C_b^{1,\delta} \subset C_b^\delta$, this proves the stated local Lipschitz continuity.
\end{proof}

\begin{remark}[Roles of the Parameter Conditions]
    The condition $\alpha>r/p$ supplies continuous representatives
    and the pointwise estimates used in the differential-equation
    argument. The composition difference estimates are taken at
    the lower smoothness $\beta=\delta\alpha$. The condition
    $\delta\alpha>r/p$ then ensures that the integrability-regain gap
    $$
        (\alpha+\beta)-\alpha
        -r\left(\frac2p-\frac1p\right)
        =
        \delta\alpha-\frac rp
    $$
    is positive. This is a sufficient condition used by the present
    regain argument. Since $0<\delta\le1$, it also implies
    $\alpha>r/p$; we retain both inequalities to indicate their roles.

    Together, these assumptions give
    $$
        (1+\delta)\alpha>\frac{2r}{p}.
    $$
    Consequently, the remaining sewing condition is
    $$
        (1+\delta)\alpha>1,
    $$
    or, at the critical boundary,
    $$
        (1+\delta)\alpha=1,
        \qquad q\le1+\delta.
    $$
    At this boundary, the retained assumptions force $p>2r$,
    so the defect's sewing threshold is $1$. Openness of the
    Muckenhoupt classes cannot lower this threshold and therefore
    does not remove the critical restriction on $q$.
\end{remark}

\section{Weighted Besov Rough Paths}
In this section we develop rough integration and rough differential equations in the weighted Besov setting. Our organisation follows that of \cite[\S 5]{FrizSeeger2022}.

\subsection{Weighted Rough Paths and Pointwise Control}
For $N \in \N$ we give
$$
T^{(N)}(\R^d) = \bigoplus_{k = 0}^{N} (\R^d)^{\otimes k}, \qquad T_1^{(N)}(\R^d) = 1 + \bigoplus_{k = 1}^N (\R^d)^{\otimes k}
$$ the standard truncated tensor product. We set $\delta_\lambda$ to be the dilation map, so
$$
\delta_\lambda(1 + x^{(1)} + ... + x^{(N)}) = 1 + \lambda x^{(1)} + ... + \lambda^N x^{(N)},
$$ and we let $\tau_k$ denote truncation at level $k$.

\begin{definition}[Weighted Besov Rough Path]
    A level-$N$ weighted Besov rough path is a measurable map
    $$
    \bold{X} = (1, \bold{X}^{(1)}, ..., \bold{X}^{(N)}): \bigtriangleup_2(0, T) \rarr T_1^{(N)}(\R^d)
    $$ such that Chen's relation
    $$
    \bold{X}_{s, t} = \bold{X}_{s, u} \otimes \bold{X}_{u, t} \qquad \textup{ for all } 0 < s < u < t \leq T
    $$ holds and
    $$
    \bold{X}^{(k)} \in \mathbb{B}^{k \alpha}_{p/k, q/k}(w; (\R^d)^{\otimes k}) \qquad \textup{ for all } 1 \leq k \leq N.
    $$
    We denote by $\bold{B}^{\alpha, N}_{p, q}(w)$ the space of level-$N$ weighted Besov rough paths satisfying the preceding conditions, and we define
    $$
    \vertiii{\bold{X}}_{\bold{B}^{\alpha, N}_{p, q}(w)} = \sum_{k = 1}^{N} \| \bold{X}^{(k)} \|^{1/k}_{\mathbb{B}^{k \alpha}_{p/k, q/k}(w)}.
    $$
\end{definition}

\begin{proposition}[Pointwise Rough Path Embedding]
\label{prop:Pointwise Rough Path Embedding}
    Suppose $\alpha > \frac{r}{p}$ and let $\bold{X}, \widetilde{\bold{X}} \in \bold{B}^{\alpha, N}_{p, q}(w)$. Then both $\bold{X}, \widetilde{\bold{X}}$ have continuous multiplicative representatives and, for $1 \leq k \leq N$,
    $$
    |\bold{X}_{s, t}^{(k)}| \lesssim \vertiii{\tau_k \bold{X}}^k_{\bold{B}^{\alpha, k}_{p, q}(w)} \zeta_{k \alpha, p/k, w}(s, t - s).
    $$
    Moreover,
    \begin{align*}
        |\bold{X}^{(k)}_{s, t} - \widetilde{\bold{X}}^{(k)}_{s, t}|  \lesssim \zeta_{k \alpha, p/k, w}(s, t - s) \cdot \sum_{j = 1}^{k} \bigg( \vertiii{\tau_k \bold{X}}_{\bold{B}^{\alpha, k}_{p, q}(w)} &\lor \vertiii{\tau_k \widetilde{\bold{X}}}_{\bold{B}^{\alpha, k}_{p, q}(w)} \bigg)^{k - j} \\ &\cdot \| \bold{X}^{(j)} - \widetilde{\bold{X}^{(j)}}\|_{\mathbb{B}^{j \alpha}_{p/j, q/j}(w)}.
    \end{align*}
\end{proposition}
\begin{proof}
    We apply the argument of \cite[Proposition 5.1]{FrizSeeger2022} after local deweighting. Suppose first that $p < \infin$ and set $u = p/r$. Write $S_k = \vertiii{\tau_k \bold{X}}_{\bold{B}^{\alpha, k}_{p, q}(w)}$, $\widetilde{S}_k = \vertiii{\tau_k \widetilde{\bold{X}}}_{\bold{B}^{\alpha, k}_{p, q}(w)}$, $H_k = S_k \lor \widetilde{S}_k$, and $D_j = \|\bold{X}^{(j)} - \widetilde{\bold{X}}^{(j)}\|_{\mathbb{B}^{j\alpha}_{p/j, q/j}(w)}$.

    For every nondegenerate interval $J \subset [0, T]$, Theorem~2.10 gives
    $$
    \|A\|_{\mathbb{B}^{j\alpha}_{u/j, q/j}(J)} \leq c_J^j \|A\|_{\mathbb{B}^{j\alpha}_{p/j, q/j}(w; J)}, \qquad c_J = [w]_{A_r}^{1/p}\frac{|J|^{r/p}}{w(J)^{1/p}},
    $$
    for $1 \leq j \leq k$. Consequently, the unweighted homogeneous sizes of the restricted rough paths are bounded by $c_J S_k$ and $c_J \widetilde{S}_k$, and their level-$j$ difference seminorms are bounded by $c_J^j D_j$.

    The induction in the proof of \cite[Proposition 5.1]{FrizSeeger2022} uses only Chen's relation and $\alpha > 1/u$. Applied first on $[0, T]$, it supplies continuous multiplicative representatives. Applying the same estimates on $J = [s, t]$, with constants independent of $J$ by rescaling, then gives
    $$
    |\bold{X}^{(k)}_{s, t}| \lesssim c_J^k |t - s|^{k(\alpha - r/p)} S_k^k
    $$
    and
    $$
    |\bold{X}^{(k)}_{s, t} - \widetilde{\bold{X}}^{(k)}_{s, t}| \lesssim c_J^k |t - s|^{k(\alpha - r/p)} \sum_{j = 1}^{k} H_k^{k-j}D_j.
    $$
    The local continuous representatives agree with the global ones by uniqueness.

    Set $h = t - s$. Since $[s, t] \subset B_h(s)$ and $|B_h(s)| \leq 2h$, Lemma~2.3 gives $w(B_h(s)) \lesssim w([s, t])$. Hence
    $$
    c_J^k h^{k(\alpha - r/p)} \lesssim \frac{h^{k\alpha}}{w([s, t])^{k/p}} \lesssim \frac{h^{k\alpha}}{w(B_h(s))^{k/p}} = \zeta_{k\alpha, p/k, w}(s, h),
    $$
    proving both estimates. If $p = \infin$, the weight is irrelevant and the same unweighted argument applies directly.
\end{proof}

\begin{corollary}[Weighted Campanato Criterion]
    Let $\bold{X}$ satisfy Chen's relation and suppose that for constants $C_k$ we have
    $$
    \Omega^w_{p/k}(\bold{X}^{(k)}, h) \leq C_k h^{k \alpha} \qquad \textup{ for all } 0 < h \leq T, 1 \leq k \leq N.
    $$
    Then if $\alpha > \frac{r}{p}$, $\bold{X}$ has a continuous version and
    $$
    |\bold{X}^{(k)}_{s, t}| \lesssim \left( \max_{1 \leq j \leq k} C_j^{1/j} \right)^k \zeta_{k \alpha, p/k, w}(s, t - s).
    $$
\end{corollary}
\begin{proof}
    The hypothesis gives
    $$
    \|\bold{X}^{(k)}\|_{\mathbb{B}^{k\alpha}_{p/k, \infin}(w)} = \sup_{0 < h \leq T} h^{-k\alpha} \Omega^w_{p/k}(\bold{X}^{(k)}, h) \leq C_k.
    $$
    Thus $\bold{X} \in \bold{B}^{\alpha, N}_{p, \infin}(w)$. Proposition~\ref{prop:Pointwise Rough Path Embedding}, applied with $q = \infin$, supplies a continuous multiplicative representative and gives
    $$
    |\bold{X}^{(k)}_{s, t}| \lesssim \left(\sum_{j = 1}^{k} C_j^{1/j}\right)^k \zeta_{k\alpha, p/k, w}(s, t - s) \lesssim \left(\max_{1 \leq j \leq k} C_j^{1/j}\right)^k \zeta_{k\alpha, p/k, w}(s, t - s),
    $$
    as required.
\end{proof}

\begin{lemma}[Trading Regularity for Integrability]
\label{lem: Trading Regularity for Integrability}
    Let $N \geq 2$, $0 < p, q \leq \infin$, and $w \in A_r$. Suppose
    $$
    \alpha > \frac{r}{p}, \qquad \bold{X}, \widetilde{\bold{X}} \in \bold{B}^{\alpha, N}_{p, q}(w).
    $$
    Define
    $$
    \Lambda_{w, T} = \left( \frac{T^r}{w([0, T])} \right)^{1/p}
    $$ for $p < \infin$, with $\Lambda_{w, T} = 1$ if $p = \infin$,
    and for $1 \leq k \leq N$ write
    $$
    S_k(\bold{X}) = \vertiii{\tau_k \bold{X}}_{\bold{B}^{\alpha, k}_{p, q}(w)} = \sum_{i = 1}^{k} \| X^{(i)} \|_{\mathbb{B}^{i \alpha}_{p/i, q/i}(w)}^{1/i}, \qquad H_k = S_k(\bold{X}) \lor S_k(\widetilde{\bold{X}}).
    $$
    Then for every $1 \leq j < k \leq N$,
    $$
    \|X^{(k)}\|_{\mathbb{B}^{j \alpha}_{p/j, q/j}(w)} \leq C \Lambda_{w, T}^{k - j} T^{(k-j)(\alpha - r/p)} S_k(\bold{X})^k
    $$
    and
    $$
    \|X^{(k)} - \widetilde{X}^{(k)}\|_{\mathbb{B}^{j \alpha}_{p/j, q/j}(w)} \leq C \Lambda_{w, T}^{k - j} T^{(k-j)(\alpha - r/p)} \sum_{i = 1}^{k} H_k^{k - i} \|X^{(i)} - \widetilde{X}^{(i)}\|_{\mathbb{B}^{i \alpha}_{p/i, q/i}(w)}.
    $$
\end{lemma}
\begin{proof}
    We follow the interpolation argument of \cite[Lemma 5.1]{FrizSeeger2022}. Fix $1 \leq j < k \leq N$ and set $\theta = j/k$ and $\epsilon = \alpha - r/p > 0$. For $A = \bold{X}^{(k)}$, put $K = S_k(\bold{X})^k$; for $A = \bold{X}^{(k)} - \widetilde{\bold{X}}^{(k)}$, put
    $$
    K = \sum_{i = 1}^{k} H_k^{k-i} \|\bold{X}^{(i)} - \widetilde{\bold{X}}^{(i)}\|_{\mathbb{B}^{i\alpha}_{p/i, q/i}(w)}.
    $$
    In either case, Proposition~\ref{prop:Pointwise Rough Path Embedding} and the small-set bound
    $$
    w(B_h(s)) \geq [w]_{A_r}^{-1}\left(\frac{h}{T}\right)^r w([0, T])
    $$
    give
    $$
    \|A\|_{\mathbb{B}^{k\alpha}_{p/k, q/k}(w)} \leq K, \qquad \Omega_\infin(A, \tau) \lesssim \Lambda_{w, T}^k K \tau^{k\epsilon}.
    $$

    Interpolation against the same measure $w_h(s) \hspace{.2em} ds$, as in Lemma~3.8, yields
    $$
    \Omega^w_{p/j}(A, \tau) \leq \Omega_\infin(A, \tau)^{1-\theta} \Omega^w_{p/k}(A, \tau)^\theta.
    $$
    For $p = \infin$, this inequality is simply an identity. Consequently,
    $$
    \tau^{-j\alpha}\Omega^w_{p/j}(A, \tau) \lesssim \Lambda_{w, T}^{k-j} K^{1-\theta} \tau^{(k-j)\epsilon}\left(\tau^{-k\alpha}\Omega^w_{p/k}(A, \tau)\right)^\theta.
    $$
    Taking the $L^{q/j}((0, T], d\tau/\tau)$ quasi-norm and using $\theta(q/j) = q/k$, with the usual supremum convention when $q = \infin$, gives
    $$
    \|A\|_{\mathbb{B}^{j\alpha}_{p/j, q/j}(w)} \lesssim \Lambda_{w, T}^{k-j} T^{(k-j)\epsilon} K^{1-\theta} \|A\|_{\mathbb{B}^{k\alpha}_{p/k, q/k}(w)}^\theta \leq \Lambda_{w, T}^{k-j} T^{(k-j)\epsilon} K.
    $$
    The two choices of $A$ and $K$ give the asserted estimates.
\end{proof}

\subsection{A Weighted Lyons Extension Theorem}
\begin{theorem}[Weighted Lyons Extension]
\label{thm: Weighted Lyons Extension}
    Let $1 \leq M \leq N$ and $\alpha > \frac{r}{p}$, and suppose
    $$
    (M + 1) \alpha > 1.
    $$
    Then there is a unique continuous map
    $$
    \mathscr{E}_{M, N}: \bold{B}^{\alpha, M}_{p, q}(w) \rarr \bold{B}^{\alpha, N}_{p, q}(w)
    $$ satisfying the following given $\bold{X}, \widetilde{\bold{X}} \in \bold{B}^{\alpha, M}_{p, q}(w)$:
    \begin{enumerate}[label=(\roman*)]
        \item For all $1 \leq k \leq M$ we have $(\mathscr{E} \bold{X})^{(k)} = \bold{X}^{(k)}$.
        \item For all $M < k \leq N$ we have $\| (\mathscr{E} \bold{X})^{(k)} \|_{\mathbb{B}^{k \alpha}_{p/k, q/k}(w)} \lesssim \vertiii{\bold{X}}^k_{\bold{B}^{\alpha, M}_{p, q}(w)}$.
        \item For all $M < k \leq N$, we have
        \begin{align*}
            \| (\mathscr{E} \bold{X})^{(k)} - (\mathscr{E} \widetilde{\bold{X}})^{(k)} \|_{\mathbb{B}^{k \alpha}_{p/k, q/k}(w)} \lesssim \sum_{j = 1}^{M} &(\vertiii{\bold{X}}_{\bold{B}^{\alpha, M}_{p, q}(w)} \lor \vertiii{\widetilde{\bold{X}}}_{\bold{B}^{\alpha, M}_{p, q}(w)})^{k - j} \\ &\cdot \| \bold{X}^{(j)} - \widetilde{\bold{X}}^{(j)} \|_{\mathbb{B}^{j \alpha}_{p/j, q/j}(w)}.
        \end{align*}
    \end{enumerate}
    In particular, $\mathscr{E}$ is locally Lipschitz in the levelwise weighted Besov topology.
\end{theorem}
\begin{proof}
    As in \cite[Theorem 5.1]{FrizSeeger2022}, it suffices to construct one additional level and then iterate. Set $k = M + 1$ and $S = \vertiii{\bold{X}}_{\bold{B}^{\alpha, M}_{p, q}(w)}$. Using the continuous representative supplied by Proposition~\ref{prop:Pointwise Rough Path Embedding}, define
    $$
    A_{s, t} = \sum_{j = 1}^{M} \bold{X}^{(k-j)}_{0, s} \otimes \bold{X}^{(j)}_{s, t}.
    $$
    The coefficients $\bold{X}^{(k-j)}_{0, s}$ are bounded, and Lemma~\ref{lem: Trading Regularity for Integrability} gives $A \in \mathbb{B}^\alpha_{p, q}(w)$. Chen's relation gives the same cancellation as in the cited proof:
    $$
    \delta A_{s, u, t} = -\sum_{j = 1}^{M} \bold{X}^{(k-j)}_{s, u} \otimes \bold{X}^{(j)}_{u, t}.
    $$

    The proof of Lemma~5.2 applies equally to products of two-parameter maps and to smoothness exponents greater than one: it uses only their moduli and the parent-child weight comparisons. Its conditions hold here because $j\alpha > j(r-1)/p$ for every $j \geq 1$. Thus
    $$
    \|\delta A\|_{\overline{\mathbb{B}}^{k\alpha}_{p/k, q/k}(w)} \lesssim \sum_{j = 1}^{M} \|\bold{X}^{(k-j)}\|_{\mathbb{B}^{(k-j)\alpha}_{p/(k-j), q/(k-j)}(w)} \|\bold{X}^{(j)}\|_{\mathbb{B}^{j\alpha}_{p/j, q/j}(w)} \lesssim S^k.
    $$
    Since $k\alpha > 1 \lor kr/p$, Theorem~\ref{thm:Subcritical Weighted Sewing} applies. For $p = \infin$, use the unweighted product estimate and \cite[Theorem 3.1]{FrizSeeger2022} instead. Keeping the lower levels unchanged, set
    $$
    (\mathscr{E}\bold{X})^{(k)} = RA = \delta(\mathcal{I}A) - A.
    $$
    Then $\delta(RA) = -\delta A$, so the extended map satisfies Chen's relation, and
    $$
    \|(\mathscr{E}\bold{X})^{(k)}\|_{\mathbb{B}^{k\alpha}_{p/k, q/k}(w)} \lesssim S^k.
    $$
    Proposition~\ref{prop:Pointwise Rough Path Embedding} supplies its continuous multiplicative representative.

    For a second input $\widetilde{\bold{X}}$, construct $\widetilde{A}$ in the same way and set
    $$
    H = S \lor \vertiii{\widetilde{\bold{X}}}_{\bold{B}^{\alpha, M}_{p, q}(w)}, \qquad D_j = \|\bold{X}^{(j)} - \widetilde{\bold{X}}^{(j)}\|_{\mathbb{B}^{j\alpha}_{p/j, q/j}(w)}.
    $$
    Expanding each product difference in $\delta A - \delta\widetilde{A}$ into two terms and applying the same product estimate gives
    $$
    \|\delta A - \delta\widetilde{A}\|_{\overline{\mathbb{B}}^{k\alpha}_{p/k, q/k}(w)} \lesssim \sum_{j = 1}^{M} H^{k-j}D_j.
    $$
    Linearity of sewing therefore yields
    $$
    \|(\mathscr{E}\bold{X})^{(k)} - (\mathscr{E}\widetilde{\bold{X}})^{(k)}\|_{\mathbb{B}^{k\alpha}_{p/k, q/k}(w)} \lesssim \sum_{j = 1}^{M} H^{k-j}D_j.
    $$

    Finally, the difference of two level-$k$ extensions agreeing at all lower levels is additive and belongs to $\mathbb{B}^{k\alpha}_{p/k, q/k}(w)$. Since $k\alpha > 1 \lor kr/p$, Proposition~2.13 makes this difference zero. Iterating the construction proves uniqueness and the asserted bounds through level $N$. The difference estimate gives local Lipschitz continuity in the stated levelwise difference quantities.
\end{proof}

\begin{remark}[Critical Extension]
    Set $L = M + 1$ and suppose
    $$
    \alpha > \frac{r}{p}, \qquad L\alpha = 1, \qquad 0 < q \leq L.
    $$
    Then $Lr/p < 1$, so the sewing threshold at level $L$ is $1$, and $q/L \leq 1$. The weighted product estimates and Theorem~\ref{thm:Critical Weighted Sewing} therefore allow the one-level extension construction to be carried out at this boundary. For $p = \infin$, the weight is irrelevant and we use \cite[Theorem 3.2 and Remark 3.3]{FrizSeeger2022}.

    Every $\bold{X} \in \bold{B}^{\alpha, M}_{p, q}(w)$ consequently admits a continuous multiplicative extension through level $L$. With the loss functions of $\S 4$, set
    $$
    \omega^{(L)}_\sigma(h) = \sup_{0 < u \leq h} u\ell_\sigma(u), \qquad 0 < \sigma \leq \infin,
    $$
    where $\ell_\infin = 1$. The new component satisfies
    $$
    \bold{X}^{(L)} \in \mathbb{B}^1_{p/L, \infin; \circ}(w) \cap \bigcap_{0 < \sigma < \infin} \mathbb{B}^{\omega^{(L)}_\sigma}_{p/L, \sigma}(w).
    $$
    Writing $S = \vertiii{\bold{X}}_{\bold{B}^{\alpha, M}_{p, q}(w)}$, we have
    $$
    \|\bold{X}^{(L)}\|_{\mathbb{B}^1_{p/L, \infin}(w)} \lesssim S^L, \qquad \|\bold{X}^{(L)}\|_{\mathbb{B}^{\omega^{(L)}_\sigma}_{p/L, \sigma}(w)} \lesssim \Lambda_\sigma S^L
    $$
    for every $0 < \sigma < \infin$. Continuity follows from the two-parameter embedding, since $1 > Lr/p$. The extension is unique among those whose level-$L$ component belongs to $\mathbb{B}^1_{p/L, \infin; \circ}(w)$: the difference of two such components is additive, so Proposition~2.13 applies.

    These conclusions do not establish $\bold{X}^{(L)} \in \mathbb{B}^1_{p/L, q/L}(w)$ and hence do not give the preceding extension theorem with its original target space; compare \cite[Remark 5.4]{FrizSeeger2022}. Openness does not resolve this endpoint distinction: for every smaller admissible weight index $r_-$,
    $$
    1 \lor \frac{Lr_-}{p} = 1 = L\alpha.
    $$
    Thus lowering the weight index neither makes this boundary subcritical nor removes the restriction $q \leq L$ in the critical sewing argument.
\end{remark}

\subsection{Level Two Controlled Rough Paths and Rough Integration}
We set $\frac{r}{p} < \alpha \leq \frac{1}{2}$ and take $\bold{X} = (1, X, \mathbb{X}) \in \bold{B}^{\alpha, 2}_{p, q}(w)$.

\begin{definition}
    Let $E$ be a finite-dimensional vector space. We say a pair
    $$(Y, Y'): [0, T] \rarr E \times \mathcal{L}(\R^d, E)$$
    is controlled by $\bold{X}$ if
    $$
    Y \in B^\alpha_{p, q}(w; E), \qquad Y' \in B^\alpha_{p, q}(w; \mathcal{L}(\R^d, E)), \qquad R^Y_{s, t} = \delta Y_{s, t} - Y'_s X_{s, t} \in \mathbb{B}^{2\alpha}_{p/2 ,q/2}(w; E).
    $$
    In this case we write $(Y, Y') \in \mathscr{D}^\alpha_{p, q; \bold{X}}(w)$ and set
    $$
    \| (Y, Y') \|_{\mathscr{D}_{\bold{X}}} = |Y_0| + |Y'_0| + [Y']_{B^\alpha_{p, q}(w)} + \| R^Y \|_{\mathbb{B}^{2\alpha}_{p/2, q/2}(w)}.
    $$
    For paths controlled by different drivers $\bold{X}, \widetilde{\bold{X}}$ we define
    $$
    d_{\bold{X}, \widetilde{\bold{X}}}((Y, Y'), (\widetilde{Y}, \widetilde{Y}')) = |Y_0 - \widetilde{Y}_0| + |Y'_0 - \widetilde{Y}'_0| + [Y' - \widetilde{Y}']_{B^\alpha_{p, q}(w)} + \|R^Y - R^{\widetilde{Y}}\|_{\mathbb{B}^{2\alpha}_{p/2, q/2}(w)}.
    $$
\end{definition}

We use the continuous representatives supplied by $\alpha > r/p$. For a fixed driver $\bold{X}$, the displayed size is a norm when $p, q \geq 2$ and a quasi-norm otherwise. The controlled-path space is complete: setting $\theta = 1 \land (p/2) \land (q/2)$, a complete metric inducing the same topology is obtained by raising each summand in $d_{\bold{X}, \bold{X}}$ to the power $\theta$ and summing. Completeness follows from Besov completeness and Lemma~\ref{lem: Controlled-Remainder Estimate}, which controls $Y$ through its controlled decomposition. For different drivers, $d_{\bold{X}, \widetilde{\bold{X}}}$ is a comparison quantity rather than a distance on the space of driver--path pairs. Throughout, the Besov difference quantities in its definition are the unpowered difference seminorms, and all subsequent Lipschitz estimates use this convention.

The next lemma recovers the Besov regularity of $Y$ from its controlled decomposition and gives pointwise control of the remainder. Allowing $\beta < 2\alpha$ will also be useful in the stability estimates.

\begin{lemma}[Controlled-Remainder Estimates]
\label{lem: Controlled-Remainder Estimate}
    Let $(Y, Y') \in \mathscr{D}^\alpha_{p, q; \bold{X}}(w; E)$ and suppose
    $$
    \alpha + \frac{r}{p} < \beta \leq 2\alpha.
    $$
    Set
    $$
    Q = [Y']_{B^\alpha_{p, q}(w)} [X]_{B^\alpha_{p, q}(w)}, \qquad E_\beta = \| R^Y \|_{\mathbb{B}^\beta_{p/2, q/2}(w)}.
    $$
    Then
    $$
    \sup_{0 \leq s < t \leq T} \frac{|R^Y_{s, t}|}{\zeta_{\beta, p/2, w}(s, t - s)} \leq C(E_\beta + T^{2\alpha - \beta} Q),
    $$
    and
    $$
    [Y]_{B^\alpha_{p, q}(w)} \leq C \left[ |Y'_0| [X]_{B^\alpha_{p, q}(w)} + \Lambda_{w, T}\left( T^{\alpha - r/p}Q + T^{\beta - \alpha - r/p}E_\beta \right) \right].
    $$
    For another controlled pair $(\widetilde{Y}, \widetilde{Y}') \in \mathscr{D}^\alpha_{p, q; \widetilde{\bold{X}}}(w; E)$ we define
    $$
    E_{\beta, \Delta} = \|R^Y - R^{\widetilde{Y}}\|_{\mathbb{B}^\beta_{p/2, q/2}(w)}
    $$ and
    $$
    Q_\Delta = [Y' - \widetilde{Y}']_{B^\alpha_{p, q}(w)} [X]_{B^\alpha_{p, q}(w)} + [\widetilde{Y}']_{B^\alpha_{p, q}(w)}[X - \widetilde{X}]_{B^\alpha_{p, q}(w)}.
    $$
    Then
    $$
    \sup_{0 \leq s < t \leq T} \frac{|R^Y_{s, t} - R^{\widetilde{Y}}_{s, t}|}{\zeta_{\beta, p/2, w}(s, t - s)} \leq C(E_{\beta, \Delta} + T^{2\alpha - \beta} Q_\Delta)
    $$ and
    $$
    [Y - \widetilde{Y}]_{B^\alpha_{p, q}(w)} \leq C \left[ |Y'_0 - \widetilde{Y}'_0| [X]_{B^\alpha_{p, q}(w)} + |\widetilde{Y}'_0| [X - \widetilde{X}]_{B^\alpha_{p, q}(w)} + \Lambda_{w, T} \left( T^{\alpha - r/p} Q_\Delta + T^{\beta - \alpha - r/p} E_{\beta, \Delta} \right) \right].
    $$
\end{lemma}
\begin{proof}
    We follow \cite[proof of Lemma 5.2]{FrizSeeger2022}. Set $K_\beta = E_\beta + T^{2\alpha-\beta}Q$. The controlled decomposition gives
    $$
    \delta R^Y_{s, u, t} = \delta Y'_{s, u}X_{u, t}.
    $$
    Suppose first that $p < \infin$ and put $v = p/r$. On any interval $J$, local deweighting and the one-parameter embedding give
    $$
    \|R^Y\|_{\mathbb{B}^\beta_{v/2, q/2}(J)} \lesssim c_J^2 E_\beta, \qquad |\delta R^Y_{s, u, t}| \lesssim c_J^2 Q (u-s)^{\alpha-r/p}(t-u)^{\alpha-r/p},
    $$
    where $c_J \lesssim |J|^{r/p}w(J)^{-1/p}$. Since $2r/p < \beta \leq 2\alpha$, the latter bound implies
    $$
    |\delta R^Y_{s, u, t}| \lesssim c_J^2 Q |J|^{2\alpha-\beta}\bigl((u-s)(t-u)\bigr)^{(\beta-2r/p)/2}.
    $$
    Applying \cite[Proposition 2.7]{FrizSeeger2022} with splitting parameter $1/2$ yields
    $$
    |R^Y_{s, t}| \lesssim c_J^2\bigl(E_\beta + |J|^{2\alpha-\beta}Q\bigr)|t-s|^{\beta-2r/p}.
    $$
    Taking $J = [s, t]$ and using Lemma~2.3 to compare $w([s,t])$ with $w(B_{t-s}(s))$ proves
    $$
    |R^Y_{s, t}| \lesssim K_\beta \zeta_{\beta, p/2, w}(s, t-s).
    $$
    For $p = \infin$, the same conclusion follows directly from the unweighted embeddings.

    The lower mass bound now gives $\Omega_\infin(R^Y, h) \lesssim \Lambda_{w,T}^2 K_\beta h^{\beta-2r/p}$. Interpolating against $w_h(s) \hspace{.2em} ds$, we obtain
    $$
    h^{-\alpha}\Omega^w_p(R^Y, h) \lesssim \Lambda_{w,T}K_\beta^{1/2}h^{\beta-\alpha-r/p}\bigl(h^{-\beta}\Omega^w_{p/2}(R^Y, h)\bigr)^{1/2}.
    $$
    Taking the $L^q(dh/h)$ quasi-norm, with the usual supremum conventions, gives
    $$
    \|R^Y\|_{\mathbb{B}^\alpha_{p,q}(w)} \lesssim \Lambda_{w,T}T^{\beta-\alpha-r/p}K_\beta^{1/2}E_\beta^{1/2} \lesssim \Lambda_{w,T}\bigl(T^{\beta-\alpha-r/p}E_\beta + T^{\alpha-r/p}Q\bigr).
    $$
    Together with
    $$
    \|Y'\|_\infin \lesssim |Y'_0| + \Lambda_{w,T}T^{\alpha-r/p}[Y']_{B^\alpha_{p,q}(w)}
    $$
    and $\delta Y = Y'X + R^Y$, this proves the estimate on $Y$.

    For the differences, write $\Delta R = R^Y - R^{\widetilde{Y}}$ and similarly for the paths. The identity
    $$
    \delta\Delta R_{s,u,t} = \delta\Delta Y'_{s,u}X_{u,t} + \delta\widetilde{Y}'_{s,u}\Delta X_{u,t}
    $$
    allows the same embedding and interpolation argument with $E_{\beta,\Delta}$ and $Q_\Delta$ in place of $E_\beta$ and $Q$. Finally, use
    $$
    \delta\Delta Y_{s,t} = \Delta Y'_s X_{s,t} + \widetilde{Y}'_s\Delta X_{s,t} + \Delta R_{s,t}
    $$
    and the one-parameter embeddings for $\Delta Y'$ and $\widetilde{Y}'$ to obtain the asserted path difference estimate.
\end{proof}

At $\alpha = 1/3$, set $\omega^{(3)}_\sigma(h) = \sup_{0 < u \leq h} u\ell_\sigma(u)$ for $0 < \sigma \leq \infin$, where the loss functions are as in $\S 4$ and $\ell_\infin = 1$.

\begin{theorem}[Weighted Rough Integration]
\label{thm: Weighted Rough Integration}
    Let $1 \leq r < \infin$, $w \in A_r$, and $0 < p, q \leq \infin$. Suppose
    $$
    \frac{r}{p} < \alpha \leq \frac{1}{2},
    $$
    with $1/\infin = 0$.
    Let $V = \R^d$, let $E$ be a finite-dimensional vector space, and take
    $$
    \bold{X} = (1, X, \mathbb{X}) \in \bold{B}^{\alpha, 2}_{p, q}(w), \qquad (Y, Y') \in \mathscr{D}^\alpha_{p, q; \bold{X}}(w; \mathcal{L}(V, E)).
    $$
    Assume either
    $$
    \textup{(Subcritical)} \hspace{1em} \alpha > \frac{1}{3}, \qquad \textup{ or } \qquad \textup{(Critical)} \hspace{1em} \alpha = \frac{1}{3} \hspace{1em} \textup{ and } \hspace{1em} 0 < q \leq 3.
    $$
    Define
    $$
    \Xi_{s, t} = Y_s X_{s, t} + Y'_s \mathbb{X}_{s, t}
    $$
    and
    $$
    \mathcal{N}_\Xi = \|R^Y\|_{\mathbb{B}^{2\alpha}_{p/2, q/2}(w)} [X]_{B^\alpha_{p, q}(w)} + [Y']_{B^\alpha_{p, q}(w)} \|\mathbb{X}\|_{\mathbb{B}^{2\alpha}_{p/2, q/2}(w)} + [Y']_{B^\alpha_{p, q}(w)} [X]_{B^\alpha_{p, q}(w)}^2.
    $$
    Write
    $$
    \nu_3 = 1 \lor \frac{3r}{p}, \qquad \Lambda_{w, T} = \left( \frac{T^r}{w([0, T])} \right)^{1/p}.
    $$
    Then for every $z \in E$ there is a unique continuous path
    $$
    Z_t = z + \int^t_0 Y_u \hspace{.2em} d\bold{X}_u
    $$
    whose remainder $R^\Xi = \delta Z - \Xi$ has the following properties.

    \begin{enumerate}[label=(\alph*)]
        \item In the subcritical case,
        $$
        R^\Xi \in \mathbb{B}^{3\alpha}_{p/3, q/3}(w), \qquad \|R^\Xi\|_{\mathbb{B}^{3\alpha}_{p/3, q/3}(w)} \lesssim C \mathcal{N}_\Xi.
        $$

        \item In the critical case,
        $$
        R^\Xi \in \mathbb{B}^1_{p/3, \infin; \circ}(w) \cap \bigcap_{0 < \sigma < \infin} \mathbb{B}^{\omega_{\sigma}^{(3)}}_{p/3, \sigma}(w)
        $$
        where
        $$
        \|R^\Xi\|_{\mathbb{B}^1_{p/3, \infin}(w)} \lesssim \mathcal{N}_\Xi
        $$
        and, for every $0 < \sigma < \infin$,
        $$
        \| R^\Xi \|_{\mathbb{B}^{\omega_{\sigma}^{(3)}}_{p/3, \sigma}(w)} \lesssim \Lambda_\sigma \mathcal{N}_\Xi.
        $$
    \end{enumerate}

    For a partition $\pi = \{0 = u_0 < ... < u_N = 1\}$ we set
    $$
    t_i = s + u_i(t - s), \qquad I^\pi \Xi_{s, t} = \sum_{i = 0}^{N - 1} \Xi_{t_i, t_{i + 1}}.
    $$

    \begin{enumerate}[label=(\alph*), start=3]
        \item In the subcritical case, for every $q/3 \leq v \leq \infin$ and every partition $\pi$ with $\|\pi\| \leq 1/2$,
        $$
        \|I^\pi \Xi - \delta Z\|_{\mathbb{B}^{3\alpha}_{p/3, v}(w)} \lesssim_v \|\pi\|^{3\alpha - \nu_3} \mathcal{N}_{\Xi}.
        $$

        \item In the critical case, for every $0 < \sigma \leq \infin$,
        $$
        \lim_{\| \pi \| \rarr 0} \|I^\pi \Xi - \delta Z \|_{\mathbb{B}^{\omega_{\sigma}^{(3)}}_{p/3, \sigma}(w)} = 0.
        $$
        In particular,
        $$
        \lim_{\| \pi \| \rarr 0} \|I^\pi \Xi - \delta Z \|_{\mathbb{B}^{1}_{p/3, \infin}(w)} = 0.
        $$

        \item In both cases,
        $$
        (Z, Y) \in \mathscr{D}^\alpha_{p, q; \bold{X}}(w), \qquad R^Z = Y' \mathbb{X} + R^\Xi,
        $$ and
        $$
        \|R^Z\|_{\mathbb{B}^{2\alpha}_{p/2, q/2}(w)} \lesssim \left[ \left( |Y'_0| + \Lambda_{w, T} T^{\alpha - r/p} [Y']_{B^\alpha_{p, q}(w)} \right) \|\mathbb{X}\|_{\mathbb{B}^{2\alpha}_{p/2, q/2}(w)} + \Lambda_{w, T} T^{\alpha - r/p} \mathcal{N}_\Xi \right].
        $$
    \end{enumerate}
\end{theorem}
\begin{proof}
    We follow \cite[Theorem 5.4]{FrizSeeger2022}. Boundedness of $Y, Y'$ and Lemma~\ref{lem: Trading Regularity for Integrability} give $\Xi \in \mathbb{B}^\alpha_{p,q}(w)$. Identifying $Y'_s(v_1 \otimes v_2)$ with $(Y'_s v_1)v_2$, Chen's relation yields
    $$
    \delta\Xi_{s,u,t} = -R^Y_{s,u}X_{u,t} - \delta Y'_{s,u}\mathbb{X}_{u,t}.
    $$
    The two-parameter version of Lemma~5.2 used in the proof of Theorem~\ref{thm: Weighted Lyons Extension} therefore gives
    $$
    \|\delta\Xi\|_{\overline{\mathbb{B}}^{3\alpha}_{p/3,q/3}(w)} \lesssim \|R^Y\|_{\mathbb{B}^{2\alpha}_{p/2,q/2}(w)}[X]_{B^\alpha_{p,q}(w)} + [Y']_{B^\alpha_{p,q}(w)}\|\mathbb{X}\|_{\mathbb{B}^{2\alpha}_{p/2,q/2}(w)} \leq \mathcal{N}_\Xi.
    $$

    Set $\eta = \alpha-r/p > 0$. On an interval $J$, the local pointwise embeddings and Lemma~\ref{lem: Controlled-Remainder Estimate} with $\beta = 2\alpha$ give
    $$
    |\delta\Xi_{s,u,t}| \lesssim c_J^3\mathcal{N}_\Xi a^\eta b^{2\eta}, \qquad c_J = \frac{|J|^{r/p}}{w(J)^{1/p}},
    $$
    where $a = (u-s)\land(t-u)$ and $b = (u-s)\lor(t-u)$, with $c_J = 1$ when $p = \infin$. In particular, taking $J = [s,t]$ gives
    $$
    |\delta\Xi_{s,u,t}| \lesssim \mathcal{N}_\Xi\frac{|t-s|^{3\alpha}}{w([s,t])^{3/p}}.
    $$
    This also controls the largest-scale term required for critical sewing: for $h \leq T/2$, use monotonicity and the preceding Besov defect bound; for $h \geq T/2$, integrate this pointwise estimate against $w_h(s) \hspace{.2em} ds$. Thus, at $\alpha = 1/3$,
    $$
    D_c(\Xi) = \|\delta\Xi\|_{\overline{\mathbb{B}}^1_{p/3,q/3}(w)} + T^{-1}\overline{\Omega}^w_{p/3}(\delta\Xi,T) \lesssim \mathcal{N}_\Xi.
    $$

    In the subcritical case, $3\alpha > 1 \lor 3r/p$. At criticality, $p > 3r$ and $q/3 \leq 1$, so the critical sewing hypotheses hold. Theorems~\ref{thm:Subcritical Weighted Sewing} and~\ref{thm:Critical Weighted Sewing} supply a sewn increment $\delta Z = \Xi + R^\Xi$ and give (a)--(d) for finite $p$. When $p = \infin$, use \cite[Theorems 3.1, 3.2 and Remark 3.3]{FrizSeeger2022}; the partition estimates are addressed below.

    We next obtain pointwise control of $R^\Xi$. Locally deweight its bound in (a), or its power bound at scale index $\infin$ in (b). Since $\delta R^\Xi = -\delta\Xi$, the preceding local defect estimate permits application of \cite[Proposition 2.7]{FrizSeeger2022} with splitting parameter $1/3$. Taking $J = [s,t]$ yields a continuous representative satisfying
    $$
    |R^\Xi_{s,t}| \lesssim \mathcal{N}_\Xi\frac{|t-s|^{3\alpha}}{w([s,t])^{3/p}}, \qquad \Omega_\infin(R^\Xi,h) \lesssim \Lambda_{w,T}^3\mathcal{N}_\Xi h^{3\eta}.
    $$
    Since $\Xi$ is continuous, the sewn increment consequently determines a continuous path $Z$ with prescribed initial value $z$.

    For completeness, when $p = \infin$, (a) and this pointwise bound give $\|R^\Xi\|_{\mathbb{B}^{3\alpha}_{\infin,v}} \lesssim_v \mathcal{N}_\Xi$ for every $v \geq q/3$. In the identity
    $$
    I^\pi\Xi_{s,t} - \delta Z_{s,t} = -\sum_i R^\Xi_{t_i,t_{i+1}},
    $$
    group the partition lengths into dyadic classes and sum the resulting geometric series with exponent $3\alpha-1$, using the $(1\land v)$-triangle inequality. This proves (c), including $v < 1$. At criticality, convergence in $\mathbb{B}^1_{\infin,\infin}$ implies convergence in every stated loss space because
    $$
    \|A\|_{\mathbb{B}^{\omega^{(3)}_\sigma}_{\infin,\sigma}} \leq \Lambda_\sigma\|A\|_{\mathbb{B}^1_{\infin,\infin}}, \qquad 0 < \sigma < \infin.
    $$
    Thus (d) holds as well.

    For finite $p$, interpolation against $w_h(s) \hspace{.2em} ds$ gives
    $$
    h^{-2\alpha}\Omega^w_{p/2}(R^\Xi,h) \lesssim \Lambda_{w,T}\mathcal{N}_\Xi^{1/3}h^\eta\bigl(h^{-3\alpha}\Omega^w_{p/3}(R^\Xi,h)\bigr)^{2/3}.
    $$
    In the subcritical case, take the $L^{q/2}(dh/h)$ quasi-norm and use (a). At criticality, use the power bound in (b) and integrate $h^\eta$. Both give
    $$
    \|R^\Xi\|_{\mathbb{B}^{2\alpha}_{p/2,q/2}(w)} \lesssim \Lambda_{w,T}T^\eta\mathcal{N}_\Xi.
    $$
    For $p = \infin$, integrate the pointwise estimate directly. The same argument with target integrability $p$ gives $R^\Xi \in \mathbb{B}^\alpha_{p,q}(w)$, so $\delta Z = \Xi + R^\Xi$ implies $Z \in B^\alpha_{p,q}(w)$.

    Finally,
    $$
    R^Z_{s,t} = Y'_s\mathbb{X}_{s,t} + R^\Xi_{s,t}, \qquad \|Y'\|_\infin \lesssim |Y'_0| + \Lambda_{w,T}T^\eta[Y']_{B^\alpha_{p,q}(w)}.
    $$
    These bounds prove (e). The difference of two candidate sewn increments is additive and belongs to the subcritical remainder space, or to $\mathbb{B}^1_{p/3,\infin;\circ}(w)$ at criticality. Proposition~2.13 and the prescribed initial value give uniqueness.
\end{proof}

The following local estimate supplies the small factor needed for the fixed-point argument: since $\beta > \alpha + r/p$, the contribution of the controlled-remainder differences tends to zero with the interval length.

\begin{theorem}[Stability of Weighted Rough Integration]
\label{thm: Stability of Weighted Rough Integration}
    Fix $T_0 > 0$, let $1 \leq r < \infin$, and let $w \in A_r$. Suppose
    $$
    0 < p, q \leq \infin, \qquad \frac{r}{p} < \alpha \leq \frac{1}{2}, \qquad \alpha + \frac{r}{p} < \beta \leq 2\alpha,
    $$
    with $1/\infin = 0$. Assume either
    $$
    \textup{(Subcritical)} \hspace{1em} \alpha + \beta > 1, \qquad \textup{ or } \qquad \textup{(Critical)} \hspace{1em} \alpha + \beta = 1 \hspace{1em} \textup{ and } \hspace{1em} 0 < q \leq 3.
    $$
    Let $V = \R^d$ and let $E$ be a finite-dimensional normed vector space, and take
    $$
    \bold{X}, \widetilde{\bold{X}} \in \bold{B}^{\alpha, 2}_{p, q}(w), \qquad (Y, Y') \in \mathscr{D}^\alpha_{p, q; \bold{X}}(w; \mathcal{L}(V, E)), \qquad (\widetilde{Y}, \widetilde{Y}') \in \mathscr{D}^\alpha_{p, q; \widetilde{\bold{X}}}(w; \mathcal{L}(V, E)),
    $$
    where all spaces here are assumed to be over $[0, T_0]$.
    For $M > 0$ suppose all homogeneous sizes of our drivers and controlled paths are $\leq M$.
    Set $J = [a, b] \subset [0, T_0]$ and set $\tau = b - a$.
    For $z, \tilde{z} \in E$ define
    $$
    Z_t = z + \int^t_a Y_u \hspace{.2em} d\bold{X}_u, \qquad \widetilde{Z}_t = \tilde{z} + \int^t_a \widetilde{Y}_u \hspace{.2em} d\widetilde{\bold{X}}_u.
    $$
    Then
    \begin{align*}
        d^J_{\bold{X}, \widetilde{\bold{X}}}((Z, Y), (\widetilde{Z}, \widetilde{Y})) \lesssim_M |z - \tilde{z}| + &|Y_a - \widetilde{Y}_a| + |Y'_a - \widetilde{Y}'_a| + \varrho_{J}(\bold{X}, \widetilde{\bold{X}}) \\&+ \tau^{\beta - \alpha - r/p} \left( [Y' - \widetilde{Y}']_{B^\alpha_{p, q}(w; J)} + \|R^Y - R^{\widetilde{Y}}\|_{\mathbb{B}^\beta_{p/2, q/2}(w; J)} \right)
    \end{align*} where
    $$
    \varrho_{J}(\bold{X}, \widetilde{\bold{X}}) = [X - \widetilde{X}]_{B^\alpha_{p, q}(w; J)} + \|\mathbb{X} - \widetilde{\mathbb{X}}\|_{\mathbb{B}^{2\alpha}_{p/2, q/2}(w)}.
    $$
\end{theorem}
\begin{proof}
    We follow \cite[Theorem 5.5]{FrizSeeger2022}. Write $\Delta$ for differences and set
    $$
    \gamma = \alpha+\beta, \qquad \eta = \alpha-r/p, \qquad \epsilon = \beta-\alpha-r/p,
    $$
    so $0 < \epsilon \leq \eta$, and put
    $$
    D_J = [\Delta Y']_{B^\alpha_{p,q}(w;J)} + \|\Delta R^Y\|_{\mathbb{B}^\beta_{p/2,q/2}(w;J)}, \qquad K_J = \varrho_J(\bold{X},\widetilde{\bold{X}}) + D_J.
    $$
    The small-set condition gives
    $$
    \Lambda_{w,J} := \left(\frac{\tau^r}{w(J)}\right)^{1/p} \lesssim \Lambda_{w,T_0},
    $$
    with both quantities interpreted as $1$ when $p = \infin$. All constants below are therefore uniform in $J$.

    For the models $\Xi_{s,t} = Y_sX_{s,t}+Y'_s\mathbb{X}_{s,t}$ and $\widetilde{\Xi}$, subtracting the defect identities gives
    \begin{align*}
        \delta\Delta\Xi_{s,u,t} = {}& -\Delta R^Y_{s,u}X_{u,t} - R^{\widetilde{Y}}_{s,u}\Delta X_{u,t} \\
        &- \delta\Delta Y'_{s,u}\mathbb{X}_{u,t} - \delta\widetilde{Y}'_{s,u}\Delta\mathbb{X}_{u,t}.
    \end{align*}
    The weighted product estimate, lowering smoothness $3\alpha$ to $\gamma$ where necessary, yields
    $$
    \|\delta\Delta\Xi\|_{\overline{\mathbb{B}}^\gamma_{p/3,q/3}(w;J)} \lesssim_M K_J.
    $$
    On any interval $I \subseteq J$, Proposition~\ref{prop:Pointwise Rough Path Embedding} and Lemma~\ref{lem: Controlled-Remainder Estimate}, including their difference estimates, also give
    \begin{align*}
        |\delta\Delta\Xi_{s,u,t}| \lesssim_M c_I^3K_J\bigl(& (u-s)^{\beta-2r/p}(t-u)^\eta \\
        &+ (u-s)^\eta(t-u)^{\beta-2r/p}\bigr), \qquad c_I = \frac{|I|^{r/p}}{w(I)^{1/p}}.
    \end{align*}
    Here $c_I = 1$ for $p = \infin$, and factors $|I|^{2\alpha-\beta}$ are absorbed using $|I| \leq T_0$.

    Since $\gamma > 3r/p$, Theorems~\ref{thm:Subcritical Weighted Sewing} and~\ref{thm:Critical Weighted Sewing} apply in the respective regimes. At $\gamma = 1$, the largest-scale argument in the preceding proof bounds the additional critical defect term by $C_MK_J$. For $p = \infin$, use the corresponding unweighted sewing results. Uniqueness identifies the sewn remainder with
    $$
    \Delta R^\Xi = R^\Xi-R^{\widetilde{\Xi}}.
    $$
    In particular, when $\gamma = 1 < 3\alpha$, the original subcritical remainders belong to $\mathbb{B}^1_{p/3,\infin;\circ}(w;J)$, so critical uniqueness still applies.

    Repeat the local embedding and interpolation steps of the preceding proof with smoothness $\gamma$. The mixed exponents above permit splitting parameter $\eta/(\gamma-3r/p) \in (0,1/2)$, giving
    $$
    |\Delta R^\Xi_{s,t}| \lesssim_M K_J\frac{|t-s|^\gamma}{w([s,t])^{3/p}}.
    $$
    Interpolation from $p/3$ to $p/2$ then yields
    $$
    \|\Delta R^\Xi\|_{\mathbb{B}^{2\alpha}_{p/2,q/2}(w;J)} \lesssim_M \Lambda_{w,J}\tau^\epsilon K_J \lesssim_M \tau^\epsilon K_J.
    $$
    In the subcritical case, use the sewn remainder bound at scale index $q/3$; at criticality, use its power bound at scale index $\infin$ and integrate $h^\epsilon$. For $p = \infin$, the pointwise estimate gives the same conclusion directly.

    Finally, the one-parameter embedding gives
    $$
    \|\Delta Y'\|_{L^\infin(J)} \lesssim |\Delta Y'_a| + \Lambda_{w,J}\tau^\eta[\Delta Y']_{B^\alpha_{p,q}(w;J)},
    $$
    while $\|\widetilde{Y}'\|_{L^\infin(J)} \lesssim_M 1$. Combining
    $$
    \Delta R^Z_{s,t} = \Delta Y'_s\mathbb{X}_{s,t} + \widetilde{Y}'_s\Delta\mathbb{X}_{s,t} + \Delta R^\Xi_{s,t}
    $$
    with Lemma~\ref{lem: Controlled-Remainder Estimate} gives
    $$
    [\Delta Y]_{B^\alpha_{p,q}(w;J)} + \|\Delta R^Z\|_{\mathbb{B}^{2\alpha}_{p/2,q/2}(w;J)} \lesssim_M |\Delta Y'_a| + \varrho_J(\bold{X},\widetilde{\bold{X}}) + \tau^\epsilon D_J.
    $$
    We used $\epsilon \leq \eta$ and absorbed the remaining powers of $T_0$ into the constant. Adding $|z-\widetilde{z}|+|Y_a-\widetilde{Y}_a|$ proves the assertion.
\end{proof}

\begin{proposition}[Composition]
\label{prop: Rough Composition Estimates}
    Let $E, E_0$ be finite-dimensional vector spaces and let $(Y, Y') \in \mathscr{D}^\alpha_{p, q; \bold{X}}(w; E)$. For $F: E \rarr E_0$ define
    $$
    \mathcal{C}_F(Y, Y') = (F(Y), DF(Y) Y').
    $$
    \begin{enumerate}[label=(\alph*)]
        \item If $F \in C^2_b(E; E_0)$ then
        $$
        \mathcal{C}_F(Y, Y') \in \mathscr{D}^\alpha_{p, q; \bold{X}}(w; E_0),
        $$ and for any $M > 0$ there is some $\mathcal{C}_{M, F}$ such that
        $$
        \| \mathcal{C}_{F}(Y, Y') \|_{\mathscr{D}_X} \leq C_{M, F}
        $$ whenever the homogeneous driver size and the input controlled-path size are $\leq M$.

        \item Let $0 < \delta \leq 1$ and $F \in C^{2, \delta}_{b}(E; E_0)$. For a second controlled pair $(\widetilde{Y}, \widetilde{Y}')$ over $\widetilde{\bold{X}}$, set
        $$
        D = \varrho(\bold{X}, \widetilde{\bold{X}}) + d_{\bold{X}, \widetilde{\bold{X}}}((Y, Y'), (\widetilde{Y}, \widetilde{Y}')).
        $$
        On sets where both driver and controlled-path sizes are $\leq M$,
        $$
        d_{\bold{X}, \widetilde{\bold{X}}}(\mathcal{C}_F(Y, Y'), \mathcal{C}_F(\widetilde{Y}, \widetilde{Y}')) \leq C_{M, F} D^\delta.
        $$

        \item In particular, if $F \in C^{2, 1}_b(E; E_0)$ then
        $$
        d_{\bold{X}, \widetilde{\bold{X}}}(\mathcal{C}_F(Y, Y'), \mathcal{C}_F(\widetilde{Y}, \widetilde{Y}')) \leq C_{M, F} D.
        $$
    \end{enumerate}
\end{proposition}
\begin{proof}
    We use the Taylor expansion underlying \cite[Proposition 5.2]{FrizSeeger2022}. Write $B = B^\alpha_{p,q}(w)$ and $\mathbb{B} = \mathbb{B}^{2\alpha}_{p/2,q/2}(w)$. The elementary product estimates
    $$
    [UV]_B \lesssim \|U\|_\infin[V]_B + \|V\|_\infin[U]_B, \qquad \|(\delta U)(\delta V)\|_{\mathbb{B}} \lesssim [U]_B[V]_B
    $$
    follow from pointwise bounds and H\"older's inequality in $w_h(s) \hspace{.2em} ds$ and in the scale variable, including the stated quasi-Banach range.

    Set $Z = F(Y)$ and $Z' = DF(Y)Y'$. Then
    $$
    [Z]_B \leq \|DF\|_\infin[Y]_B, \qquad [Z']_B \lesssim \|DF\|_\infin[Y']_B + \|D^2F\|_\infin\|Y'\|_\infin[Y]_B.
    $$
    Taylor's formula gives
    $$
    R^Z_{s,t} = DF(Y_s)R^Y_{s,t} + T_F(Y)_{s,t},
    $$
    where
    $$
    T_F(Y)_{s,t} = \int_0^1 (1-\theta)D^2F(Y_s+\theta\delta Y_{s,t})[\delta Y_{s,t},\delta Y_{s,t}] \hspace{.2em} d\theta.
    $$
    Consequently,
    $$
    \|R^Z\|_{\mathbb{B}} \lesssim \|DF\|_\infin\|R^Y\|_{\mathbb{B}} + \|D^2F\|_\infin[Y]_B^2.
    $$
    Lemma~\ref{lem: Controlled-Remainder Estimate} with $\beta = 2\alpha$ and the one-parameter embedding control $[Y]_B$ and $\|Y'\|_\infin$ on the prescribed bounded sets. Together with the initial-value bounds, this proves (a).

    For (b), write $\Delta Y = Y-\widetilde{Y}$, $\Delta Y' = Y'-\widetilde{Y}'$, and $e = \|\Delta Y\|_\infin$. The same embeddings give
    $$
    [\Delta Y]_B + e + \|\Delta Y'\|_\infin \lesssim_M D.
    $$
    Put $G = DF(Y)-DF(\widetilde{Y})$. The fundamental theorem of calculus and the $\delta$-H\"older bound on $D^2F$ imply
    $$
    \|G\|_\infin \lesssim_F e, \qquad [G]_B \lesssim_F [\Delta Y]_B + e^\delta[\widetilde{Y}]_B.
    $$
    Indeed, corresponding points on the segments joining $Y_s$ to $Y_t$ and $\widetilde{Y}_s$ to $\widetilde{Y}_t$ differ by at most $e$. Using
    $$
    Z'-\widetilde{Z}' = GY' + DF(\widetilde{Y})\Delta Y'
    $$
    and the product estimate therefore gives $[Z'-\widetilde{Z}']_B \lesssim_{M,F} D+D^\delta$.

    Subtracting the integral Taylor remainders yields
    $$
    |T_F(Y)_{s,t}-T_F(\widetilde{Y})_{s,t}| \lesssim_F (|\delta Y_{s,t}|+|\delta\widetilde{Y}_{s,t}|)|\delta\Delta Y_{s,t}| + e^\delta|\delta\widetilde{Y}_{s,t}|^2.
    $$
    Hence
    $$
    \|T_F(Y)-T_F(\widetilde{Y})\|_{\mathbb{B}} \lesssim_F ([Y]_B+[\widetilde{Y}]_B)[\Delta Y]_B + e^\delta[\widetilde{Y}]_B^2.
    $$
    Combining this with
    \begin{align*}
        R^Z_{s,t}-R^{\widetilde{Z}}_{s,t} = {}& G_sR^Y_{s,t} + DF(\widetilde{Y}_s)(R^Y_{s,t}-R^{\widetilde{Y}}_{s,t}) \\
        &+ T_F(Y)_{s,t}-T_F(\widetilde{Y})_{s,t}
    \end{align*}
    proves $\|R^Z-R^{\widetilde{Z}}\|_{\mathbb{B}} \lesssim_{M,F} D+D^\delta$. The initial-value differences are bounded by $C_{M,F}D$. Since $D$ is bounded on the prescribed set, $D+D^\delta \lesssim_M D^\delta$, proving (b). Taking $\delta = 1$ gives (c).
\end{proof}

\subsection{Level-Two Weighted Rough Differential Equations}
We continue to assume $r/p < \alpha \leq 1/2$. We now apply the preceding integration and composition estimates to rough differential equations. Throughout we take $F \in C^{2, 1}_b$, which gives local Lipschitz continuity of composition in the full controlled-path space.

The two relevant thresholds are
$$
\alpha > \frac{r}{p}, \qquad 3\alpha > 1 \lor \frac{3r}{p}.
$$
The first provides pointwise control and the positive exponent $\alpha - r/p$ needed for contraction on sufficiently short intervals. The second permits subcritical sewing and, under the first condition, is equivalent to $\alpha > 1/3$. At $\alpha = 1/3$, we instead use critical sewing with $q \leq 3$; the subsequent integrability regain still benefits from the strictly positive gap $\alpha - r/p$.

\begin{theorem}[Level-$2$ Existence and Uniqueness]
\label{thm: Level-2 Existence and Uniqueness}
    Suppose
    $$
    0 < p, q \leq \infin, \qquad \frac{r}{p} < \alpha \leq \frac{1}{2}.
    $$
    Let
    $$
    F \in C^{2, 1}_{b}(\R^m; \mathcal{L}(\R^d, \R^m)), \qquad \bold{X} = (1, X, \mathbb{X}) \in \bold{B}^{\alpha, 2}_{p, q}(w),
    $$
    and assume either
    $$
    \textup{(Subcritical)} \hspace{1em} \alpha > \frac{1}{3}, \qquad \textup{ or } \qquad \textup{(Critical)} \hspace{1em} \alpha = \frac{1}{3} \hspace{1em} \textup{ and } \hspace{1em} 0 < q \leq 3.
    $$
    Then for every $\xi \in \R^m$ there exists a unique continuous controlled solution
    $$
    Y_t = \xi + \int^t_0 F(Y_s) \hspace{.2em} d\bold{X}_s, \qquad (Y, F(Y)) \in \mathscr{D}^\alpha_{p, q; \bold{X}}(w; \R^m).
    $$
    Moreover, for every $M > 0$ there is a $C_M > 0$ such that the bounds
    $$
    |\xi|, \|F\|_{C^{2, 1}_{b}}, \vertiii{\bold{X}}_{\bold{B}^{\alpha, 2}_{p, q}(w)} \leq M
    $$ imply
    $$
    |\xi| + |F(\xi)| + [F(Y)]_{B^\alpha_{p, q}(w)} + \|R^Y\|_{\mathbb{B}^{2\alpha}_{p/2, q/2}(w)} \leq C_M,
    $$ where
    $$
    R^Y_{s, t} = \delta Y_{s, t} - F(Y_s) X_{s, t}.
    $$
\end{theorem}
\begin{proof}
    We follow the fixed-point argument of \cite[Theorem 5.6]{FrizSeeger2022}. Fix bounds $M$ on the data and put $\eta = \alpha-r/p > 0$. For $J = [a,b]$, write $\tau = b-a$ and
    $$
    \mathcal{S}_J(Y,Y') = [Y']_{B^\alpha_{p,q}(w;J)} + \|R^Y\|_{\mathbb{B}^{2\alpha}_{p/2,q/2}(w;J)}.
    $$
    Prescribe $Y_a = y$ and $Y'_a = F(y)$, and define
    $$
    \Phi(Y,Y') = \left(y+\int_a^\cdot F(Y_s) \hspace{.2em} d\bold{X}_s,\ F(Y)\right).
    $$
    Proposition~\ref{prop: Rough Composition Estimates} and Theorem~\ref{thm: Weighted Rough Integration} show that $\Phi$ maps controlled pairs with these initial values into the same affine space. The closed set $\mathcal{S}_J(Y,Y') \leq R$ is nonempty, since it contains
    $$
    Y^0_t = y+F(y)X_{a,t}, \qquad (Y^0)'_t = F(y), \qquad R^{Y^0} = 0.
    $$
    It is complete for the metric obtained by raising the two difference seminorms to the power $\theta = 1\land(p/2)\land(q/2)$ and summing.

    All local constants can be chosen independently of $a$ and $y$. Indeed, translation replaces $F$ by $F_y(x) = F(y+x)$ without changing its $C_b^{2,1}$ bound, and the small-set condition gives
    $$
    \Lambda_{w,J} := \left(\frac{\tau^r}{w(J)}\right)^{1/p} \lesssim \Lambda_{w,T},
    $$
    with the usual convention when $p = \infin$. On the closed set above, Lemma~\ref{lem: Controlled-Remainder Estimate} with $\beta = 2\alpha$ yields
    $$
    [Y]_{B^\alpha_{p,q}(w;J)} \lesssim |F(y)|[X]_{B^\alpha_{p,q}(w;J)} + \Lambda_{w,J}\tau^\eta R\left(1+[X]_{B^\alpha_{p,q}(w;J)}\right).
    $$
    Consequently, the composition estimates and Theorem~\ref{thm: Weighted Rough Integration}(e), applied to $(F(Y),DF(Y)Y')$, give
    $$
    \mathcal{S}_J(\Phi(Y,Y')) \leq C_0+C_R\tau^\eta.
    $$
    Here $C_0$ is independent of $R$: the initial derivative of the integrand is $DF(y)F(y)$, while all remaining contributions have the factor $\tau^\eta$. Choose $R > 2C_0$, then choose $\tau_0 > 0$ sufficiently small that $\Phi$ preserves this closed set whenever $\tau \leq \tau_0$.

    For two pairs in the set, apply Theorem~\ref{thm: Stability of Weighted Rough Integration} with $\beta = 2\alpha$ to their compositions. Their initial integrand values are both $F(y)$ and their initial integrand derivatives are both $DF(y)F(y)$, so all initial differences vanish. Proposition~\ref{prop: Rough Composition Estimates}(c) gives
    $$
    d^J_{\bold{X},\bold{X}}\bigl(\Phi(Y,Y'),\Phi(\widetilde{Y},\widetilde{Y}')\bigr) \leq C_R\tau^\eta d^J_{\bold{X},\bold{X}}\bigl((Y,Y'),(\widetilde{Y},\widetilde{Y}')\bigr).
    $$
    Shrinking $\tau_0$ makes $\Phi$ a contraction in the complete metric described above. Its fixed point satisfies $Y' = F(Y)$ and the required integral equation. This argument includes $\alpha = 1/3$, since $q \leq 3$ permits critical sewing and the regain exponent remains $\eta > 0$.

    The admissible length $\tau_0$ is uniform in the starting point and initial value. We therefore iterate the construction on finitely many overlapping intervals of length at most $\tau_0$. Restriction preserves the bound $\mathcal{S}_J \leq R$, so local uniqueness makes the solutions agree on overlaps. The local path estimates and the one-parameter embedding bound the oscillation on each interval uniformly; hence the resulting path is bounded on $[0,T]$.

    Choose an overlap length $h_* > 0$ such that every increment of length at most $h_*$ lies in one of these intervals. Summing the local weighted integrals gives the global Besov bounds for $Y$, $F(Y)$, and $R^Y$ at these scales. Larger increments are bounded using the supremum bounds and $\int_0^{T-h}w_h(s) \hspace{.2em} ds \leq w([0,T])$. This proves that $(Y,F(Y))$ is globally controlled and gives the asserted a priori estimate. Locality and additivity of the rough integral give the equation on all of $[0,T]$.

    Finally, any two controlled solutions have finite controlled-path sizes. On sufficiently short intervals with the same initial value, the preceding difference estimate, with a constant depending on those sizes, can be absorbed to show that they coincide. Iterating proves uniqueness on $[0,T]$.
\end{proof}

\begin{proposition}[Davie Characterisation]
    Under the conditions of Theorem~\ref{thm: Level-2 Existence and Uniqueness} let $Y \in B^\alpha_{p, q}(w; \R^m)$ with $Y_0 = \xi$. Then $(Y, F(Y))$ is the solution to
    $$
    Y_t = \xi + \int^t_0 F(Y_s) \hspace{.2em} d\bold{X}_s
    $$ if and only if
    $$
    D_{s, t} = \delta Y_{s, t} - F(Y_s) X_{s, t} - (DFF)(Y_s) \mathbb{X}_{s, t}
    $$ satisfies
    $$
    \begin{cases}
        D \in \mathbb{B}^{3\alpha}_{p/3, q/3}(w), & \alpha > 1/3, \\
        D \in \mathbb{B}^{\omega_{q/3}^{(3)}}_{p/3, q/3}(w) \cap \mathbb{B}^{1}_{p/3, \infin; \circ}(w), & \alpha = 1/3,
    \end{cases}
    $$
\end{proposition}
\begin{proof}
    We follow \cite[Proposition 5.3]{FrizSeeger2022}. Necessity follows from Proposition~\ref{prop: Rough Composition Estimates} and Theorem~\ref{thm: Weighted Rough Integration}, applied to the controlled integrand $(F(Y), (DFF)(Y))$.

    Conversely, suppose $D$ satisfies the stated condition and set $\eta = \alpha - r/p > 0$. The path embedding and Proposition~\ref{prop:Pointwise Rough Path Embedding} give
    $$
    |\delta Y_{s, t}| + |X_{s, t}| \lesssim |t - s|^\eta, \qquad |\mathbb{X}_{s, t}| \lesssim |t - s|^{2\eta}.
    $$
    Since $F$ and $DF$ are bounded, $\Omega_\infin(D, h) \lesssim h^\eta$. Interpolation against the measure $w_h(s) \hspace{.2em} ds$ therefore gives
    $$
    \Omega^w_{p/2}(D, h) \leq \Omega_\infin(D, h)^{1/3} \Omega^w_{p/3}(D, h)^{2/3} \lesssim h^{\eta/3} \Omega^w_{p/3}(D, h)^{2/3},
    $$
    with the same inequality under the usual conventions when $p = \infin$. If $\alpha > 1/3$, then
    $$
    h^{-2\alpha} \Omega^w_{p/2}(D, h) \lesssim h^{\eta/3} \left( h^{-3\alpha} \Omega^w_{p/3}(D, h) \right)^{2/3},
    $$
    and taking the $L^{q/2}(dh/h)$ quasi-norm yields $D \in \mathbb{B}^{2\alpha}_{p/2, q/2}(w)$. If $\alpha = 1/3$, the little-space assumption gives $\Omega^w_{p/3}(D, h) \lesssim h$, so
    $$
    h^{-2/3} \Omega^w_{p/2}(D, h) \lesssim h^{\eta/3},
    $$
    which gives the same conclusion.

    Lipschitz composition yields $F(Y) \in B^\alpha_{p, q}(w)$, and
    $$
    R^Y_{s, t} = (DFF)(Y_s) \mathbb{X}_{s, t} + D_{s, t} \in \mathbb{B}^{2\alpha}_{p/2, q/2}(w).
    $$
    Thus $(Y, F(Y))$ is controlled by $\bold{X}$. Proposition~\ref{prop: Rough Composition Estimates} and Theorem~\ref{thm: Weighted Rough Integration} now define
    $$
    \widetilde{Y}_t = \xi + \int^t_0 F(Y_s) \hspace{.2em} d\bold{X}_s,
    $$
    whose remainder $\widetilde{D}_{s, t} = \delta\widetilde{Y}_{s, t} - F(Y_s)X_{s, t} - (DFF)(Y_s)\mathbb{X}_{s, t}$ satisfies the stated remainder condition. Consequently,
    $$
    \delta(Y - \widetilde{Y}) = D - \widetilde{D} \in
    \begin{cases}
        \mathbb{B}^{3\alpha}_{p/3, q/3}(w), & \alpha > 1/3, \\
        \mathbb{B}^{1}_{p/3, \infin; \circ}(w), & \alpha = 1/3.
    \end{cases}
    $$
    In the first case $3\alpha > 1 \lor (3r/p)$, while in the second $3r/p < 1$. The vacuity criterion of Proposition~2.13 therefore implies that $Y - \widetilde{Y}$ is constant. Both paths are continuous and start at $\xi$, so $Y = \widetilde{Y}$.
\end{proof}

\begin{theorem}[Local Lipschitz Continuity of Itô-Lyons Map]
\label{thm: Local Lipschitz Continuity of Rough Itô-Lyons Map}
    Let $1 \leq r < \infin$ and $w \in A_r$, and suppose
    $$
    0 < p, q \leq \infin, \qquad \frac{r}{p} < \alpha \leq \frac{1}{2}.
    $$
    Suppose either
    $$
    \textup{(Subcritical)} \hspace{1em} \alpha > \frac{1}{3}, \qquad \textup{ or } \qquad \textup{(Critical)} \hspace{1em} \alpha = \frac{1}{3} \hspace{1em} \textup{ and } \hspace{1em} 0 < q \leq 3.
    $$
    For $i = 1,2$ set
    $$
    \xi^i \in \R^m, \qquad F^i \in C^{2, 1}_{b}(\R^m; \mathcal{L}(\R^d, \R^m)), \qquad \bold{X}^i \in \bold{B}^{\alpha, 2}_{p, q}(w),
    $$
    and let $(Y^i, F^i(Y^i))$ be the corresponding solutions given by Theorem~\ref{thm: Level-2 Existence and Uniqueness}.
    Then for every $M > 0$ there is a $C_M$ such that whenever
    $$
    \max_{i = 1, 2} \left\{ |\xi^i|, \| F^i \|_{C^{2, 1}_{b}}, \vertiii{\bold{X}^i}_{\bold{B}^{\alpha, 2}_{p, q}(w)} \right\} \leq M
    $$
    we have
    $$
    d_{\bold{X}^1, \bold{X}^2}((Y^1, F^1(Y^1)), (Y^2, F^2(Y^2))) \leq C_M \left( |\xi^1 - \xi^2| + \|F^1 - F^2\|_{C^{2}_{b}} + \varrho(\bold{X}^1, \bold{X}^2) \right),
    $$
    where
    $$
    \varrho(\bold{X}^1, \bold{X}^2) = [X^1 - X^2]_{B^\alpha_{p, q}(w)} + \| \mathbb{X}^1 - \mathbb{X}^2 \|_{\mathbb{B}^{2\alpha}_{p/2, q/2}(w)}
    $$
    and
    \begin{align*}
        d_{\bold{X}^1, \bold{X}^2}((Y^1, F^1(Y^1)), (Y^2, F^2(Y^2))) = |\xi^1 - \xi^2| &+ |F^1(\xi^1) - F^2(\xi^2)| \\ &+ [F^1(Y^1) - F^2(Y^2)]_{B^\alpha_{p, q}(w)} + \|R^{Y^1} - R^{Y^2}\|_{\mathbb{B}^{2\alpha}_{p/2, q/2}(w)}.
    \end{align*}
\end{theorem}
\begin{proof}
    We follow \cite[Theorem 5.7]{FrizSeeger2022}. Set $(Y^i)' = F^i(Y^i)$, $\eta = \alpha - r/p > 0$, and
    $$
    \varepsilon = \|F^1 - F^2\|_{C^2_b} + \varrho(\bold{X}^1, \bold{X}^2), \qquad \mathcal{E} = |\xi^1 - \xi^2| + \varepsilon.
    $$
    Theorem~\ref{thm: Level-2 Existence and Uniqueness} and Lemma~\ref{lem: Controlled-Remainder Estimate} bound the controlled-path sizes and the supremum norms of both solutions by $C_M$. All local estimates below are uniform over subintervals $J$: the small-set condition bounds $(|J|^r/w(J))^{1/p}$ uniformly, with the usual convention when $p = \infin$.

    Define
    $$
    (W^i, (W^i)') = \mathcal{C}_{F^i}(Y^i, (Y^i)'), \qquad (\widetilde{W}, \widetilde{W}') = \mathcal{C}_{F^1}(Y^2, (Y^2)').
    $$
    For $J = [a, b]$, write $\tau = b - a$ and
    $$
    D_J = d^J_{\bold{X}^1, \bold{X}^2}((Y^1, (Y^1)'), (Y^2, (Y^2)')).
    $$
    Proposition~\ref{prop: Rough Composition Estimates}(c) controls the difference between $(W^1, (W^1)')$ and $(\widetilde{W}, \widetilde{W}')$. For the remaining difference, the Taylor estimates in its proof, applied to $F^1 - F^2$ along $(Y^2, (Y^2)')$, give a bound linear in $\|F^1 - F^2\|_{C^2_b}$. Consequently,
    $$
    [(W^1)' - (W^2)']_{B^\alpha_{p, q}(w; J)} + \|R^{W^1} - R^{W^2}\|_{\mathbb{B}^{2\alpha}_{p/2, q/2}(w; J)} \lesssim_M D_J + \varepsilon.
    $$
    Moreover, since $W^i = F^i(Y^i)$ and $(W^i)' = (DF^i F^i)(Y^i)$,
    $$
    |W^1_a - W^2_a| + |(W^1)'_a - (W^2)'_a| \lesssim_M |Y^1_a - Y^2_a| + \varepsilon.
    $$
    Applying Theorem~\ref{thm: Stability of Weighted Rough Integration} with $\beta = 2\alpha$, in either the subcritical or critical case, therefore yields, for $\tau \leq 1$,
    $$
    D_J \leq C_M \left( |Y^1_a - Y^2_a| + \varepsilon + \tau^\eta D_J \right).
    $$
    Choose $\tau_0 > 0$ so that $C_M \tau_0^\eta \leq 1/2$. Absorption, followed by Lemma~\ref{lem: Controlled-Remainder Estimate} and the path embedding, gives
    $$
    D_J + \|Y^1 - Y^2\|_{L^\infin(J)} \lesssim_M |Y^1_a - Y^2_a| + \varepsilon, \qquad |J| \leq \tau_0.
    $$
    Iterating over finitely many such intervals and using $(Y^i)' = F^i(Y^i)$ gives
    $$
    \|Y^1 - Y^2\|_{L^\infin} + \|(Y^1)' - (Y^2)'\|_{L^\infin} \lesssim_M \mathcal{E}.
    $$
    Hence $D_J \lesssim_M \mathcal{E}$ on every interval of length at most $\tau_0$.

    To recover the global seminorms, choose a finite overlapping cover by such intervals and $h_0 > 0$ so that every interval of length at most $h_0$ lies in one member of the cover. Summing the local weighted integral bounds controls the contribution of scales below $h_0$. For the remaining scales, write $\Delta Y = Y^1 - Y^2$, $\Delta Y' = (Y^1)' - (Y^2)'$, and $\Delta X = X^1 - X^2$. The identity
    $$
    R^{Y^1}_{s, t} - R^{Y^2}_{s, t} = \delta\Delta Y_{s, t} - \Delta Y'_s X^1_{s, t} - (Y^2)'_s \Delta X_{s, t}
    $$
    and the pointwise driver estimates give $\|R^{Y^1} - R^{Y^2}\|_{L^\infin} \lesssim_M \mathcal{E}$. Together with $\sup_h \int_0^{T-h} w_h(s) \hspace{.2em} ds \leq w([0, T])$, this controls all scales above $h_0$. Thus
    $$
    [(Y^1)' - (Y^2)']_{B^\alpha_{p, q}(w)} + \|R^{Y^1} - R^{Y^2}\|_{\mathbb{B}^{2\alpha}_{p/2, q/2}(w)} \lesssim_M \mathcal{E}.
    $$
    The initial-value differences are bounded by $C_M \mathcal{E}$ as well, proving the assertion.
\end{proof}

\subsection{Arbitrary-Level Weighted Besov Rough Paths}
\label{subsec: Arbitrary Level Weighted Rough Paths}

We extend the controlled-path construction to an arbitrary finite level. The analytic ingredients are the weighted sewing and embedding theorems already established. Nonlinear composition also requires an algebraic condition: we use rough paths taking values in the free nilpotent group. This is the geometric setting considered in \cite[\S 5.4]{FrizSeeger2022}.

Throughout this subsection, let $N\geq2$, let $V=\mathbb R^d$, let $E$ be a finite-dimensional normed vector space, and suppose
\begin{equation}
    1\leq r<\infin,\qquad w\in A_r,\qquad
    0<p,q\leq\infin,\qquad
    \frac rp<\alpha\leq\frac1N.
    \label{eq: Arbitrary Level Parameters}
\end{equation}
For integration and differential equations we additionally require
\begin{equation}
    \alpha>\frac1{N+1},
    \qquad\textup{or}\qquad
    \alpha=\frac1{N+1}\quad\textup{and}\quad0<q\leq N+1.
    \label{eq: Arbitrary Level Sewing Regime}
\end{equation}
We put
$$
\eta=\alpha-\frac rp>0,
\qquad
\Lambda_{w,J}=\left(\frac{|J|^r}{w(J)}\right)^{1/p},
$$
with $\Lambda_{w,J}=1$ when $p=\infin$. The small-set estimate gives $\Lambda_{w,J}\lesssim\Lambda_{w,[0,T]}$ uniformly over nondegenerate subintervals $J\subset[0,T]$. Constants below may depend on $N,\alpha,p,q,r,T$, the fixed tensor norms, $[w]_{A_r}$, and $\Lambda_{w,[0,T]}$. Any dependence on bounds for the driver, the controlled paths, or the vector field is indicated separately.

For brevity, for $k>0$ and $J\subset[0,T]$ write
$$
\|A\|_{k;J}=\|A\|_{\mathbb B^{k\alpha}_{p/k,q/k}(w;J)},
\qquad
\rho_J(\bold X,\widetilde{\bold X})
=\sum_{k=1}^N\|X^{(k)}-\widetilde X^{(k)}\|_{k;J}.
$$

Let $G^{(N)}(V)$ denote the free step-$N$ nilpotent group inside $T_1^{(N)}(V)$. Equivalently, its elements satisfy the shuffle identities
\begin{equation}
    \langle g,a\rangle\langle g,b\rangle
    =\langle g,a\shuffle b\rangle,
    \qquad |a|+|b|\leq N,
    \label{eq: Arbitrary Level Shuffle Identity}
\end{equation}
where $a,b$ are words in the dual tensor algebra, and $a\shuffle b$ is the sum of all order-preserving interleavings, counted with multiplicity.

\begin{definition}[Weakly Geometric Weighted Rough Path]
\label{def: Arbitrary Level Weakly Geometric}
    A rough path $\bold X\in B^{\alpha,N}_{p,q}(w)$ is weakly geometric if its continuous multiplicative representative satisfies
    $$
    \bold X_{s,t}\in G^{(N)}(V),\qquad 0\leq s\leq t\leq T.
    $$
    We denote this subspace by $B^{\alpha,N}_{p,q;\mathrm{wg}}(w)$.
\end{definition}

We identify $\mathcal L(V^{\otimes0},E)$ with $E$. If $A\in\mathcal L(V^{\otimes(i+j)},E)$ and $x\in V^{\otimes j}$, the contraction $Ax\in\mathcal L(V^{\otimes i},E)$ always means
$$
(Ax)(v)=A(x\otimes v).
$$
Thus the increment occupies the first tensor slots. For integrands, we identify $\mathcal L(V^{\otimes i},\mathcal L(V,E))$ with $\mathcal L(V^{\otimes(i+1)},E)$ by placing the final integration direction last.

\begin{definition}[Controlled Weighted Besov Path of Order $N$]
\label{def: Arbitrary Level Controlled Path}
    Let $\bold X\in B^{\alpha,N}_{p,q}(w)$. A collection of continuous paths
    $$
    \bold Y=(Y^{(0)},\ldots,Y^{(N-1)}),
    \qquad Y^{(i)}:[0,T]\longrightarrow\mathcal L(V^{\otimes i},E),
    $$
    is controlled by $\bold X$ if, for $0\leq i\leq N-1$, the remainders
    \begin{equation}
        R^{\bold Y,i}_{s,t}
        =\delta Y^{(i)}_{s,t}
        -\sum_{j=1}^{N-1-i}Y^{(i+j)}_sX^{(j)}_{s,t}
        \quad\textup{satisfy}\quad
        \|R^{\bold Y,i}\|_{N-i}<\infin.
        \label{eq: Arbitrary Level Controlled Expansion}
    \end{equation}
    Empty sums are zero, so $R^{\bold Y,N-1}=\delta Y^{(N-1)}$. We write $\bold Y\in\mathcal D^{\alpha,N}_{p,q;\bold X}(w;E)$ and set, for $J=[a,b]$,
    $$
    S_J(\bold Y)=\sum_{i=0}^{N-1}\|R^{\bold Y,i}\|_{N-i;J},
    \qquad
    \|\bold Y\|_{\mathcal D;J}
    =\sum_{i=0}^{N-1}|Y^{(i)}_a|+S_J(\bold Y).
    $$
    For paths controlled by possibly different drivers, set
    \begin{align*}
        S_{\Delta,J}(\bold Y,\widetilde{\bold Y})
        &=\sum_{i=0}^{N-1}
        \|R^{\bold Y,i}-R^{\widetilde{\bold Y},i}\|_{N-i;J},\\
        d_J(\bold Y,\widetilde{\bold Y})
        &=\sum_{i=0}^{N-1}|Y^{(i)}_a-\widetilde Y^{(i)}_a|
        +S_{\Delta,J}(\bold Y,\widetilde{\bold Y}).
    \end{align*}
    The drivers in $d_J$ are understood from the arguments.
\end{definition}

For $N=2$, this is the controlled-pair space of \S 6.3, with $Y=Y^{(0)}$, $Y'=Y^{(1)}$, $R^Y=R^{\bold Y,0}$, and $\delta Y'=R^{\bold Y,1}$. The next lemma shows that all coefficient paths in the preceding definition belong to $B^\alpha_{p,q}(w)$.

\begin{lemma}[Higher Controlled-Remainder Estimates]
\label{lem: Arbitrary Level Remainder Estimates}
    Let $\bold Y\in\mathcal D^{\alpha,N}_{p,q;\bold X}(w;E)$ and $J=[a,b]$. Put
    $$
    H_J=\sum_{k=1}^N\|X^{(k)}\|_{k;J}^{1/k},
    \qquad E_{i,J}=\|R^{\bold Y,i}\|_{N-i;J},
    \qquad K_{i,J}=\sum_{\ell=i}^{N-1}H_J^{\ell-i}E_{\ell,J},
    $$
    with $H_J^0=1$. For $s<t$ in $J$,
    \begin{equation}
        |R^{\bold Y,i}_{s,t}|
        \leq C K_{i,J}
        \frac{|t-s|^{(N-i)\alpha}}{w([s,t])^{(N-i)/p}}.
        \label{eq: Arbitrary Level Pointwise Remainder}
    \end{equation}
    If $m=N-i$ and $1\leq k<m$, then
    \begin{equation}
        \|R^{\bold Y,i}\|_{k;J}
        \leq C\Lambda_{w,J}^{m-k}|J|^{(m-k)\eta}K_{i,J}.
        \label{eq: Arbitrary Level Remainder Regain}
    \end{equation}
    On sets where the driver and controlled-path sizes are bounded by $M$,
    $$
    \sum_{i=0}^{N-1}
    \left(\|Y^{(i)}\|_{\infin;J}+[Y^{(i)}]_{B^\alpha_{p,q}(w;J)}\right)
    \leq C_M.
    $$
    For two such pairs, the corresponding difference is bounded by
    $$
    C_M\bigl(d_J(\bold Y,\widetilde{\bold Y})
    +\rho_J(\bold X,\widetilde{\bold X})\bigr).
    $$
    The pointwise and regain estimates also hold for remainder differences, with $K_{i,J}$ replaced by this last comparison quantity and the constant enlarged to $C_M$.
\end{lemma}
\begin{proof}
    Chen's relation and \eqref{eq: Arbitrary Level Controlled Expansion} give the exact identity
    \begin{equation}
        \delta R^{\bold Y,i}_{s,u,t}
        =\sum_{j=1}^{N-1-i}
        R^{\bold Y,i+j}_{s,u}X^{(j)}_{u,t}.
        \label{eq: Arbitrary Level Remainder Defect}
    \end{equation}
    We argue downwards in $i$. For $i=N-1$, the pointwise assertion is the one-parameter embedding. Suppose it holds at all larger indices, and put $m=N-i\geq2$. On an interval $I$, local deweighting, the induction hypothesis, and Proposition~6.2 give
    $$
    |\delta R^{\bold Y,i}_{s,u,t}|
    \leq Cc_I^m K_{i,J}
    \sum_{j=1}^{m-1}|u-s|^{(m-j)\eta}|t-u|^{j\eta},
    \qquad
    c_I=[w]_{A_r}^{1/p}\frac{|I|^{r/p}}{w(I)^{1/p}},
    $$
    for $s<u<t$ in $I\subset J$. Each product is bounded by $a^\eta b^{(m-1)\eta}$, where $a$ and $b$ are the smaller and larger adjacent gaps. Also,
    $$
    \|R^{\bold Y,i}\|_{\mathbb B^{m\alpha}_{p/(rm),q/m}(I)}
    \leq c_I^m E_{i,J}.
    $$
    The unweighted two-parameter embedding used in \S 3, with splitting parameter $1/m$, therefore applies: its positive regularity gap is $m\eta$. Taking $I=[s,t]$ proves \eqref{eq: Arbitrary Level Pointwise Remainder}. For $p=\infin$, omit the weight factors and use the same argument.

    The lower-mass estimate now yields
    $$
    \Omega_\infin(R^{\bold Y,i},h;J)
    \leq C\Lambda_{w,J}^mK_{i,J}h^{m\eta}.
    $$
    With $\theta=k/m$, interpolation at each scale gives
    $$
    h^{-k\alpha}\Omega^w_{p/k}(R^{\bold Y,i},h;J)
    \leq C\Lambda_{w,J}^{m-k}K_{i,J}^{1-\theta}h^{(m-k)\eta}
    \left(h^{-m\alpha}\Omega^w_{p/m}(R^{\bold Y,i},h;J)\right)^\theta.
    $$
    Take the $L^{q/k}(dh/h)$ quasi-norm and use $\theta q/k=q/m$. This proves \eqref{eq: Arbitrary Level Remainder Regain}, including the supremum conventions. Starting with $Y^{(N-1)}$, the controlled decomposition, Lemma~6.4, and this regain estimate then give the asserted path bounds by downward induction. The one-parameter embedding controls each supremum norm from its initial value.

    For differences, subtract \eqref{eq: Arbitrary Level Remainder Defect}. Each summand becomes
    $$
    (R^{\bold Y,i+j}-R^{\widetilde{\bold Y},i+j})_{s,u}X^{(j)}_{u,t}
    +R^{\widetilde{\bold Y},i+j}_{s,u}
    (X^{(j)}-\widetilde X^{(j)})_{u,t}.
    $$
    Repeat the same induction, using the difference estimate in Proposition~6.2. It is linear in the remainder and driver differences on bounded sets. Interpolation preserves that linear bound because the pointwise and Besov bounds are both linear in the same comparison quantity. Subtracting the controlled decompositions finishes the proof.
\end{proof}

For a fixed driver the controlled-path space is complete. Indeed, its initial values and remainders determine all coefficient paths by the preceding estimates. A Cauchy sequence converges uniformly in every coefficient and in every remainder space, and the controlled identities pass to the limit. For indices below one, an equivalent complete metric is obtained from the individual difference seminorms using a sufficiently small positive power. Thus the contraction argument below applies in the full stated range of $q$.

We now develop rough integration at arbitrary levels. Write
$$
\gamma=(N+1)\alpha,
\qquad P=\frac p{N+1},
\qquad Q=\frac q{N+1},
\qquad \nu=1\lor\frac{(N+1)r}{p}.
$$
At the endpoint in \eqref{eq: Arbitrary Level Sewing Regime}, $\gamma=\nu=1$, $P>r$, and $Q\leq1$. For $0<\sigma<\infin$, let $\ell_\sigma$ be a loss function as in \S 4.2 and put
$$
\omega_\sigma(h)=\sup_{0<u\leq h}u\ell_\sigma(u),
\qquad \omega_\infin(h)=h.
$$
We use $\mathbb B^1_{P,\infin;\circ}(w)$ for the little space defined by
$$
\lim_{h\downarrow0}h^{-1}\Omega_P^w(A,h)=0.
$$

\begin{theorem}[Arbitrary-Level Weighted Rough Integration]
\label{thm: Arbitrary Level Rough Integration}
    Assume \eqref{eq: Arbitrary Level Parameters}--\eqref{eq: Arbitrary Level Sewing Regime}, take $\bold X\in B^{\alpha,N}_{p,q}(w)$, and let
    $$
    \bold Y\in\mathcal D^{\alpha,N}_{p,q;\bold X}(w;\mathcal L(V,E)).
    $$
    Define
    \begin{equation}
        \Xi_{s,t}=\sum_{j=0}^{N-1}Y^{(j)}_sX^{(j+1)}_{s,t}.
        \label{eq: Arbitrary Level Integral Model}
    \end{equation}
    For $z\in E$, there is a unique continuous path $Z^{(0)}$ starting at $z$ whose sewing remainder $\mathcal R=\delta Z^{(0)}-\Xi$ belongs to
    \begin{equation}
        \begin{cases}
        \mathbb B^\gamma_{P,Q}(w),&\alpha>1/(N+1),\\
        \mathbb B^1_{P,\infin;\circ}(w),&\alpha=1/(N+1).
        \end{cases}
        \label{eq: Arbitrary Level Sewing Remainder Class}
    \end{equation}
    We write $Z^{(0)}_t=z+\int_0^t\bold Y_u\,d\bold X_u$.

    Set $K_i=K_{i,[0,T]}$ as in Lemma~\ref{lem: Arbitrary Level Remainder Estimates}, and
    $$
    \mathcal N=\sum_{j=0}^{N-1}K_jH_{[0,T]}^{j+1}.
    $$
    Then
    \begin{align}
        \|\mathcal R\|_{\mathbb B^\gamma_{P,Q}(w)}&\leq C\mathcal N,
        &&\alpha>1/(N+1),
        \label{eq: Arbitrary Level Subcritical Bound}\\
        \|\mathcal R\|_{\mathbb B^1_{P,\infin}(w)}&\leq C\mathcal N,
        &&\alpha=1/(N+1).
        \label{eq: Arbitrary Level Critical Bound}
    \end{align}
    At criticality, $\mathcal R$ also belongs to $\mathbb B^{\omega_\sigma}_{P,\sigma}(w)$ for every finite $\sigma>0$, with the loss-space bounds of Theorem~4.5. In both regimes,
    \begin{equation}
        |\mathcal R_{s,t}|\leq C\mathcal N
        \frac{|t-s|^{(N+1)\alpha}}{w([s,t])^{(N+1)/p}}.
        \label{eq: Arbitrary Level Pointwise Sewing}
    \end{equation}

    The collection
    $$
    \bold Z=(Z^{(0)},Y^{(0)},\ldots,Y^{(N-2)})
    $$
    belongs to $\mathcal D^{\alpha,N}_{p,q;\bold X}(w;E)$. Its remainders are
    \begin{align}
        R^{\bold Z,0}_{s,t}
        &=Y^{(N-1)}_sX^{(N)}_{s,t}+\mathcal R_{s,t},
        \label{eq: Arbitrary Level Integrated Remainder Zero}\\
        R^{\bold Z,i}_{s,t}
        &=Y^{(N-1)}_sX^{(N-i)}_{s,t}
        +R^{\bold Y,i-1}_{s,t},\qquad1\leq i\leq N-1.
        \label{eq: Arbitrary Level Integrated Remainders}
    \end{align}
    For fixed $\bold X$ and $z=0$, integration is a bounded linear map between these controlled-path spaces. It is locally Lipschitz jointly in the driver, the controlled integrand, and the initial value.

    More precisely, on bounded sets and on $J=[a,b]$ with $\tau=|J|\leq1$, integration starting at $z$ satisfies
    \begin{equation}
        S_J(\bold Z)
        \leq C|Y^{(N-1)}_a|\sum_{k=1}^N\|X^{(k)}\|_{k;J}
        +C_M\tau^\eta S_J(\bold Y).
        \label{eq: Arbitrary Level Local Integral Size}
    \end{equation}
    For a second set of bounded data,
    \begin{align}
        d_J(\bold Z,\widetilde{\bold Z})
        \leq C_M\Bigl(&|z-\widetilde z|
        +\sum_{i=0}^{N-1}|Y^{(i)}_a-\widetilde Y^{(i)}_a|
        +\rho_J(\bold X,\widetilde{\bold X})\nonumber\\
        &+\tau^\eta S_{\Delta,J}(\bold Y,\widetilde{\bold Y})\Bigr).
        \label{eq: Arbitrary Level Local Integral Difference}
    \end{align}
    The constants are uniform in the position of $J$.
\end{theorem}
\begin{proof}
    Lemma~6.4 and boundedness of the coefficient paths give $\Xi\in\mathbb B^\alpha_{p,q}(w)$. The controlled identities and Chen's relation give
    \begin{equation}
        \delta\Xi_{s,u,t}
        =-\sum_{j=0}^{N-1}R^{\bold Y,j}_{s,u}X^{(j+1)}_{u,t}.
        \label{eq: Arbitrary Level Integral Defect}
    \end{equation}
    Each summand has total degree $(N-j)+(j+1)=N+1$. The two-parameter product estimate used in the proof of Theorem~6.5 yields
    $$
    \|\delta\Xi\|_{\overline{\mathbb B}^\gamma_{P,Q}(w)}
    \leq C\sum_{j=0}^{N-1}E_{j,[0,T]}\|X^{(j+1)}\|_{j+1}
    \leq C\mathcal N.
    $$
    Its strict subdivision inequalities follow from $\alpha>r/p$. Lemma~\ref{lem: Arbitrary Level Remainder Estimates} and Proposition~6.2 also give, on every interval $I$,
    $$
    |\delta\Xi_{s,u,t}|
    \leq Cc_I^{N+1}\mathcal N
    \sum_{j=0}^{N-1}|u-s|^{(N-j)\eta}|t-u|^{(j+1)\eta}
    \leq Cc_I^{N+1}\mathcal N a^\eta b^{N\eta}.
    $$
    Here $a$ and $b$ are the smaller and larger adjacent gaps. At criticality this pointwise estimate controls the largest-scale term in $D_c(\Xi)$, exactly as in Theorem~6.9; hence $D_c(\Xi)\leq C\mathcal N$.

    In the strict regime, $\gamma>1\lor r/P$. At criticality, $\gamma=1$, $P>r\geq1$, and $Q\leq1$. Theorems~4.3 and~4.5 therefore apply. For $p=\infin$, apply the unweighted sewing theorem as in Theorem~6.9. These results give the asserted Besov bounds and little-space membership. Locally deweight the remainder bound and use $\delta\mathcal R=-\delta\Xi$ in the two-parameter embedding, with splitting parameter $1/(N+1)$. The positive gap is $(N+1)\eta$, also at criticality. Taking $I=[s,t]$ proves \eqref{eq: Arbitrary Level Pointwise Sewing} and supplies a continuous sewn path. Uniqueness follows from the vacuity criterion, using the little space at the endpoint.

    For $1\leq k\leq N$, interpolate this pointwise bound with \eqref{eq: Arbitrary Level Subcritical Bound}. As in Lemma~\ref{lem: Arbitrary Level Remainder Estimates},
    \begin{equation}
        \|\mathcal R\|_{k;J}
        \leq C\Lambda_{w,J}^{N+1-k}|J|^{(N+1-k)\eta}\mathcal N_J,
        \label{eq: Arbitrary Level Sewing Regain}
    \end{equation}
    where $\mathcal N_J$ is defined from the restricted data. At criticality use the power estimate \eqref{eq: Arbitrary Level Critical Bound}; interpolation gives
    $$
    h^{-k\alpha}\Omega^w_{p/k}(\mathcal R,h;J)
    \leq C\Lambda_{w,J}^{N+1-k}\mathcal N_Jh^{(N+1-k)\eta},
    $$
    which is integrable to every finite scale exponent. Thus \eqref{eq: Arbitrary Level Sewing Regain} holds in both regimes, without a logarithmic factor. The same proof includes $p=\infin$.

    The identities \eqref{eq: Arbitrary Level Integrated Remainder Zero}--\eqref{eq: Arbitrary Level Integrated Remainders} follow directly from the definitions. Apply \eqref{eq: Arbitrary Level Sewing Regain} to the first and \eqref{eq: Arbitrary Level Remainder Regain} to the others. Splitting the common coefficient as
    $$
    Y^{(N-1)}_s=Y^{(N-1)}_a+\delta Y^{(N-1)}_{a,s},
    \qquad
    \|\delta Y^{(N-1)}_{a,\cdot}\|_{\infin;J}
    \leq C\Lambda_{w,J}\tau^\eta E_{N-1,J},
    $$
    gives \eqref{eq: Arbitrary Level Local Integral Size}. The resulting controlled bounds and the preceding lemma put $Z^{(0)}$ in $B^\alpha_{p,q}(w)$.

    For differences, subtract \eqref{eq: Arbitrary Level Integral Defect} and apply the difference part of Lemma~\ref{lem: Arbitrary Level Remainder Estimates}. Sewing is linear in its two-parameter input, so its remainder difference obeys the same estimates, with $\mathcal N_J$ replaced by
    $$
    C_M\bigl(S_{\Delta,J}(\bold Y,\widetilde{\bold Y})
    +\rho_J(\bold X,\widetilde{\bold X})\bigr).
    $$
    Subtracting \eqref{eq: Arbitrary Level Integrated Remainder Zero}--\eqref{eq: Arbitrary Level Integrated Remainders}, and separating the initial values of the highest coefficients, now gives \eqref{eq: Arbitrary Level Local Integral Difference}. Bounded linearity for a fixed driver follows from the construction and the same estimates.
\end{proof}

The compensated sums $I_\pi\Xi$ converge to $\delta Z^{(0)}$ along arbitrary rescaled partitions in the sewing norms: in $\mathbb B^\gamma_{P,Q}(w)$ in the strict regime, and in $\mathbb B^1_{P,\infin}(w)$ and every $\mathbb B^{\omega_\sigma}_{P,\sigma}(w)$ at criticality. These are the norm-convergence assertions of \S 4.

For $i\geq1$, let $\mathcal P(i)$ be the set of partitions of $\{1,\ldots,i\}$ into nonempty, unordered blocks. If $B=\{b_1<\cdots<b_j\}$ and $v=v_1\otimes\cdots\otimes v_i$, write $v_B=v_{b_1}\otimes\cdots\otimes v_{b_j}$. Define
\begin{align}
    (\mathcal C_F\bold Y)^{(0)}&=F(Y^{(0)}),\nonumber\\
    (\mathcal C_F\bold Y)^{(i)}(v)
    &=\sum_{\pi\in\mathcal P(i)}
    D^{|\pi|}F(Y^{(0)})
    \bigl[Y^{(|B|)}(v_B):B\in\pi\bigr],\qquad1\leq i\leq N-1.
    \label{eq: Arbitrary Level Composition Formula}
\end{align}
The order of the blocks does not matter because the derivatives of $F$ are symmetric. No factorial is needed, since the blocks are unordered. In particular,
$$
(\mathcal C_F\bold Y)^{(1)}=DF(Y^{(0)})Y^{(1)},
\qquad
(\mathcal C_F\bold Y)^{(2)}
=DF(Y^{(0)})Y^{(2)}+D^2F(Y^{(0)})[Y^{(1)},Y^{(1)}].
$$
This is the usual shuffle-based composition formula; see also \cite[\S 4]{BoedihardjoGeng2022}.

\begin{proposition}[Composition at Arbitrary Level]
\label{prop: Arbitrary Level Composition}
    Let $\bold X\in B^{\alpha,N}_{p,q;\mathrm{wg}}(w)$ and $\bold Y\in\mathcal D^{\alpha,N}_{p,q;\bold X}(w;E)$. If $F\in C_b^N(E;E_0)$, then
    $$
    \mathcal C_F\bold Y\in\mathcal D^{\alpha,N}_{p,q;\bold X}(w;E_0),
    \qquad
    \|\mathcal C_F\bold Y\|_{\mathcal D}\leq C_M\|F\|_{C_b^N}
    $$
    whenever the driver and controlled-path sizes are bounded by $M$.

    If $F,\widetilde F\in C_b^{N,1}$ and the two vector fields, drivers, and controlled paths have sizes bounded by $M$, then
    \begin{align}
        d_J(\mathcal C_F\bold Y,
        \mathcal C_{\widetilde F}\widetilde{\bold Y})
        \leq C_M\Bigl(&d_J(\bold Y,\widetilde{\bold Y})
        +\rho_J(\bold X,\widetilde{\bold X})
        +\|F-\widetilde F\|_{C_b^N}\Bigr).
        \label{eq: Arbitrary Level Composition Difference}
    \end{align}
    The same conclusions hold uniformly on subintervals.
\end{proposition}

\begin{proof}
    We give the degree argument underlying the weighted estimates. Assign degree $j$ to $X^{(j)}_{s,t}$, degree $N-j$ to $R^{\bold Y,j}_{s,t}$, and degree one to $\delta Y^{(0)}_{s,t}$. A product of factors of degrees $d_1,\ldots,d_l$ belongs to
    $$
    \mathbb B^{d\alpha}_{p/d,q/d}(w),\qquad d=d_1+\cdots+d_l,
    $$
    whenever the individual factors have their indicated Besov bounds. This follows from H\"older's inequality against the same measure $w_h(s)\,ds$ and then against $dh/h$. It remains valid when the target indices are below one. Bounded coefficient functions can be included without changing the estimate. If $d>k$, pointwise control and interpolation lower the degree to $k$, with the factor $C\Lambda_{w,J}^{d-k}|J|^{(d-k)\eta}$; this is precisely the calculation in \eqref{eq: Arbitrary Level Remainder Regain}.

    Fix $i$ and put $m=N-i$. In \eqref{eq: Arbitrary Level Composition Formula} at time $t$, expand each $Y^{(j)}_t$ by \eqref{eq: Arbitrary Level Controlled Expansion}, and expand each derivative of $F$ about $Y^{(0)}_s$ to order $m-1$. All derivatives needed are of order at most $i+m=N$. Retain terms of total increment degree less than $m$.

    To identify those terms, distribute the $i$ fixed tensor slots among the blocks in \eqref{eq: Arbitrary Level Composition Formula}. Distribute the additional increment slots among the same blocks and the new blocks arising from the Taylor expansion. The shuffle identities express products of increment coordinates as the sum of all such order-preserving distributions. Consequently, the terms of increment degree $j<m$ sum to
    $$
    (\mathcal C_F\bold Y)^{(i+j)}_sX^{(j)}_{s,t},
    $$
    including the constant term $j=0$. This is a finite algebraic identity in each degree. Truncation only affects terms of degree at least $m$.

    The remaining terms contain a controlled remainder or have total increment degree at least $m$. Taylor's integral remainder is bounded by a constant times $|\delta Y^{(0)}_{s,t}|^m$, with bounded coefficient factors. In a term containing $R^{\bold Y,j}$, its degree $N-j$, together with the degrees of its other increment factors, is at least $m$: otherwise the fixed and increment slots in that factor would exceed the truncation used in the expansion. The preceding product and regain estimates therefore put every error term in $\mathbb B^{m\alpha}_{p/m,q/m}(w)$. This proves the controlled remainder bound at index $i$. The initial coefficient bounds follow directly from \eqref{eq: Arbitrary Level Composition Formula}.

    For differences, subtract the finite expansions. A product difference is a sum of products with one factor replaced by its difference. Lemma~\ref{lem: Arbitrary Level Remainder Estimates} supplies both Besov and pointwise difference bounds. For the Taylor remainders, the Lipschitz bound on $D^NF$ gives, on the bounded sets under consideration, a sum of terms of the form
    $$
    C_M|Y^{(0)}_s-\widetilde Y^{(0)}_s|
    |\delta\widetilde Y^{(0)}_{s,t}|^m
    +C_M\bigl(|\delta Y^{(0)}_{s,t}|+|\delta\widetilde Y^{(0)}_{s,t}|\bigr)^{m-1}
    |\delta(Y^{(0)}-\widetilde Y^{(0)})_{s,t}|,
    $$
    with the same types of bounded coefficient factors and their differences. Apply the product estimates again. Finally separate the change from $F$ to $\widetilde F$; the single-map Taylor bound is linear in $\|F-\widetilde F\|_{C_b^N}$. This proves \eqref{eq: Arbitrary Level Composition Difference}.
\end{proof}

For $F:E\to\mathcal L(V,E)$, define the coefficient maps $F^{[k]}:E\to\mathcal L(V^{\otimes k},E)$ recursively by
\begin{align}
    F^{[1]}(y)v_1&=F(y)v_1,\nonumber\\
    F^{[k+1]}(y)(v_1\otimes\cdots\otimes v_{k+1})
    &=D\bigl(F^{[k]}(\cdot)(v_2\otimes\cdots\otimes v_{k+1})\bigr)(y)
    [F(y)v_1].
    \label{eq: Arbitrary Level Vector Field Coefficients}
\end{align}
The order agrees with the contraction convention above. Thus $F^{[2]}=DF\,F$. Write
$$
\bold F_{<N}(y)
=(y,F^{[1]}(y),\ldots,F^{[N-1]}(y)).
$$

\begin{theorem}[Arbitrary-Level Weighted Rough Differential Equations]
\label{thm: Arbitrary Level RDE}
    Assume \eqref{eq: Arbitrary Level Parameters}--\eqref{eq: Arbitrary Level Sewing Regime}. Let
    $$
    \bold X\in B^{\alpha,N}_{p,q;\mathrm{wg}}(w),
    \qquad F\in C_b^{N,1}(E;\mathcal L(V,E)),\qquad \xi\in E.
    $$
    There is a unique controlled solution
    $$
    \bold Y\in\mathcal D^{\alpha,N}_{p,q;\bold X}(w;E),
    \qquad
    Y^{(0)}_t=\xi+\int_0^t\mathcal C_F\bold Y_u\,d\bold X_u,
    $$
    where the equation holds as an equality of controlled collections. Its coefficients satisfy
    $$
    Y^{(i)}=F^{[i]}(Y^{(0)}),\qquad1\leq i\leq N-1.
    $$
    For data bounded by $M$,
    $$
    |\xi|,\ \|F\|_{C_b^{N,1}},\ 
    \sum_{k=1}^N\|X^{(k)}\|_k^{1/k}\leq M
    \quad\Longrightarrow\quad
    \|\bold Y\|_{\mathcal D}\leq C_M.
    $$
    For two such sets of data, the solutions satisfy
    \begin{equation}
        d_{[0,T]}(\bold Y,\widetilde{\bold Y})
        \leq C_M\left(
        |\xi-\widetilde\xi|+\|F-\widetilde F\|_{C_b^N}
        +\rho_{[0,T]}(\bold X,\widetilde{\bold X})\right).
        \label{eq: Arbitrary Level Ito Lyons}
    \end{equation}
    The same bound controls the sum of the supremum and $B^\alpha_{p,q}(w)$ seminorms of all coefficient differences.
\end{theorem}

\begin{proof}
    Fix $J=[a,b]$ and $y\in E$. Prescribe the initial collection $\bold Y_a=\bold F_{<N}(y)$ and consider the map
    $$
    \Phi(\bold Y)
    =\left(y+\int_a^\cdot\mathcal C_F\bold Y_u\,d\bold X_u,
    (\mathcal C_F\bold Y)^{(0)},\ldots,
    (\mathcal C_F\bold Y)^{(N-2)}\right).
    $$
    The first coordinate denotes the path component of the integral. The algebraic formula \eqref{eq: Arbitrary Level Composition Formula}, or induction in \eqref{eq: Arbitrary Level Vector Field Coefficients}, gives
    $$
    (\mathcal C_F\bold F_{<N}(y))^{(i)}=F^{[i+1]}(y),
    \qquad0\leq i\leq N-1.
    $$
    Hence $\Phi$ preserves the prescribed initial collection. The affine controlled space is nonempty: it contains the collection
    $$
    Y_t^{(i),0}=\sum_{j=0}^{N-1-i}g^{(i+j)}X^{(j)}_{a,t},
    \qquad
    (g^{(0)},\ldots,g^{(N-1)})=\bold F_{<N}(y),
    \qquad X^{(0)}=1,
    $$
    whose remainders vanish by Chen's relation.

    On a closed set $S_J(\bold Y)\leq R$, the composition estimate and \eqref{eq: Arbitrary Level Local Integral Size} yield
    $$
    S_J(\Phi(\bold Y))\leq C_0+C_R|J|^\eta.
    $$
    Here $C_0$ is independent of $R$: the initial highest coefficient of the integrand is the fixed value $F^{[N]}(y)$. The other coefficients and their oscillations are controlled by the prescribed data and the remainder bound. These constants are uniform in $y$, since all derivatives of $F$ used here are bounded; the path value $Y^{(0)}$ may be translated by $y$ when applying the local estimates. Choose $R>2C_0$ and then choose $\tau_0>0$ so that $\Phi$ preserves the closed set whenever $|J|\leq\tau_0$.

    For two collections in this set, their initial integrand coefficients agree. Equations~\eqref{eq: Arbitrary Level Local Integral Difference} and~\eqref{eq: Arbitrary Level Composition Difference} therefore give
    $$
    d_J(\Phi(\bold Y),\Phi(\widetilde{\bold Y}))
    \leq C_R|J|^\eta d_J(\bold Y,\widetilde{\bold Y}).
    $$
    Shrinking $\tau_0$ makes this a contraction in an equivalent complete metric. The fixed point has the stated coefficient identities by induction in the component index. This proves local existence and uniqueness.

    The interval length can be chosen uniformly in its position and initial value. Iterate on finitely many overlapping intervals. Local uniqueness identifies the solutions on overlaps. Each coefficient has a uniform oscillation bound on each interval, and the zeroth component therefore remains bounded on $[0,T]$. For sufficiently small increment lengths, every increment lies in one interval of the cover, and summing the local weighted integrals gives the global Besov estimates. Larger increments are controlled by the coefficient supremum bounds, Chen's relation, and $\int_0^{T-h}w_h(s)\,ds\leq w([0,T])$. This proves the global controlled bound and the equation on $[0,T]$.

    For stability, let
    $$
    \varepsilon=\|F-\widetilde F\|_{C_b^N}
    +\rho_{[0,T]}(\bold X,\widetilde{\bold X}).
    $$
    On a sufficiently short interval $J=[a,b]$, the initial coefficient differences are bounded by $C_M(|Y^{(0)}_a-\widetilde Y^{(0)}_a|+\varepsilon)$. Apply the local integration and composition difference estimates to the two solutions. They give
    $$
    d_J(\bold Y,\widetilde{\bold Y})
    \leq C_M\bigl(|Y^{(0)}_a-\widetilde Y^{(0)}_a|+\varepsilon\bigr)
    +C_M|J|^\eta d_J(\bold Y,\widetilde{\bold Y}).
    $$
    Absorb the last term, use Lemma~\ref{lem: Arbitrary Level Remainder Estimates} to pass to the next initial value, and iterate over the finite cover. The same localisation of increment norms proves \eqref{eq: Arbitrary Level Ito Lyons}; the coefficient estimates follow from the lemma. Applying this argument to any two controlled solutions also proves global uniqueness.

    At $\alpha=1/(N+1)$, critical sewing is used only in constructing the integral. All subsequent regain and contraction estimates retain the positive power $|J|^\eta$, so the argument includes the endpoint without an additional logarithmic loss.
\end{proof}

The preceding construction also admits a characterisation by its order-$N$ increment expansion.

\begin{proposition}[Arbitrary-Level Davie Characterisation]
\label{prop: Arbitrary Level Davie}
    Under the hypotheses of Theorem~\ref{thm: Arbitrary Level RDE}, let $Y\in B^\alpha_{p,q}(w;E)$ be continuous, with $Y_0=\xi$. Define
    \begin{equation}
        D_{s,t}=\delta Y_{s,t}
        -\sum_{k=1}^NF^{[k]}(Y_s)X^{(k)}_{s,t}.
        \label{eq: Arbitrary Level Davie Defect}
    \end{equation}
    Then $Y$ is the path component of the solution if and only if
    $$
    \begin{cases}
    D\in\mathbb B^{(N+1)\alpha}_{p/(N+1),q/(N+1)}(w),
        &\alpha>1/(N+1),\\
    D\in\mathbb B^1_{p/(N+1),\infin;\circ}(w),
        &\alpha=1/(N+1).
    \end{cases}
    $$
    For the solution, on bounded sets of data,
    \begin{equation}
        |D_{s,t}|\leq C_M
        \frac{|t-s|^{(N+1)\alpha}}{w([s,t])^{(N+1)/p}}.
        \label{eq: Arbitrary Level Davie Pointwise}
    \end{equation}
    The corresponding difference bound holds with the right-hand side multiplied by
    $$
    |\xi-\widetilde\xi|+\|F-\widetilde F\|_{C_b^N}
    +\rho_{[0,T]}(\bold X,\widetilde{\bold X}).
    $$
\end{proposition}

\begin{proof}
    Necessity, the pointwise bound, and its difference version follow by applying Theorem~\ref{thm: Arbitrary Level Rough Integration} to $\mathcal C_F\bold Y$, whose coefficient of index $k-1$ is $F^{[k]}(Y)$.

    Conversely, suppose $D$ has the stated regularity. The path embedding and Proposition~6.2 first give the crude bound
    $$
    \Omega_\infin(D,h)\leq Ch^\eta.
    $$
    Interpolate with the assumed high-order norm, with $\theta=N/(N+1)$. In the strict regime,
    $$
    h^{-N\alpha}\Omega^w_{p/N}(D,h)
    \leq Ch^{\eta/(N+1)}
    \left(h^{-(N+1)\alpha}\Omega^w_{p/(N+1)}(D,h)\right)^{N/(N+1)}.
    $$
    Taking the $L^{q/N}(dh/h)$ quasi-norm proves $\|D\|_N<\infin$. At criticality, the little-space assumption implies $\Omega^w_{p/(N+1)}(D,h)\leq Ch$, and the same calculation gives the same conclusion. Interpolating once more with the crude pointwise bound also gives $\|D\|_m<\infin$ for every integer $1\leq m\leq N$.

    Set $Y^{(i)}=F^{[i]}(Y)$ for $1\leq i\leq N-1$. We verify their controlled expansions without assuming them in advance. For a fixed $i$, Taylor-expand $F^{[i]}(Y_t)$ about $Y_s$ through degree $N-1-i$ and substitute \eqref{eq: Arbitrary Level Davie Defect}. The shuffle identities and the recursive definition of $F^{[k]}$ identify the terms of degree $j<N-i$ as
    $$
    F^{[i+j]}(Y_s)X^{(j)}_{s,t}.
    $$
    The remaining Taylor terms are products of increments of total degree at least $N-i$, or terms containing $D$. The product estimates in the composition proof and Lemma~6.4 control the former. The latter are controlled by $\|D\|_{N-i}$ and boundedness of the other factors. Thus $\bold Y=(Y,Y^{(1)},\ldots,Y^{(N-1)})$ is controlled. At $i=0$, this follows directly from \eqref{eq: Arbitrary Level Davie Defect} and $\|D\|_N<\infin$.

    Define $\widehat Y_t=\xi+\int_0^t\mathcal C_F\bold Y_u\,d\bold X_u$. The integral model is exactly the sum in \eqref{eq: Arbitrary Level Davie Defect}. Its sewing remainder $\widehat D$ belongs to the same strict or critical uniqueness class as $D$. Therefore
    $$
    \delta(Y-\widehat Y)=D-\widehat D
    $$
    is additive and belongs to a vacuous increment class. The vacuity criterion makes $Y-\widehat Y$ constant, and the initial values agree. Hence $Y=\widehat Y$. The coefficient identities give equality of the full controlled collections, completing the proof.
\end{proof}

\begin{remark}[Endpoint and Level Conventions]
\label{rem: Arbitrary Level Endpoint}
    The sewing condition is $(N+1)\alpha>1\lor((N+1)r/p)$ in the strict regime. Under $\alpha>r/p$, this reduces to $(N+1)\alpha>1$. At the endpoint, $p>(N+1)r$ and $q\leq N+1$ are precisely what permit critical sewing, while $\eta>0$ still permits integrability regain.

    The pointwise estimate \eqref{eq: Arbitrary Level Davie Pointwise} has no logarithmic factor at this endpoint. This does not assert that the sewing remainder belongs to $\mathbb B^1_{p/(N+1),q/(N+1)}(w)$: its guaranteed power bound has scale index $\infin$, and finite-scale-index bounds use the loss spaces. In particular, this argument does not provide an endpoint version of the Lyons extension theorem with the original finite scale index.

    Taking $N=2$ recovers the geometric subcase of \S\S 6.3--6.4. The restriction $F\in C_b^{N,1}$ is a sufficient hypothesis ensuring Lipschitz composition in the full controlled topology, consistent with the choice $C_b^{2,1}$ made there; no optimality in the vector-field regularity is claimed.
\end{remark}

\section{Stochastic Processes in the Weighted Setting}

In this section we study stochastic processes in the weighted Besov spaces of the preceding sections and apply the RDE theory of $\S 6$ to obtain location-sensitive solution estimates. The central example is Theorem~\ref{thm: Singular Diffusion}, concerning a diffusion whose coefficient is singular at the initial time. Its weight compensates for the concentration of fluctuations near zero, while the corresponding interval masses retain the location dependence in the solution estimates. Corollary~\ref{cor: Location Sensitive Estimate for Singular Diffusion} makes this explicit through a stability estimate for RDE solution increments.

We begin with three preliminary results. First, Theorem~\ref{thm: Weighted Besov BDG} gives a weighted Besov Burkholder--Davis--Gundy inequality for continuous local martingales, corresponding to the first-level estimate in \cite[(5.10)]{FrizSeeger2022}. Next, Lemma~\ref{lem: Transfer Lemma} shows that higher unweighted integrability yields weighted regularity without loss of smoothness. We use this transfer to establish endpoint weighted Besov regularity for Brownian and fractional Brownian motions and their lifts. These results connect the stochastic examples with the deterministic spaces developed earlier.

Throughout, time weights are deterministic, and martingales are defined on a filtered probability space $(\Omega, \mathcal{F}, (\mathcal{F}_t)_{0 \leq t \leq T}, \mathbb{P})$ satisfying the usual conditions. For a real-valued continuous local martingale $M$, we denote its quadratic variation by $\langle M \rangle$.

Let $M$ be a continuous local martingale (with $M_0 = 0$). We denote by $S_M(s, t)$ the sqaure-root increments $(\langle M \rangle_t - \langle M \rangle_s)^{1/2}$ of quadratic variation.

\begin{theorem}[Weighted Besov BDG]
\label{thm: Weighted Besov BDG}
    Let $1 < p, q, r < \infin$, $w \in A_r$, and $\frac{r-1}{p} < \alpha < 1$. Let $M$ be a real-valued continuous local martingale with $M_0 = 0$. Then
    \begin{equation}
        \| [M]_{B^\alpha_{p, q}(w)} \|_{L^r(\Omega)} \lesssim \left\| \|S_M\|_{\mathbb{B}^\alpha_{p, q}(w)} \right\|_{L^r(\Omega)}. \label{eq:Weighted BDG}
    \end{equation}
    If $M$ is a true martingale on $[0, T]$, or if the right-hand side above is finte, then
    $$
    \| [M]_{B^\alpha_{p, q}(w)} \|_{L^r(\Omega)} \asymp \left\| \|S_M\|_{\mathbb{B}^\alpha_{p, q}(w)} \right\|_{L^r(\Omega)}. \label{eq:Weighted BDG Comparability}
    $$
\end{theorem}
\begin{proof}
    Set $h_n = 2^{-n}T$ and let $D_n(f)$ be as in Proposition~2.16. Since $\alpha > (r-1)/p$, that proposition gives
    $$
    [f]_{B^\alpha_{p,q}(w)} \asymp \left\|\bigl(h_n^{-\alpha}D_n(f)\bigr)_{n\geq1}\right\|_{\ell^q}.
    $$
    Put
    $$
    U_n = D_n(M), \qquad V_n = \left(\int_0^{T-h_n} S_M(s,s+h_n)^p w_{h_n}(s) \hspace{.2em} ds\right)^{1/p}.
    $$
    Applying the preceding characterisation to $M$ and to $\langle M\rangle$ with the metric $d(x,y)=|x-y|^{1/2}$ yields
    $$
    [M]_{B^\alpha_{p,q}(w)} \asymp \|(h_n^{-\alpha}U_n)_{n\geq1}\|_{\ell^q}, \qquad \|S_M\|_{\mathbb{B}^\alpha_{p,q}(w)} \asymp \|(h_n^{-\alpha}V_n)_{n\geq1}\|_{\ell^q}.
    $$

    Fix $N \geq 1$ and use the UMD Banach function space $E = \ell^q(\N;L^p(0,T))$. For $1\leq n\leq N$ and $0\leq s\leq T-h_n$, set
    $$
    c_n(s) = h_n^{-\alpha}w_{h_n}(s)^{1/p}, \qquad Z_t^N(n,s) = c_n(s)\bigl(M_{t\land(s+h_n)}-M_{t\land s}\bigr),
    $$
    extending by zero elsewhere. Localising $M$ by bounded stopping times shows that $Z^N$ is a continuous $E$-valued local martingale. Its terminal value and coordinatewise quadratic variation satisfy
    $$
    Z_T^N(n,s) = c_n(s)(M_{s+h_n}-M_s), \qquad \langle Z^N(n,s)\rangle_T^{1/2} = c_n(s)S_M(s,s+h_n).
    $$
    The martingale-field BDG inequality \cite[Theorem 4.1]{VeraarYaroslavtsev2019} therefore gives
    $$
    \|(h_n^{-\alpha}U_n)_{n=1}^N\|_{L^r(\Omega;\ell^q)} = \|Z_T^N\|_{L^r(\Omega;E)} \leq \left\|\sup_{0\leq t\leq T}\|Z_t^N\|_E\right\|_{L^r(\Omega)} \asymp \|(h_n^{-\alpha}V_n)_{n=1}^N\|_{L^r(\Omega;\ell^q)}.
    $$
    The constants are independent of $N$, since the ambient space $E$ is fixed.

    For the reverse estimate, suppose first that $M$ is a true martingale and $\|Z_T^N\|_{L^r(\Omega;E)}<\infin$. Optional sampling gives $Z_t^N=\mathbb{E}[Z_T^N\mid\mathcal{F}_t]$, so Doob's inequality yields
    $$
    \left\|\sup_{0\leq t\leq T}\|Z_t^N\|_E\right\|_{L^r(\Omega)} \leq \frac{r}{r-1}\|Z_T^N\|_{L^r(\Omega;E)}.
    $$
    Combined with BDG, this proves the reverse sequence estimate; when the terminal norm is infinite, that estimate is automatic. Alternatively, if the right-hand side of \eqref{eq:Weighted BDG} is finite, the preceding BDG estimate gives an integrable $r$th power of the supremum of $Z^N$. Thus $Z^N$ is a true martingale, and the same argument applies.

    Letting $N \rarr \infin$, monotone convergence and the two dyadic characterisations prove the assertions.
\end{proof}

\begin{remark}[The Martingale Requirement]
    The reverse inequality need not hold for an arbitrary continuous local martingale. Even with $w \equiv 1$, $T = 2$, $p = q = r = 4/3$, and $\alpha = 1/3$, there is such an $M$ with
    $$
    \mathbb{E}[M]_{B^\alpha_{p,p}}^p < \infin, \qquad \mathbb{E}\|S_M\|_{\mathbb{B}^\alpha_{p,p}}^p = \infin.
    $$
    These parameters satisfy $\alpha > (r-1)/p = 1/4$.

    Let $B$ be standard Brownian motion and set
    $$
    \tau = \inf\{u \geq 0 : B_u = -1\}, \qquad \theta(t) = \frac{t}{1-t}, \qquad
    M_t =
    \begin{cases}
        B_{\theta(t)\land\tau}, & 0 \leq t < 1, \\
        -1, & 1 \leq t \leq 2.
    \end{cases}
    $$
    Use the usual augmentation of the filtration $\mathcal{G}_t = \mathcal{F}^B_{\theta(t)}$ for $t < 1$ and $\mathcal{G}_t = \mathcal{F}^B_\infin$ thereafter. Since $\tau < \infin$ almost surely, each path becomes constant before time $1$. The quadratic variation is $\langle M\rangle_t = \theta(t)\land\tau$ for $t < 1$ and $\langle M\rangle_t = \tau$ thereafter. Stopping its quadratic variation at $n$ gives square-integrable martingales, so $M$ is a continuous local martingale. It is not a true martingale, since $M_0 = 0$ and $M_1 = -1$.

    The reflection principle and classical BDG inequality give
    $$
    \mathbb{P}(\tau > u) \asymp (1+u)^{-1/2}, \qquad \mathbb{E}|B_{u\land\tau}|^p \lesssim (1+u)^{(p-1)/2}.
    $$
    For $0 < h < 1/4$ and $0 \leq s \leq 1-2h$, conditional BDG therefore yields
    $$
    \mathbb{E}|M_{s+h}-M_s|^p \lesssim \mathbb{P}(\tau>\theta(s))\bigl(\theta(s+h)-\theta(s)\bigr)^{p/2} \lesssim h^{p/2}(1-s)^{1/2-p}.
    $$
    Since $1/2-p = -5/6 > -1$, its integral over these basepoints is bounded by $Ch^{p/2}$. On $[1-2h,1]$, the stopped Brownian moment bound gives
    $$
    \int_{1-2h}^1 \mathbb{E}|M_{s+h}-M_s|^p \hspace{.2em} ds \lesssim h^{1-(p-1)/2} = h^{5/6}.
    $$
    Increments starting after time $1$ vanish. Thus
    $$
    \mathbb{E}\int_0^{2-h}|M_{s+h}-M_s|^p \hspace{.2em} ds \lesssim h^{p/2}.
    $$
    The larger dyadic scales have finite moments as well. Taking $h_n = 2^{-n}T$, the dyadic characterisation gives
    $$
    \mathbb{E}[M]_{B^\alpha_{p,p}}^p \lesssim 1 + \sum_{\{n:h_n<1/4\}} h_n^{p(1/2-\alpha)} = 1 + \sum_{\{n:h_n<1/4\}} h_n^{2/9} < \infin.
    $$

    On the other hand, for $s \in (1/2,1)$,
    $$
    S_M(s,s+1/2)^p = (\tau-\theta(s))_+^{p/2}.
    $$
    Its expectation is infinite because $p/2 = 2/3 > 1/2$. Tonelli's theorem and integration over scale parameters above $1/2$ imply $\mathbb{E}\|S_M\|_{\mathbb{B}^\alpha_{p,p}}^p = \infin$.

    This example lies below the unweighted embedding threshold $\alpha > 1/p$. In the weighted embedding regime $\alpha > r/p$, a finite $L^r(\Omega)$ moment of the path seminorm controls the supremum of $M$ and makes it a true martingale, so the two-sided estimate then applies.
\end{remark}

The following lemma converts higher unweighted integrability into weighted regularity without loss of smoothness. Combined with the reverse H\"older inequality for Muckenhoupt weights, it will allow us to transfer endpoint regularity of Brownian and fractional Brownian paths and their lifts to the weighted setting.

\begin{lemma}[Transfer Lemma]
\label{lem: Transfer Lemma}
    Let $1 < a < \infin$, $w \in L^a([0, T])$, and set $a' = \frac{a}{a-1}$. Then for every $\alpha > 0$, $0 < p < \infin$, $0 < q \leq \infin$, and measurable two-parameter map $A$ we have
    $$
    \| A \|_{\mathbb{B}^\alpha_{p, q}(w)} \lesssim \| w \|^{1/p}_{L^a} \|A\|_{\mathbb{B}^\alpha_{pa', q}}.
    $$
\end{lemma}
\begin{proof}
    Jensen's inequality and Tonelli's theorem give, for $0 < h < T$,
    $$
    \int_0^{T-h} w_h(s)^a \hspace{.2em} ds \leq \frac{1}{h}\int_0^{T-h}\int_s^{s+h} w(t)^a \hspace{.2em} dt \hspace{.2em} ds \leq \int_0^T w(t)^a \hspace{.2em} dt.
    $$
    Thus $\|w_h\|_{L^a(0,T-h)} \leq \|w\|_{L^a(0,T)}$ uniformly in $h$. H\"older's inequality with exponents $a'$ and $a$ yields
    $$
    \left(\int_0^{T-h}\|A_{s,s+h}\|_V^p w_h(s) \hspace{.2em} ds\right)^{1/p} \leq \|w\|_{L^a}^{1/p}\left(\int_0^{T-h}\|A_{s,s+h}\|_V^{pa'} \hspace{.2em} ds\right)^{1/(pa')}.
    $$
    Taking the supremum over $0 < h \leq \tau$ gives
    $$
    \Omega_p^w(A,\tau) \leq \|w\|_{L^a}^{1/p}\Omega_{pa'}(A,\tau).
    $$
    Multiply by $\tau^{-\alpha}$ and take the $L^q((0,T],d\tau/\tau)$ quasi-norm, with the usual supremum convention when $q = \infin$, to obtain the assertion.
\end{proof}

\begin{theorem}[Weighted Brownian and fractional Brownian Regularity]
\label{thm:weighted-stochastic-regularity}
    Let $T > 0$, $d \geq 1$, and let $w \in A_\infin([0, T])$.

    \begin{enumerate}[label = (\alph*)]
        \item Let $W$ be a standard $d$-dimensional Brownian motion and let
        $$
        \bold{W}_{s, t} = (1, W_t - W_s, \mathbb{W}_{s, t})
        $$
        be its level-two lift, with $\mathbb{W}$ interpreted either in the It\^o or in the Stratonovich sense. Then, almost surely, for every $0 < p < \infin$,
        $$
        W \in B^{1/2}_{p, \infin}(w), \qquad
        \mathbb{W} \in \mathbb{B}^{1}_{p/2, \infin}(w).
        $$

        \item Let $B^H$ be a $d$-dimensional fractional Brownian motion with independent standard components and fixed Hurst exponent $H \in (0, 1)$. Then, almost surely, for every $0 < p < \infin$,
        $$
        B^H \in B^H_{p, \infin}(w).
        $$

        If $H > 1/4$, let $\bold{B}^H$ denote its canonical geometric lift: the Gaussian rough-path lift when $H \leq 1/2$, and the Young lift when $H > 1/2$. Write $(\bold{B}^H)^{(k)}$ for the level-$k$ component of its signature extension. Then, almost surely, for every $0 < p < \infin$ and every integer $k \geq 1$,
        $$
        (\bold{B}^H)^{(k)} \in \mathbb{B}^{kH}_{p/k, \infin}(w),
        $$
        where
        $$
        (\bold{B}^H)^{(1)}_{s, t} = B^H_t - B^H_s.
        $$
        At $H = 1/2$, this canonical geometric lift is the Stratonovich lift.

        \item For $B^H$ as in part (b), almost surely,
        $$
        B^H \notin B^\beta_{p, \infin}(w)
        \qquad
        \text{for every } 0 < p < \infin
        \text{ and } H < \beta < 1.
        $$
        More generally, almost surely, for every $0 < p < \infin$ and every $\beta > H$,
        $$
        \sup_{0 < h < T} h^{-\beta}
        \left(
            \int_0^{T-h}
            |B^H_{s+h} - B^H_s|^p w_h(s)
            \hspace{.2em} ds
        \right)^{1/p}
        = \infin.
        $$
    \end{enumerate}

    For each fixed weight $w$ and, in parts (b)--(c), each fixed Hurst exponent $H$, the assertions hold on an event of probability one simultaneously over all the stated integrability exponents, smoothness exponents, and integer levels.
\end{theorem}
\begin{proof}
    We use \cite[Theorems 5.1 and 5.2]{FrizSeeger2022} for the unweighted lifts. The weighted assertions then follow from Lemma~\ref{lem: Transfer Lemma} and the deweighting estimate of Theorem~2.10.

    We first establish the unweighted first-level regularity and sharpness for every $H \in (0,1)$. Let $b^H$ be a scalar fractional Brownian motion, fix $u > 0$, and set
    $$
    h_n = 2^{-n}T, \qquad
    J_n = h_n^{-Hu}\int_0^{T-h_n}|b^H_{s+h_n}-b^H_s|^u \hspace{.2em} ds,
    \qquad m_u = \mathbb{E}|Z|^u,
    $$
    where $Z$ is standard normal. Stationarity and scaling give $\mathbb{E}J_n = (T-h_n)m_u$. The normalised increments
    $$
    G_h(s) = h^{-H}(b^H_{s+h}-b^H_s)
    $$
    are standard Gaussian, with correlation
    $$
    \mathbb{E}[G_h(s)G_h(t)] = \rho_H((t-s)/h), \qquad
    \rho_H(v) = \frac12\left(|v+1|^{2H}+|v-1|^{2H}-2|v|^{2H}\right).
    $$
    The second-difference formula gives $|\rho_H(v)| \lesssim_H (1+|v|)^{2H-2}$. Since the centred even function $x \mapsto |x|^u-m_u$ has Hermite rank at least two, its Hermite expansion yields
    $$
    \left|\operatorname{Cov}(|G_h(s)|^u,|G_h(t)|^u)\right|
    \lesssim_u |\rho_H((t-s)/h)|^2.
    $$
    Consequently,
    $$
    \operatorname{Var}(J_n)
    \lesssim_{H,u,T}
    h_n\int_0^{T/h_n}(1+v)^{4H-4} \hspace{.2em} dv.
    $$
    These bounds are summable in $n$ for every $H<1$. Chebyshev's inequality and Borel--Cantelli therefore give
    $$
    J_n \longrightarrow Tm_u > 0
    \qquad \text{almost surely}.
    $$
    Proposition~2.16 with $w \equiv 1$ now gives $b^H \in B^H_{u,\infin}$ almost surely. Moreover, for every $\beta>H$,
    $$
    h_n^{-\beta}
    \left(\int_0^{T-h_n}|b^H_{s+h_n}-b^H_s|^u
    \hspace{.2em} ds\right)^{1/u}
    = h_n^{H-\beta}J_n^{1/u}
    \longrightarrow \infin.
    $$
    The positive assertion extends componentwise to $B^H$, while any one component gives the negative assertion.

    Since $w \in A_\infin$, reverse H\"older gives $w \in L^a([0,T])$ for some $a>1$. Put $a'=a/(a-1)$. Lemma~\ref{lem: Transfer Lemma} gives
    $$
    \|A\|_{\mathbb{B}^{\eta}_{v,\infin}(w)}
    \leq \|w\|_{L^a}^{1/v}
    \|A\|_{\mathbb{B}^{\eta}_{va',\infin}},
    \qquad \eta>0,\quad v>0.
    $$
    Applying this to $\delta B^H$ proves the first-level assertion in part (b). For part (a), \cite[Theorem 5.2]{FrizSeeger2022}, applied at a sufficiently large integrability exponent, and finite-measure inclusion give
    $$
    W \in B^{1/2}_{pa',\infin}, \qquad
    \mathbb{W} \in \mathbb{B}^{1}_{pa'/2,\infin}
    \qquad \text{almost surely}.
    $$
    The transfer estimate at indices $p$ and $p/2$ proves the assertion for both the It\^o and Stratonovich lifts.

    For the lifted assertion in part (b), suppose $H>1/4$ and set $M=\lfloor 1/H\rfloor$. Fix $N \geq M$ and choose $P$ sufficiently large that
    $$
    P>pa', \qquad H-\frac1P>\frac1{M+1}.
    $$
    When $H\leq1/2$, \cite[Theorem 5.2]{FrizSeeger2022} supplies the unweighted endpoint regularity through level $M$. When $H>1/2$, we have $M=1$, and the required starting estimate is the first-level result above. The Besov extension theorem \cite[Theorem 5.1]{FrizSeeger2022} therefore gives an extension through level $N$ with components in $\mathbb{B}^{kH}_{P/k,\infin}$. By uniqueness of rough-path extension in the H\"older class of exponent $H-1/P>1/(M+1)$, this extension agrees with the canonical signature extension. Thus, for $1\leq k\leq N$,
    $$
    \|(\bold{B}^H)^{(k)}\|_{\mathbb{B}^{kH}_{p/k,\infin}(w)}
    \leq
    \|w\|_{L^a}^{k/p}
    T^{k/(pa')-k/P}
    \|(\bold{B}^H)^{(k)}\|_{\mathbb{B}^{kH}_{P/k,\infin}}
    <\infin
    \qquad \text{almost surely}.
    $$
    Since $N$ was arbitrary, this proves the assertion at every finite level. The path-space $L^p(w)$ terms are finite by continuity and integrability of $w$.

    Finally, choose $r>1$ with $w \in A_r$. Theorem~2.10, applied to the two-parameter map $\delta B^H$, gives
    $$
    \|\delta B^H\|_{\mathbb{B}^{\beta}_{p/r,\infin}}
    \lesssim
    \|\delta B^H\|_{\mathbb{B}^{\beta}_{p,\infin}(w)}.
    $$
    For $\beta>H$, the left-hand side is infinite by the unweighted divergence established above. Hence the weighted first-increment seminorm is infinite as well. This proves part (c), including $\beta\geq1$, since the two-parameter spaces are defined for every positive smoothness exponent.

    For each fixed $w$ and $H$, intersect the probability-one events over rational positive integrability exponents and integer levels. Finite-measure inclusion extends the positive assertions from larger rational exponents and the negative assertions from smaller rational exponents, using
    $$
    \sup_{0<h<T}\int_0^{T-h}w_h(s)\hspace{.2em} ds
    \leq w([0,T]).
    $$
    The dyadic divergence holds simultaneously for every $\beta>H$. This gives the stated simultaneous assertions.
\end{proof}

We now consider the process
$$
X_t = \int^t_0 \sigma(r) \hspace{.2em} dW_r, \qquad |\sigma(r)| \asymp r^{-\gamma}, \qquad 0 < \gamma < \frac{1}{2}.
$$
Since $|\sigma(r)| \asymp r^{-\gamma}$ this process has larger fluctuations on short intervals near the initial time (as compared to, say, intervals away from $t = 0$).

\begin{theorem}[Singular Diffusion]
\label{thm: Singular Diffusion}
    Let $T > 0$, let $W$ be a standard $d$-dimensional Brownian motion, and let $0 < \gamma < 1/2$. Suppose that $\sigma : (0,T] \to \mathbb{R}$ is deterministic and measurable, and that there are constants $0 < c \leq C < \infin$ such that
    $$
    c t^{-\gamma} \leq |\sigma(t)| \leq C t^{-\gamma}
    \qquad \text{for almost every } t \in (0,T].
    $$
    Set
    $$
    X_t = \int_0^t \sigma(u) \hspace{.2em} dW_u,
    \qquad
    w_p(t) = t^{\gamma p},
    \qquad 0 < p < \infin.
    $$
    The process $X$ has a continuous version on $[0,T]$, which we use throughout. The following assertions hold.

    \begin{enumerate}[label = (\alph*)]
        \item For every $0 < p < \infin$,
        $$
        X \in B^{1/2}_{p,\infin}(w_p)
        \qquad \text{almost surely}.
        $$

        \item Let $\bold{X}^{(k)}$ denote the level-$k$ iterated-integral increment of $X$, with all levels interpreted consistently either in the It\^o or in the Stratonovich sense. Then, for every $0 < p < \infin$ and every integer $k \geq 1$,
        $$
        \bold{X}^{(k)} \in
        \mathbb{B}^{k/2}_{p/k,\infin}(w_p)
        \qquad \text{almost surely},
        $$
        where $\bold{X}^{(1)}_{s,t}=X_t-X_s$. The same weight $w_p$ is used at every level. In particular, writing $\mathbb{X}=\bold{X}^{(2)}$,
        $$
        X \in B^{1/2}_{p,\infin}(w_p),
        \qquad
        \mathbb{X} \in \mathbb{B}^{1}_{p/2,\infin}(w_p)
        \qquad \text{almost surely}.
        $$

        \item Fix $0 < p < \infin$ such that
        $$
        \gamma < \frac12-\frac1p.
        $$
        Choose $r$ with $1+\gamma p<r<p/2$. Then $w_p \in A_r$ and $r/p<1/2$. Almost surely, for every $0 \leq s<t \leq T$,
        $$
        |X_t-X_s|
        \lesssim
        [X]_{B^{1/2}_{p,\infin}(w_p)}
        |t-s|^{1/2-1/p}
        \max\{s,|t-s|\}^{-\gamma}.
        $$
        The implicit constant is deterministic and depends only on $p,\gamma,r,T$.

        \item For every $0 < p < \infin$ and $0<h\leq T/2$,
        $$
        \mathbb{E}\int_0^{T-h}|X_{s+h}-X_s|^p
        \hspace{.2em} ds
        \asymp
        \begin{cases}
            h^{p/2}, & \gamma p<1, \\
            h^{p/2}\log(eT/h), & \gamma p=1, \\
            h^{p(1/2-\gamma)+1}, & \gamma p>1.
        \end{cases}
        $$
        The comparison constants may depend on $p,\gamma,c,C,d,T$, but not on $h$. Set
        $$
        s_*(p)=\min\left\{\frac12,\frac12-\gamma+\frac1p\right\}.
        $$
        Almost surely,
        $$
        X \in B^s_{p,\infin}
        \qquad \text{for every } 0<s<s_*(p),
        $$
        whereas
        $$
        X \notin B^s_{p,\infin}
        \qquad \text{for every } s_*(p)<s<1.
        $$
        No assertion is made here about unweighted membership at $s=s_*(p)$.
    \end{enumerate}

    For each fixed coefficient $\sigma$ and exponent $\gamma$, the almost-sure assertions hold on a common event of probability one simultaneously over all the stated integrability exponents, smoothness exponents, and integer levels.
\end{theorem}
\begin{proof}
    Put
    $$
    A(t)=\int_0^t \sigma(u)^2 \hspace{.2em} du,
    \qquad
    V(s,h)=A(s+h)-A(s).
    $$
    Since $2\gamma<1$, we have $A(T)<\infin$, so $X$ is a continuous square-integrable martingale. The coefficient bounds and elementary integration give
    $$
    V(s,h)\asymp h(s+h)^{-2\gamma},
    \qquad
    (w_p)_h(s)\asymp (s+h)^{\gamma p}.
    $$
    In particular,
    $$
    h^{-p/2}V(s,h)^{p/2}(w_p)_h(s)\asymp1,
    $$
    uniformly over the admissible $s,h$.

    For parts (a)--(b), we adapt the argument of \cite[Theorem 5.2]{FrizSeeger2022}. Fix $p>0$ and an integer $N\geq1$, and regard
    $$
    g_t=(1,\bold{X}^{(1)}_{0,t},\ldots,\bold{X}^{(N)}_{0,t})
    $$
    as a continuous path in the truncated tensor group. For the homogeneous metric $d_N$ of \cite[Remark 5.3]{FrizSeeger2022}, Chen's relation gives
    $$
    d_N(g_s,g_t)\asymp_N
    \max_{1\leq k\leq N}|\bold{X}^{(k)}_{s,t}|^{1/k}.
    $$
    This applies to either consistent interpretation of the iterated integrals.

    The deterministic clock $A$ identifies $X$ and its iterated integrals with time-changed Brownian motion and its corresponding iterated integrals. Brownian scaling and homogeneity therefore give
    $$
    \mathbb{E}d_N(g_s,g_{s+h})^m
    =c_{N,m,d}V(s,h)^{m/2},
    \qquad m>0,
    $$
    with finite constants. Finiteness follows from iterated BDG estimates and, in the Stratonovich case, Itô--Stratonovich conversion.

    Set
    $$
    F_h(s)=h^{-p/2}d_N(g_s,g_{s+h})^p(w_p)_h(s),
    \qquad
    Y_h=\int_0^{T-h}F_h(s)\hspace{.2em} ds.
    $$
    The preceding cancellation gives
    $$
    \sup_{h,s}\mathbb{E}F_h(s)\lesssim1,
    \qquad
    \sup_{h,s}\mathbb{E}F_h(s)^2\lesssim1.
    $$
    Moreover, $F_h(s)$ depends only on the Brownian increments over $[s,s+h]$. Thus $F_h(s)$ and $F_h(t)$ are independent when $|s-t|\geq h$, and
    $$
    \operatorname{Var}(Y_h)
    \lesssim
    \int_0^{T-h}\int_0^{T-h}
    \mathbf{1}_{\{|s-t|<h\}}
    \hspace{.2em} ds\hspace{.2em} dt
    \lesssim Th.
    $$
    Taking $h_n=2^{-n}T$, Chebyshev's inequality and Borel--Cantelli give
    $$
    Y_{h_n}-\mathbb{E}Y_{h_n}\longrightarrow0
    \qquad\text{almost surely}.
    $$
    Since the expectations are uniformly bounded,
    $$
    \sup_{n\geq1}h_n^{-1/2}
    \left(
        \int_0^{T-h_n}
        d_N(g_s,g_{s+h_n})^p(w_p)_{h_n}(s)
        \hspace{.2em} ds
    \right)^{1/p}
    <\infin
    \qquad\text{almost surely}.
    $$

    Choose
    $$
    1+\gamma p<r<1+\frac p2.
    $$
    This is possible because $\gamma<1/2$, and gives $w_p\in A_r$ and $(r-1)/p<1/2$. Proposition~2.16 therefore yields
    $$
    [g]_{B^{1/2}_{p,\infin}(w_p)}<\infin
    \qquad\text{almost surely}.
    $$
    The metric comparison implies, for $1\leq k\leq N$,
    $$
    \|\bold{X}^{(k)}\|_{\mathbb{B}^{k/2}_{p/k,\infin}(w_p)}^{1/k}
    \lesssim_N
    [g]_{B^{1/2}_{p,\infin}(w_p)}.
    $$
    Since $N$ was arbitrary, this proves the increment estimates in (a)--(b). Continuity of $X$ and integrability of $w_p$ supply the path-space $L^p(w_p)$ term.

    For part (c), choose $1+\gamma p<r<p/2$. Theorem~3.3 gives, with $h=t-s$,
    $$
    |X_t-X_s|
    \lesssim
    [X]_{B^{1/2}_{p,\infin}(w_p)}
    \frac{h^{1/2}}{w_p([s,t])^{1/p}}.
    $$
    Since
    $$
    w_p([s,s+h])\asymp h(s+h)^{\gamma p}
    \asymp h\max\{s,h\}^{\gamma p},
    $$
    substitution proves the asserted estimate.

    For part (d), $X_{s+h}-X_s$ is a centred Gaussian vector with covariance $V(s,h)I_d$. Hence
    $$
    \mathbb{E}|X_{s+h}-X_s|^p
    \asymp h^{p/2}(s+h)^{-\gamma p}.
    $$
    Tonelli's theorem gives
    $$
    \mathbb{E}\int_0^{T-h}|X_{s+h}-X_s|^p
    \hspace{.2em} ds
    \asymp
    h^{p/2}\int_h^T u^{-\gamma p}\hspace{.2em} du.
    $$
    Evaluating this integral proves the three moment estimates, including the logarithm when $\gamma p=1$.

    Fix $0<s<s_*(p)$. Those estimates imply
    $$
    \mathbb{E}\sum_{n\geq1}h_n^{-sp}
    \int_0^{T-h_n}|X_{u+h_n}-X_u|^p
    \hspace{.2em} du
    <\infin.
    $$
    Indeed, the summands decay geometrically, with an additional factor of order $1+n$ when $\gamma p=1$. Proposition~2.16 with $w\equiv1$ therefore gives $X\in B^s_{p,\infin}$ almost surely.

    For non-membership above $1/2$, fix an interval $[a,b]\subset(0,T)$ and consider the first coordinate $X^1$. Set
    $$
    Q_h=h^{-p/2}\int_a^{b-h}
    |X^1_{u+h}-X^1_u|^p\hspace{.2em} du.
    $$
    On this interval $V(u,h)\asymp h$, so $\mathbb{E}Q_h\gtrsim1$ for small $h$. The same independence argument as above gives $\operatorname{Var}(Q_h)\lesssim h$. Consequently,
    $$
    \liminf_{n\to\infin}Q_{h_n}>0
    \qquad\text{almost surely}.
    $$
    For every $s>1/2$,
    $$
    h_n^{-s}
    \left(
        \int_a^{b-h_n}|X^1_{u+h_n}-X^1_u|^p
        \hspace{.2em} du
    \right)^{1/p}
    =h_n^{1/2-s}Q_{h_n}^{1/p}
    \longrightarrow\infin.
    $$
    Thus $X\notin B^s_{p,\infin}$ for $1/2<s<1$.

    It remains to consider $\gamma p>1$. Put $H_0=1/2-\gamma$, so that $s_*(p)=H_0+1/p$. Since $A(h)\asymp h^{2H_0}$, the Gaussian density bound gives, for every $\varepsilon>0$,
    $$
    \mathbb{P}\left(
        |X^1_{h_n}|\leq h_n^{H_0+\varepsilon}
    \right)
    \lesssim h_n^\varepsilon.
    $$
    Borel--Cantelli implies that $|X^1_{h_n}|>h_n^{H_0+\varepsilon}$ eventually. Taking a countable sequence $\varepsilon\downarrow0$ excludes every H\"older exponent greater than $H_0$ at zero. If $X\in B^s_{p,\infin}$ with $s>H_0+1/p$, Theorem~3.3 with $w\equiv1$ and $r=1$ would give
    $$
    |X_h-X_0|\lesssim [X]_{B^s_{p,\infin}}h^{s-1/p},
    $$
    contradicting this obstruction. This proves the remaining non-membership assertion.

    Finally, the assertions may be made simultaneous in the stated parameters. For weighted membership, first intersect over integer exponents $P$ and integer levels $N$. For $0<p<P$, the comparison
    $$
    (w_p)_h(u)^{1/p}\asymp(u+h)^\gamma
    \asymp(w_P)_h(u)^{1/P}
    $$
    and finite-measure inclusion give
    $$
    \|\bold{X}^{(k)}\|_{\mathbb{B}^{k/2}_{p/k,\infin}(w_p)}^{1/k}
    \lesssim
    T^{1/p-1/P}
    \|\bold{X}^{(k)}\|_{\mathbb{B}^{k/2}_{P/k,\infin}(w_P)}^{1/k}.
    $$
    For the unweighted assertions, intersect over rational integrability and smoothness exponents, then use finite-measure inclusions and continuity of $s_*(p)$. Part (c) follows deterministically from weighted membership. Intersecting also over the two interpretations of the iterated integrals completes the proof.
\end{proof}

\begin{corollary}[Location-Sensitive Estimate]
\label{cor: Location Sensitive Estimate for Singular Diffusion}
    Let $0<\gamma<1/2$, $0<p<\infin$, $0<q\leq\infin$, and set $w_p(t)=t^{\gamma p}$. Choose $r>1+\gamma p$ and suppose that
    $$
    \frac rp<\alpha\leq\frac12,
    $$
    with either
    $$
    \alpha>\frac13,
    \qquad\text{or}\qquad
    \alpha=\frac13 \quad\text{and}\quad 0<q\leq3.
    $$
    Let
    $$
    \bold{X}=(1,X,\mathbb{X}),
    \qquad
    \widetilde{\bold{X}}=(1,\widetilde{X},\widetilde{\mathbb{X}})
    \in \bold{B}^{\alpha,2}_{p,q}(w_p),
    $$
    and let
    $$
    F\in C_b^{2,1}
    \bigl(\mathbb{R}^m;\mathcal{L}(\mathbb{R}^d,\mathbb{R}^m)\bigr).
    $$
    For $\xi\in\mathbb{R}^m$, denote by $Y,\widetilde{Y}$ the unique controlled solutions of
    $$
    dY=F(Y)\hspace{.2em}d\bold{X},
    \qquad
    d\widetilde{Y}=F(\widetilde{Y})\hspace{.2em}d\widetilde{\bold{X}},
    \qquad
    Y_0=\widetilde{Y}_0=\xi.
    $$

    Set
    $$
    E=
    [X-\widetilde{X}]_{B^\alpha_{p,q}(w_p)}
    +
    \|\mathbb{X}-\widetilde{\mathbb{X}}\|_
    {\mathbb{B}^{2\alpha}_{p/2,q/2}(w_p)},
    $$
    and define the controlled remainders by
    $$
    R^Y_{s,t}=\delta Y_{s,t}-F(Y_s)X_{s,t},
    \qquad
    R^{\widetilde{Y}}_{s,t}
    =\delta\widetilde{Y}_{s,t}
    -F(\widetilde{Y}_s)\widetilde{X}_{s,t}.
    $$
    Write
    $$
    K(\bold{X})=
    [X]_{B^\alpha_{p,q}(w_p)}
    +
    \|\mathbb{X}\|_
    {\mathbb{B}^{2\alpha}_{p/2,q/2}(w_p)}^{1/2},
    $$
    and define $K(\widetilde{\bold{X}})$ analogously.

    For every $M>0$, there is a constant $C_M$, depending only on $M,\alpha,p,q,r,\gamma,T$ and the dimensions, such that whenever
    $$
    |\xi|+\|F\|_{C_b^{2,1}}
    +K(\bold{X})+K(\widetilde{\bold{X}})
    \leq M,
    $$
    we have
    $$
    [Y-\widetilde{Y}]_{B^\alpha_{p,q}(w_p)}
    +
    [F(Y)-F(\widetilde{Y})]_{B^\alpha_{p,q}(w_p)}
    +
    \|R^Y-R^{\widetilde{Y}}\|_
    {\mathbb{B}^{2\alpha}_{p/2,q/2}(w_p)}
    \leq C_M E.
    $$
    In particular, for every $0\leq s<t\leq T$,
    $$
    |\delta(Y-\widetilde{Y})_{s,t}|
    \leq
    C_M E\,|t-s|^{\alpha-1/p}
    \max\{s,|t-s|\}^{-\gamma}.
    $$

    These estimates are deterministic and apply pathwise to stochastic drivers satisfying the hypotheses. The lift in Theorem~\ref{thm: Singular Diffusion} satisfies the driver hypothesis almost surely whenever $\alpha<1/2$, for every $q$ allowed above, and also when $\alpha=1/2$ and $q=\infin$. The estimates then apply with $M$ bounding the realised data.
\end{corollary}
\begin{proof}
    Since $r>1+\gamma p$, the power weight $w_p$ belongs to $A_r$. The remaining assumptions are precisely those required by Theorems~6.12 and~6.14, in either the subcritical or critical case. The data bound and Theorem~6.12 give uniform bounds for both controlled solution pairs.

    Write
    $$
    Y'=F(Y), \qquad
    \widetilde{Y}'=F(\widetilde{Y}), \qquad
    \Delta R=R^Y-R^{\widetilde{Y}}.
    $$
    Since the vector field and initial condition agree, Theorem~6.14 gives
    $$
    [Y'-\widetilde{Y}']_{B^\alpha_{p,q}(w_p)}
    +
    \|\Delta R\|_{\mathbb{B}^{2\alpha}_{p/2,q/2}(w_p)}
    \leq C_M E.
    $$
    There are no initial-value contributions, since
    $$
    Y_0=\widetilde{Y}_0=\xi,
    \qquad
    Y'_0=\widetilde{Y}'_0=F(\xi).
    $$

    To recover the solution-path seminorm, apply the difference estimate in Lemma~6.8 with $\beta=2\alpha$. This is admissible because $\alpha>r/p$. The uniform controlled-path bounds and equality of the initial derivatives yield
    $$
    [Y-\widetilde{Y}]_{B^\alpha_{p,q}(w_p)}
    \leq C_M\left(
        [X-\widetilde{X}]_{B^\alpha_{p,q}(w_p)}
        +
        [Y'-\widetilde{Y}']_{B^\alpha_{p,q}(w_p)}
        +
        \|\Delta R\|_{\mathbb{B}^{2\alpha}_{p/2,q/2}(w_p)}
    \right)
    \leq C_M E.
    $$
    Combining the preceding estimates proves the first assertion.

    Set $Z=Y-\widetilde{Y}$. Theorem~3.3 gives, for $0\leq s<t\leq T$,
    $$
    |\delta Z_{s,t}|
    \lesssim
    [Z]_{B^\alpha_{p,q}(w_p)}
    \frac{|t-s|^\alpha}{w_p([s,t])^{1/p}}.
    $$
    Writing $h=t-s$, elementary integration gives
    $$
    w_p([s,s+h])
    =\int_s^{s+h}u^{\gamma p}\hspace{.2em}du
    \asymp h(s+h)^{\gamma p}
    \asymp h\max\{s,h\}^{\gamma p}.
    $$
    Consequently,
    $$
    |\delta(Y-\widetilde{Y})_{s,t}|
    \leq
    C_M E\,|t-s|^{\alpha-1/p}
    \max\{s,|t-s|\}^{-\gamma},
    $$
    as claimed.

    Finally, Theorem~\ref{thm: Singular Diffusion} gives, almost surely, for either consistent choice of iterated integrals,
    $$
    \Omega^{w_p}_{p/k}(\bold{X}^{(k)},\tau)
    \lesssim_\omega \tau^{k/2},
    \qquad k=1,2.
    $$
    If $\alpha<1/2$ and $q<\infin$, then
    $$
    \int_0^T
    \left(
        \tau^{-k\alpha}
        \Omega^{w_p}_{p/k}(\bold{X}^{(k)},\tau)
    \right)^{q/k}
    \frac{d\tau}{\tau}
    \lesssim_\omega
    \int_0^T \tau^{q(1/2-\alpha)}
    \frac{d\tau}{\tau}
    <\infin.
    $$
    For $q=\infin$, the corresponding supremum is finite whenever $\alpha\leq1/2$. Thus the stochastic lift satisfies the asserted driver hypotheses. The deterministic estimates apply to each such realisation and any comparison driver satisfying the hypotheses, with $M$ bounding the realised data.
\end{proof}

\appendix
\section{Choice of Increment Weight}
\label{app: Choice of Increment Weight}

This appendix compares the increment weights mentioned in Remarks~2.4 and~2.12. We explain their normalisation, verify the promised deweighting statements, and identify the interval estimates used in the direct sewing argument. We also record how these comparisons specialise to the stochastic example of \S 7.

For $0<h<T$ and $0\leq s\leq T-h$, consider
$$
\rho_h^L(s)=w(s),\qquad
\rho_h^R(s)=w(s+h),\qquad
\rho_h^M(s)=\frac{w(s)+w(s+h)}2,
$$
and
$$
\rho_h^{\mathrm{av}}(s)=w_h(s)=\frac{w([s,s+h])}{h},
\qquad
\rho_h^{\mathrm{mass}}(s)=w([s,s+h]).
$$
Here $M$ denotes the mean of the two endpoint values. For a measurable two-parameter map $A$ taking values in a finite-dimensional normed space, set
$$
D_p^\rho(A,h)=\left(\int_0^{T-h}|A_{s,s+h}|^p\rho_h(s)\hspace{.2em}ds\right)^{1/p},
\qquad
\Omega_p^\rho(A,\tau)=\sup_{0<h\leq\tau}D_p^\rho(A,h),
$$
and, for $\alpha>0$ and $0<q\leq\infin$,
$$
\|A\|_{\mathbb B^\alpha_{p,q}(\rho)}
=\|\tau^{-\alpha}\Omega_p^\rho(A,\tau)\|_{L^q((0,T],d\tau/\tau)}.
$$
We use the usual supremum convention when $q=\infin$, and interpret the integral at $h=T$ as zero. Path seminorms use $A=\delta f$, or the corresponding metric increment. For three-parameter maps, the weight is $\rho_h(s)$ on the full interval $[s,s+h]$, independently of the splitting point. Restrictions to a subinterval are defined as in \S 2. Endpoint values may be assigned arbitrarily on the exceptional null set of the weight: for each fixed $h$, this changes the integrands only on a null set. Unless stated otherwise, the comparisons below concern $0<p<\infin$; at $p=\infin$, all choices use the unweighted essential-supremum convention.

\subsection{Normalisation and Elementary Properties}

The four choices $\rho^L,\rho^R,\rho^M,\rho^{\mathrm{av}}$ all recover the unweighted increment norms when $w\equiv1$. The unnormalised mass instead satisfies
$$
D_p^{\mathrm{mass}}(A,h)=h^{1/p}D_p^{\mathrm{av}}(A,h).
$$
Thus retaining the same smoothness index with the mass convention changes the scale being measured.

More precisely, if $\alpha>1/p$ and $0<q\leq\infin$, then
\begin{equation}
\|A\|_{\mathbb B^\alpha_{p,q}(\rho^{\mathrm{mass}})}
\asymp
\|A\|_{\mathbb B^{\alpha-1/p}_{p,q}(w)}.
\label{eq: Appendix Mass Shift}
\end{equation}
Indeed, $\Omega_p^{\mathrm{mass}}(A,\tau)\leq\tau^{1/p}\Omega_p^{\mathrm{av}}(A,\tau)$ gives one direction. For the other, put $b=\alpha-1/p>0$ and $H(\tau)=\tau^{-\alpha}\Omega_p^{\mathrm{mass}}(A,\tau)$. Splitting $0<h\leq\tau$ into dyadic shells gives
$$
\tau^{-b}\Omega_p^{\mathrm{av}}(A,\tau)
\leq 2^{1/p}\sum_{j\geq0}2^{-jb}H(2^{-j}\tau).
$$
The geometric coefficients are summable, so taking the scale quasi-norm proves \eqref{eq: Appendix Mass Shift}, using subadditivity when $q<1$. This comparison is stated for the two-parameter seminorms; path spaces retain the smoothness range prescribed in \S 2.

Interval averaging retains the original weight at small scales: Lebesgue differentiation gives
$$
w_h(s)\longrightarrow w(s)
\qquad\text{for almost every }s\text{ as }h\downarrow0.
$$
It is also compatible with time reversal. If $\check w(u)=w(T-u)$, then
$$
(\check w)_h(s)=w_h(T-s-h).
$$
Consequently, reversing both the path and the weight preserves the interval-average increment seminorms. Time reversal exchanges the left and right conventions, and preserves the endpoint-mean and mass conventions.

For each $\rho\in\{\rho^L,\rho^R,\rho^M,\rho^{\mathrm{av}}\}$,
\begin{equation}
\sup_{0<h<T}\int_0^{T-h}\rho_h(s)\hspace{.2em}ds\leq w([0,T]).
\label{eq: Appendix Uniform Mass}
\end{equation}
For the endpoint choices this follows by restriction and translation. For the interval average, Tonelli's theorem gives
$$
\int_0^{T-h}w_h(s)\hspace{.2em}ds
=\frac1h\int_0^T w(u)\,|[u-h,u]\cap[0,T-h]|\hspace{.2em}du
\leq w([0,T]).
$$

\subsection{Deweighting for the Endpoint Choices}

Let $w\in A_r$, $1\leq r<\infin$, and fix $J=[a,b]\subset[0,T]$. Write $J_h=[a,b-h]$. If $r>1$, put $\sigma=w^{-1/(r-1)}$. For each of the four normalised choices,
\begin{equation}
\int_{J_h}\rho_h(s)^{-1/(r-1)}\hspace{.2em}ds\leq\sigma(J).
\label{eq: Appendix Inverse Weight}
\end{equation}
For $\rho^L$ and $\rho^R$, this is immediate from restriction and translation. Convexity gives
$$
\left(\frac{w(s)+w(s+h)}2\right)^{-1/(r-1)}
\leq\frac{\sigma(s)+\sigma(s+h)}2,
\qquad
w_h(s)^{-1/(r-1)}\leq\sigma_h(s),
$$
which proves the remaining cases after integration.

Set $u=p/r$. H\"older's inequality and \eqref{eq: Appendix Inverse Weight} imply, for every measurable $g$ on $J_h$,
$$
\|g\|_{L^u(J_h)}
\leq\sigma(J)^{(r-1)/p}
\left(\int_{J_h}|g(s)|^p\rho_h(s)\hspace{.2em}ds\right)^{1/p}.
$$
This remains valid when $p<1$ or $u<1$, since the H\"older exponents used in its derivation are $r$ and $r/(r-1)$. For $r=1$, each $\rho_h$ is bounded below almost everywhere on $J_h$ by $\operatorname*{ess\,inf}_Jw$, giving the same conclusion with $u=p$ and constant $(\operatorname*{ess\,inf}_Jw)^{-1/p}$.

Thus the conclusions of Theorem~2.10 hold for all four choices, with the same local constant
$$
c_J=
\begin{cases}
\sigma(J)^{(r-1)/p},&r>1,\\
(\operatorname*{ess\,inf}_Jw)^{-1/p},&r=1,
\end{cases}
\qquad
c_J\leq[w]_{A_r}^{1/p}\frac{|J|^{r/p}}{w(J)^{1/p}}.
$$
In particular,
$$
\|A\|_{\mathbb B^\alpha_{p/r,q}(J)}
\leq c_J\|A\|_{\mathbb B^\alpha_{p,q}(\rho;J)}.
$$
Taking the additional supremum over the splitting point proves the three-parameter version. The same argument applies to path moduli, to the general scale functions used in \S 2, and at each tensor level after replacing $p,q$ by $p/k,q/k$. The path-space $L^p(w)$ term is unchanged by the increment convention. At $p=\infin$, the deweighting constant is one.

The transfer estimate of Lemma~7.3 is also shared by these four choices. Indeed, if $w\in L^a(J)$ for $a>1$, then restriction, translation, and Jensen's inequality give
$$
\|\rho_h\|_{L^a(J_h)}\leq\|w\|_{L^a(J)}.
$$
Writing $a'=a/(a-1)$ and applying H\"older yields
$$
\|A\|_{\mathbb B^\alpha_{p,q}(\rho;J)}
\leq\|w\|_{L^a(J)}^{1/p}
\|A\|_{\mathbb B^\alpha_{pa',q}(J)}.
$$
For the unnormalised mass, the fixed-scale deweighting estimate instead acquires the factor $h^{-1/p}$, by $\rho_h^{\mathrm{mass}}=h\rho_h^{\mathrm{av}}$. This agrees with the shift in \eqref{eq: Appendix Mass Shift}.

\subsection{Subdivision and the Sewing Threshold}

The interval-average convention supplies the geometric comparisons in Lemma~2.3. For subintervals $I_i\subset I$, writing $\lambda_i=|I_i|/|I|$, the small-set estimate and monotonicity of weighted mass give
$$
\langle w\rangle_I\leq[w]_{A_r}\lambda_i^{-(r-1)}\langle w\rangle_{I_i},
\qquad
\langle w\rangle_{I_i}\leq\lambda_i^{-1}\langle w\rangle_I.
$$
For a partition $I=\bigcup_{i=1}^N I_i$, additivity of the integral also gives the exact identity
$$
\langle w\rangle_I=\sum_{i=1}^N\lambda_i\langle w\rangle_{I_i}.
$$
For equal-length intervals with basepoints satisfying $|s-t|\leq Lh$, comparison inside their interval hull yields
$$
w_h(s)\leq[w]_{A_r}(1+L)^r w_h(t),
$$
and the reverse bound follows by interchanging $s,t$. All intervals are understood to lie in $[0,T]$.

The improved sewing threshold uses the joint summation estimate of Lemma~2.3(v). Put $P=p\lor r$. For every partition as above and vectors $z_1,\ldots,z_N$,
\begin{equation}
\langle w\rangle_I\left|\sum_{i=1}^N z_i\right|^p
\leq[w]_{A_r}\sum_{i=1}^N
\lambda_i^{1-P}\langle w\rangle_{I_i}|z_i|^p.
\label{eq: Appendix Joint Subdivision}
\end{equation}
To recall why the exponent is $P$, the inclusion $A_r\subseteq A_P$ and weighted H\"older give
$$
\langle w\rangle_I\langle f\rangle_I^P
\leq[w]_{A_r}\langle f^Pw\rangle_I
$$
for nonnegative $f$; when $P=1$, use the $A_1$ lower bound. Apply this with
$$
f=\frac{|z_i|^{p/P}}{\lambda_i}\quad\text{on }I_i.
$$
Since $p/P\leq1$, subadditivity and the triangle inequality give
$$
\left|\sum_i z_i\right|^p
\leq\left(\sum_i|z_i|^{p/P}\right)^P
=\langle f\rangle_I^P,
$$
while $\langle f^Pw\rangle_I$ is the sum on the right of \eqref{eq: Appendix Joint Subdivision}.

For a partition $\pi=\{0=\tau_0<\cdots<\tau_N=1\}$, set $\lambda_i=\tau_i-\tau_{i-1}$ and
$$
(I_\pi B)_{s,s+u}=\sum_{i=1}^N
B_{s+\tau_{i-1}u,s+\tau_i u}.
$$
Apply \eqref{eq: Appendix Joint Subdivision} on $[s,s+u]$, integrate in $s$, and translate the $i$th integral by $\tau_{i-1}u$. Enlarging the resulting integration domain to $[0,T-\lambda_i u]$ and taking the supremum over $0<u\leq h$ gives
$$
\Omega_p^w(I_\pi B,h)^p
\leq[w]_{A_r}\sum_{i=1}^N
\lambda_i^{1-p\nu}\Omega_p^w(B,\lambda_i h)^p,
\qquad
\nu=1\lor\frac rp.
$$
For $N$ equal children, this yields the coefficient $N^\nu$ after taking $p$th roots. Applied to the defects in successive dyadic Riemann sums, it gives the estimate in Lemma~\ref{lem:Dyadic Bounds for Sewing}:
$$
\Omega_p^w(R_{\pi_{\ell+1}}\Xi-R_{\pi_\ell}\Xi,h)
\lesssim2^{\ell\nu}\overline{\Omega}_p^w(\delta\Xi,2^{-\ell}h).
$$
Thus defect regularity of order $\beta$ produces the geometric factor $2^{-\ell(\beta-\nu)}$. Subcritical sewing uses $\beta>\nu$; at equality, the critical scale-summability assumptions of \S 4 are needed. For $p=\infin$, ordinary summation gives $\nu=1$.

Comparing the parent and child weights separately before summing would instead give the potentially larger exponent
$$
\left(1\lor\frac1p\right)+\frac{r-1}{p}.
$$
The improvement to $1\lor r/p$ comes from the joint estimate \eqref{eq: Appendix Joint Subdivision}. In particular, the threshold is $1$ when $p\geq r$. The index $r$ may be replaced by any smaller admissible Muckenhoupt index, as in the openness arguments in the main text.

The three-point norm uses the weight of the full increment interval so that these comparisons apply uniformly to either child interval. Together with the shift estimates and \eqref{eq: Appendix Uniform Mass}, this is the interval structure used by the direct construction.

\subsection{Failure of Uniform Endpoint Comparisons}

The $A_r$ condition does not supply the preceding parent--child and shift bounds for the endpoint choices. Fix $a\in(0,T)$ and $0<b<1$, and take
$$
w(u)=|u-a|^{-b},
$$
assigning a positive finite value at $u=a$. This is an $A_1$ weight: its average over an interval is bounded by a constant depending only on $b$ times its essential infimum, as follows by considering separately intervals whose distance from $a$ exceeds their length and the remaining intervals.

Fix a sufficiently small $h>0$ and set $s=a+\varepsilon$, where $0<\varepsilon<h/4$. Then
$$
w_h(s)\asymp h^{-b},
\qquad
\rho_h^L(s)=\varepsilon^{-b},
\qquad
\rho_h^M(s)\geq\tfrac12\varepsilon^{-b},
$$
whereas on the right child,
$$
\rho_{h/2}^L(s+h/2)\asymp h^{-b},
\qquad
\rho_{h/2}^M(s+h/2)\asymp h^{-b}.
$$
The parent-to-child ratios therefore tend to infinity as $\varepsilon\downarrow0$, even though the length ratio is fixed at two and the weight is fixed. The same example, comparing basepoints $s$ and $s+h/2$ at length $h$, disproves a uniform shift bound. Time reversal gives the corresponding counterexamples for the right-endpoint choice.

These examples concern the uniform weight comparisons required by the present proof. Deweighting remains valid, as established above; a sewing theory based on endpoint weights would require additional estimates or assumptions.

\subsection{Power Weights and the Stochastic Example}

For $w(u)=u^\beta$ with $\beta>-1$, elementary integration gives
$$
w_h(s)=\frac{(s+h)^{\beta+1}-s^{\beta+1}}{(\beta+1)h}
\asymp_\beta(s+h)^\beta.
$$
One may verify this by separating $h\leq s$, where the weight is comparable throughout the interval, from $s<h$, where rescaling by $h$ gives uniform bounds. Hence
$$
\rho_h^R(s)\asymp_\beta w_h(s).
$$
For $\beta\geq0$, also
$$
\tfrac12(s+h)^\beta\leq\rho_h^M(s)\leq(s+h)^\beta,
$$
so the interval-average, right-endpoint, and endpoint-mean conventions give equivalent increment seminorms. By contrast, for $s>0$,
$$
\frac{\rho_h^L(s)}{w_h(s)}
\asymp_\beta\left(\frac{s}{s+h}\right)^\beta.
$$
For $\beta\ne0$, this precludes uniform two-sided pointwise comparability as $s/h\downarrow0$. When $-1<\beta<0$, the endpoint mean is likewise not uniformly comparable to the interval average. The equivalences just established for the other choices are special to these power weights and do not provide the general interval comparisons discussed above.

For the diffusion of Theorem~\ref{thm: Singular Diffusion},
$$
\mathbb E|X_{s+h}-X_s|^p\asymp h^{p/2}(s+h)^{-\gamma p}.
$$
The choice $w_p(u)=u^{\gamma p}$ gives
$$
\mathbb E|X_{s+h}-X_s|^p(w_p)_h(s)\asymp h^{p/2}.
$$
The same cancellation holds at tensor level $k$, with integrability $p/k$ and the same weight $w_p$:
$$
\mathbb E|\bold X^{(k)}_{s,s+h}|^{p/k}(w_p)_h(s)
\asymp h^{p/2}.
$$
These moment comparisons explain the choice of weight; the almost-sure endpoint statements follow from the additional dyadic argument in Theorem~\ref{thm: Singular Diffusion}.

The corresponding interval mass gives
$$
\frac{h^\alpha}{w_p([s,s+h])^{1/p}}
\asymp h^{\alpha-1/p}(s+h)^{-\gamma}
\asymp h^{\alpha-1/p}\max\{s,h\}^{-\gamma}.
$$
This is the location dependence in Corollary~\ref{cor: Location Sensitive Estimate for Singular Diffusion}. Away from zero, when $h\leq s$, it is comparable to $h^{\alpha-1/p}s^{-\gamma}$; when $s\leq h$, it is comparable to $h^{\alpha-\gamma-1/p}$.

The RDE application requires an index $r$ satisfying
$$
1+\gamma p<r<\alpha p,
\qquad
\alpha\leq\frac12,
$$
together with $\alpha>1/3$, or $\alpha=1/3$ and $q\leq3$. Such an $r$ exists exactly when $\gamma+1/p<\alpha$. Theorem~\ref{thm: Singular Diffusion} supplies the required driver regularity for $\alpha<1/2$ and every admissible $q$, and at $\alpha=1/2$ for $q=\infin$. Under these conditions, $\alpha-\gamma-1/p>0$, so the location-sensitive bound still decays on increments meeting the initial time.

\bibliographystyle{alpha}
\bibliography{references}
\end{document}